\documentclass[11pt, a4paper,reqno]{amsart}

\usepackage{graphicx}
\usepackage{latexsym}
\usepackage{xcolor}
\definecolor{RED}{rgb}{1,0,0}
\usepackage{enumerate}
\usepackage{float}
\usepackage{caption, subcaption}
\usepackage{amsmath,amsthm,amsfonts, amssymb, mathrsfs}
\usepackage{bbold}
\usepackage{thmtools}
\usepackage{hyperref}
\usepackage[capitalise,noabbrev]{cleveref}
\usepackage[backend=biber,
style=alphabetic,
sorting=nty]{biblatex}
\renewbibmacro{in:}{}
\DeclareFieldFormat[article, inbook]{title}{#1}
\usepackage{anyfontsize}
\usepackage{t1enc}

\usepackage{enumitem}
\setlist[itemize,1]{label=$\circ$}

\usepackage{tikz}

\newcommand{\R}{\mathbb{R}}
\newcommand{\N}{\mathbb{N}}

\newcommand{\tailtype}{\tau}

\newcommand{\rauzymat}{M}
\newcommand{\finiteperm}{\pi}

\newtheorem{thm}{Theorem}[section]
\newtheorem{ex}[thm]{Example}
\newtheorem{cor}[thm]{Corollary}
\newtheorem{lem}[thm]{Lemma}
\newtheorem{prop}[thm]{Proposition}
\theoremstyle{remark}
\newtheorem{rem}[thm]{Remark}
\theoremstyle{definition}
\newtheorem{defs}[thm]{Definition}

\theoremstyle{plain}
\newcommand{\thistheoremname}{}
\newtheorem{genericthm}[thm]{\thistheoremname}

  \newtheorem*{genericthm*}{\thistheoremname}
\newenvironment{namedthm*}[1]
  {\renewcommand{\thistheoremname}{#1}%
   \begin{genericthm*}}
  {\end{genericthm*}}

  \newtheorem*{theoremA}{Theorem A}
    \newtheorem*{theoremB}{Theorem B}
      \newtheorem*{theoremC}{Theorem C}

\usepackage[textsize=scriptsize, textwidth=2.5cm]{todonotes}
  \author[C.\ Fougeron]{Charles Fougeron}
\address{LAGA - Université Sorbonne Paris Nord, 99 Av. Jean Baptiste Clément,
93430 Villetaneuse, France}
\email{charles.fougeron@math.cnrs.fr}

\author[S.\ Schmidhuber]{Sophie Schmidhuber}
\address{Institut f\"ur Mathematik, Universit\"at Z\"urich, Winterthurerstrasse 190,
CH-8057 Z\"urich, Switzerland}
\email{sophie.schmidhuber@math.uzh.ch}

\title[Rauzy-Veech induction for infinite-type IETs]{Rauzy-Veech induction for infinite-type interval exchange transformations}

\begin{document}
 \begin{abstract} We consider infinite-type IETs arising from elementary examples of finite-area translation surfaces of infinite genus such as the Baker's surface. We call such IETs \textit{tail-reversing} and we show that
 for any tail-reversing permutation the subset of the simplex of lengths $\Delta$ for which the corresponding infinite-type IET is uniquely ergodic contains a dense $G_\delta$ set with respect to the $\ell^1$-topology.
 To this end, we generalize Rauzy-Veech induction to a large class of infinite-type IETs, where we prove a minimality criterion as a generalization of Keane's criterion in the finite setting. We then restrict ourselves to tail-reversing IETs and obtain our genericity result through a combinatorial analysis of their infinite-type Rauzy diagrams. Moreover, we derive an explicit condition for a tail-reversing IET to be uniquely ergodic by studying the diameter of its induction matrices.
 \end{abstract}
 \maketitle

\section{Introduction}

An \textit{interval exchange transformation} (IET) is a piecewise continuous map from an interval $[0,l]$ to itself whose restriction to each interval of continuity is a translation.
Equivalently, an IET partitions $[0,l]$ into countably many subintervals and \textit{permutes} them.
Depending on whether the number of intervals is finite or infinite, we say the IET is of \textit{finite-type} or \textit{infinite-type}. This class of maps is extremely rich: every invertible aperiodic measure-preserving transformation is isomorphic to such a map \cite{ArnouxOrnsteinWeiss85}. Nevertheless, one may study infinite-type IETs by restricting to classes with simplified combinatorics.

\begin{figure}[H]
    \centering
    \includegraphics[width=0.6\linewidth]{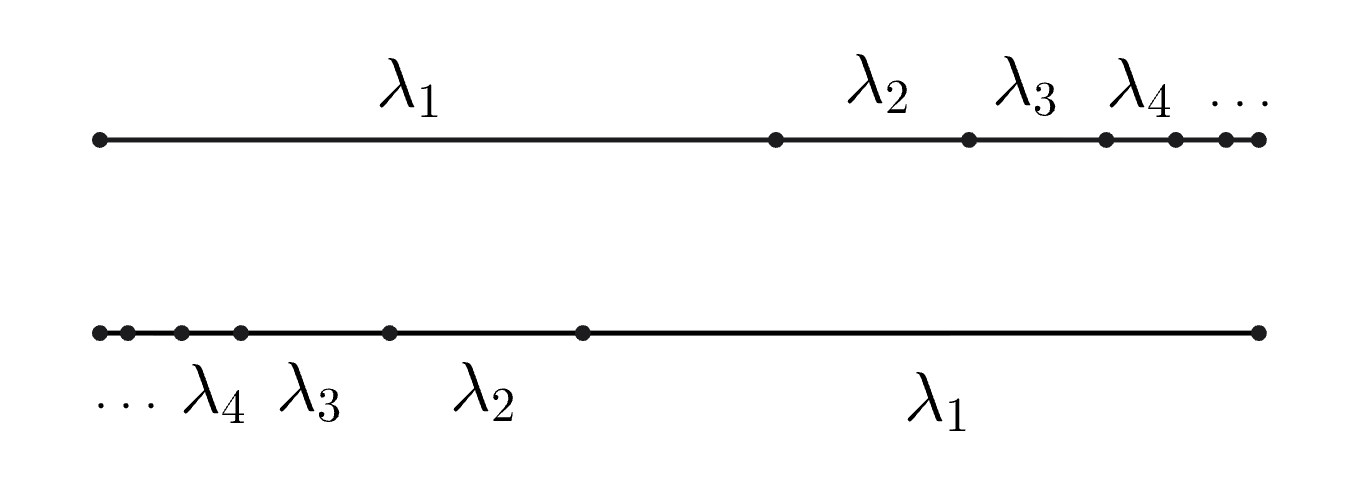}
    \caption{The inverse permutation.}
    \label{fig:intro}
\end{figure}

In this paper we focus on infinite-type IETs arising from elementary examples of finite-area translation surfaces of infinite genus such as the Baker's surface \cite{DeHuVa}.
A basic family of examples is provided in Figure \ref{fig:intro}: given a summable sequence of positive reals $(\lambda_i)_{i \in \mathbb{N}}$, consider the IET which reverses the order of intervals of lengths $\lambda_i$. We call this order the \textit{inverse permutation}; it is a special case of the class of \textit{tail-reversing} permutations studied later in this paper.

Even within this class, infinite-type IETs can exhibit dynamical behaviour far more complex than their finite-type counterparts.
Recall that one of the earliest results in the finite-type theory, \textit{Keane's criterion} \cite{Keane1975}, states that a finite-type IET without \textit{connections} --- that is, without an orbit connecting two singularities (see Definition~\ref{def:connection}) --- is necessarily \textit{minimal}, meaning every infinite orbit is dense. This implication fails for infinite-type IETs already in a simple setting: in \cite{DeHuVa}, the authors construct an infinite-type IET with inverse permutation and without connections, yet admitting an orbit whose closure is an invariant Cantor set and therefore not dense.
Such Cantor-like dynamics for a piecewise translation are a genuinely infinite-type phenomenon, requiring infinitely many intervals of continuity to occur.

The main goal of this paper is to show that, despite this added complexity, unique ergodicity is still topologically generic in the space of length parameters for the class of IETs we study.

\subsection*{Relation to previous works} A natural approach in the study of infinite-type IETs is to approximate them by finite-type projections in an inverse-limit type construction.
This is the strategy employed by M\'alaga Sabogal--Troubetzkoy \cite{MalagaTroubetzkoy} to prove generic unique ergodicity for certain infinite genus translation surfaces - the geometric counterpart to infinite IETs - where, as they note, the approximations must converge ``sufficiently quickly'' for the argument to go through.
A similar path appears in the work of Rafi--Randecker \cite{RafiRandecker} on an attempted generalization of Masur's criterion to finite-area translation surfaces of infinite genus, where the idea takes the geometric form of truncating the surface by discarding small-area pieces near which infinite genus
concentrates.
The core difficulty in both cases is that the Teichm\"uller flow tends to rapidly zoom into precisely these discarded regions, forcing one to retain precise information about the pieces that were meant to be forgotten.
Our approach confronts this tension directly, working at the level of the combinatorics of the IET rather than its geometry.

\subsection{Content of the paper} We extend the \textit{Rauzy--Veech induction} --- the classical renormalization scheme for finite-type IETs developed around 1980 by Rauzy \cite{Rauzy} and Veech \cite{Veech}, which has been central to the finite-type theory --- to the infinite-type setting.
We define the induction for a broad class of infinite-type IETs, which we call \textit{renormalizable}, and use it to establish a generalization of Keane's criterion.
Concretely, at the $n$-th induction step we associate to an infinite-type IET $T$ a new infinite-type IET $T^{(n)}$, obtained by restricting $T$ to a smaller subinterval $I^{(n)} \subset [0,l]$.
To any sequence of induction steps between times $m$ and $n$ we associate an infinite matrix $M^{(m,n)}$, indexed by the intervals of the IET, encoding how interval lengths transform.
We call $T$ \textit{combinatorially minimal} if, starting from any time $m$, all entries of $M^{(m,n)}$ eventually become positive.
Our first result is:

\begin{theoremA}
    An infinitely renormalizable IET $T$ is combinatorially minimal if and only if it has no connections and is minimal.
\end{theoremA}

We then consider a class of infinite-type IETs which we call \textit{tail-reversing}.
These are infinite-type IETs whose intervals can be assembled into finitely many \textit{groups}, such that the closure of each group consisting of infinitely many intervals (called  \textit{tail intervals}) is connected, and the order of intervals on the top and bottom is reversed (see Definition~\ref{def:tail-reversing}). For instance, any IET with the inverse permutation of Figure~\ref{fig:intro} is tail-reversing.

The key feature of tail-reversing IETs is that the order of the grouped intervals, as well as the order in which intervals accumulate in the tails, is encoded by a single permutation on finitely many elements.
We track how this grouped permutation evolves under Rauzy--Veech induction via an additional technique which we call \textit{revealing}, and record the result in \textit{Rauzy diagrams}. As in the finite-type case, any tail-reversing IET defines (through its length vector) a \textit{path} in the corresponding Rauzy diagram (see Definition \ref{def:rotationnumber}), and similarly a given path may correspond to one or more IETs. We perform a combinatorial study of these paths to establish the following genericity result.

Let $\Pi$ be the grouped permutation of a tail-reversing IET that is \textit{proper} and \textit{irreducible}: irreducible is defined analogously to the standard notion for finite-type permutations, while proper means that the left endpoint $0$ is not simultaneously an accumulation point of top and bottom singularities, and that the same holds for the right endpoint $l$ (see Definition~\ref{def:proper}).
Let $\Delta$ denote the space of strictly positive sequences in $\mathbb{R}_{>0}^{\mathbb{N}}$ summing to one, equipped with the $\ell^1$-topology.

\begin{theoremB}
    For any proper and irreducible permutation $\Pi$, the set of lengths $\lambda \in \Delta$ for which the IET $T_{(\Pi,\lambda)}$ is uniquely ergodic contains a dense $G_\delta$ set with respect to the $\ell^1-$topology.
\end{theoremB}

This result should be compared with the finite-type case.
Topological genericity of unique ergodicity for finite-type IETs was established by Keane and Rauzy \cite{KeaneRauzy1980} in 1980 and is in fact used in the proof of Theorem B. Two years later, Masur \cite{Masur1982IET} and Veech \cite{Veech} independently proved the stronger statement that unique ergodicity holds for \textit{almost every} finite-type IET, in the sense of the Lebesgue measure on the finite-dimensional simplex of length vectors.
Neither of these results has so far been extended to infinite-type IETs; Theorem~B provides the first step in this direction.

We conclude the present study of infinite-type IETs by proving a sufficient condition for unique ergodicity of a given tail-reversing IET. We define the \textit{Hilbert metric} on the infinite-dimensional simplex $\Delta$ and for a finite subset $A \subset \N$ we define the \textit{diameter} $D_{A}(M)$ of a given Rauzy-Veech matrix $M$ restricted to $A$ as the diameter of the matrix obtained by restricting $M$ to the rows indexed by $A$. In addition we say that $M$ is $A$-supported if it acts trivially on rows outside of $A$. We then prove the following:

\begin{theoremC} Let $\boldsymbol{\gamma}$ be an infinite path in a Rauzy diagram $\mathcal{R}$. Assume there exist
		\begin{itemize}
			\item an increasing sequence of finite subsets $A_1 \subset A_2 \subset \dots \subset \N$ with $\bigcup_n A_n = \N$,
			\item a factorisation $\boldsymbol{\gamma} = \zeta_1 \cdot \zeta_2 \cdots$ such that, for every $n \in \N$, the matrix $M_n = M^{\zeta_n}$ is $A_n$-supported and $\zeta_n$ reveals at least $n$ intervals in each tail,
			\item a constant $C$ with $\sup_n D_{A_n}(M_n) < C$.
		\end{itemize}
		Then there is a unique IET associated to $\boldsymbol{\gamma}$, and it is uniquely ergodic.
\end{theoremC}

Theorem C provides an alternative proof for Theorem B without using the previously mentioned finite-type result by Keane and Rauzy \cite{KeaneRauzy1980}. In addition however, Theorem C enables us to construct a large set of paths which determine uniquely ergodic IETs and might prove useful in the future for establishing measure-theoretic genericity results for tail-reversing IETs.

\section{Preliminaries}  We define interval exchange transformations as well as some important notions regarding their parametrization and their possible dynamics. We then highlight the differences between the finite-type and the infinite-type setting: we introduce two classical theorems for finite-type IETs, namely \textit{Keane's criterion} and \textit{Maier's theorem}, and show that both fail in the infinite setting using a counterexample constructed by Delecroix, Hubert and Valdez in \cite{DeHuVa}.

\subsection{Interval exchange transformations} We introduce IETs as well as their combinatorial and length data.

\subsubsection{Definition of IETs} IETs are piecewise continuous translations on the interval with possibly infinitely many intervals of continuity. More formally,

\begin{defs}\label{def:gGIETs} Let $\mathcal{A}$ be a countable set. An \textit{interval exchange transformation (IET)} on $\mathcal{A}$ is a map $T$ from the interval $[0, l]$ to itself such that
\begin{enumerate}[label = \roman*)]
\item there are two partitions $\bigsqcup_{\alpha \in \mathcal{A}}I_{\alpha}^t$, $\bigsqcup_{\alpha \in \mathcal{A}} I_{\alpha}^b$ of $[0, l]$ into disjoint open subintervals $\{ I_{\alpha}^t\}_{\alpha \in \mathcal{A}}$, $\{ I_{\alpha}^b\}_{\alpha \in \mathcal{A}}$ called the \textit{top} and \textit{bottom} intervals.
\item for each $\alpha \in \mathcal{A}$, the map $T$ restricted to $I_{\alpha}^t$ is a translation of the form $x \to x + t_{\alpha}, \hspace{1mm}t_{\alpha} \in \mathbb{R}$ with image $T(I_{\alpha}^t) = I_{\alpha}^b$.
\end{enumerate}

If $\mathcal{A}$ is finite, we say that $T$ is of \textit{finite-type}, otherwise we say that $T$ is of \textit{infinite-type}.
\end{defs}

\subsubsection{Singularities}\label{sec:singularities} The points $\mathcal{S}^t = [0,l] \backslash \bigsqcup_{\alpha \in \mathcal{A}}I_{\alpha}^t$ are called the \textit{top singularities} of $T$, the points  $\mathcal{S}^b = [0,l] \backslash \bigsqcup_{\alpha \in \mathcal{A}}I_{\alpha}^b$ are called the \textit{{bottom singularities}}. The set $Acc(\mathcal{S}^t)$ (resp. $Acc(\mathcal{S}^b$)) denotes the set of accumulation points of $\mathcal{S}^t$ (resp. $\mathcal{S}^b$). For the remainder of this text, we assume that this set of accumulation points is \textbf{finite}.

We say that an element in $Acc(\mathcal{S}^t)$ ($Acc(\mathcal{S}^b$)) is \textit{right-sided} if it is accumulated by singularities from the right, and similarly it is \textit{left-sided} if it is accumulated by singularities from the left. An element in $Acc(\mathcal{S}^t)$ ($Acc(\mathcal{S}^b$)) may be both right-sided and left-sided.

\subsubsection{Enumeration}
\label{sec:enumerationmap}
Since we want to define vectors and later on matrices whose entries are indexed by $\mathcal{A}$, we choose an injective map $\pi: \mathcal{A} \to \N$ (bijective if  $|\mathcal{A}| = \infty$) and call it the \textit{enumeration map} of $\mathcal{A}$. To simplify notation, we will often denote $\pi(\alpha)$ by $\alpha$. Note that in general it is not possible to choose $\pi$ in such a way that the order of the natural numbers corresponds to the order of the intervals in $[0,l]$, meaning in such a way that the sequence
\begin{align*}
	I_{\pi^{-1}(1)}^t , I_{\pi^{-1}(2)}^t , I_{\pi^{-1}(3)}^t, \dots
\end{align*}
describes the order of the intervals in $[0,l]$. This is for example not possible when there is more than one accumulation point of singularities.

\subsubsection{Combinatorial and length data}
\label{sec:length_data}
As explained previously, we cannot simply encode the order of intervals by arranging them in an infinite vector, hence we use total orders. For $T$ an infinite-type IET on $\mathcal{A}$, consider two total orders $(\mathcal{A},\leq_t)$ and $(\mathcal{A},\leq_b)$ defined as follows:
for two labels $\alpha, \beta \in \mathcal{A}$ we have
\begin{align*}
    \alpha \leq_t \beta \iff \inf I_{\alpha}^t < \inf I_{\beta}^t, \\
    \alpha \leq_b \beta \iff \inf I_{\alpha}^b < \inf I_{\beta}^b.
\end{align*}
This describes the order of the intervals in the top and bottom partition of $[0,l]$ from left to right.
In Section \ref{sec:permutationtail-reversing} we associate to an infinite-type IET satisfying certain assumptions a permutation on finitely many letters by \textit{grouping} together intervals.

We also define the \textit{length vector} associated to an IET:

\begin{defs}\label{def:lengthvector} For $T:[0,l] \to [0,l] $ an IET, its \textit{length vector} is defined to be the countably infinite vector $\lambda \in \R_{>0}^{\mathcal{A}}$ such that $\sum_{\alpha \in \mathcal A} \lambda_{\alpha} = l$ and $\lambda_{\alpha} = |I_{\alpha}^t| = |I_{\alpha}^b|$.
\end{defs}

\subsection{Orbits of IETs} We now consider orbits of interval exchange transformations, define recurrence and classify the possible dynamical behaviour of recurrent orbits.

\subsubsection{Definition of an orbit} We define the forward orbit  $\mathcal{O^+}(x)$ of a point $x \in [0,l] \backslash \mathcal{S}^t$ as
\[
    \mathcal{O^+}(x) := \{y \in [0,l] \hspace{1mm} | \hspace{1mm} \exists N \in \mathbb{N} \; \text{s.t} \;T^{n}(x) \; \notin \mathcal{S}^t \; \text{for all $n < N$ and} \; T^{N}(x) = y \},
\]
and the backward orbit of a point $\mathcal{O^-}(x) $ of a point $x \in [0,l] \backslash \mathcal{S}^b$ as
\[
    \mathcal{O^-}(x) := \{y \in [0,l] \hspace{1mm} | \hspace{1mm} \exists M \in \mathbb{N} \; \text{s.t} \;T^{-m}(x) \; \notin \mathcal{S}^b \; \text{for all $m < M$ and} \; T^{-M}(x) = y \}.
\]
The \textit{full orbit} $\mathcal{O}(x)$ of a point $x \in [0,l] \backslash \mathcal{S}^t \cup \mathcal{S}^b$ is then defined by
$$\mathcal{O}(x) := \mathcal{O}^+(x) \cup \mathcal{O}^-(x).$$

$\mathcal{O^+}(x), \mathcal{O^-}(x)$ and $\mathcal{O}(x)$ are called \textit{finite}, resp. \textit{infinite}, if their cardinality is finite, resp. infinite. If $\mathcal{O}^+(x)$ is finite, we call $T^{n}(x)$ the \textit{forward endpoint} of $\mathcal{O}(x)$, where $n$ is the largest integer for which $T^{n}(x) \in \mathcal{O}(x)$, similarly, if $\mathcal{O}^-(x)$ is finite, we call $T^{-m}(x)$ the \textit{backward endpoint} of $\mathcal{O}(x)$, where $m$ is the largest integer for which $T^{-m}(x) \in \mathcal{O}(x)$. Note that an orbit $\mathcal{O}(x)$ is contained in the domain of definition of $T$ except for its forward (backward) endpoint, if it exists, which may be equal to a top (bottom) singularity.

\begin{defs}\label{def:connection} A \textit{regular orbit} is an orbit which is infinite in both directions.
A \textit{connection} of $T$ is an orbit which is finite both in the forward and backward direction. A \textit{connection between accumulation points} is a connection between an element in $Acc(\mathcal{S}^t)$ and $Acc(\mathcal{S}^b)$.

\end{defs}

\subsubsection{Recurrence} A \textit{recurrent} orbit is an orbit which comes back infinitely often to any point which lies on it both in the future and in the past.

\begin{defs}  If $\mathcal{O}^+(x)$ is infinite then the $\omega$-limit set $\omega(x)$ is the set of accumulation points of $\mathcal{O^+}(x)$ in $[0,l]$. If $\mathcal{O}^-(x)$ is infinite then the $\alpha$-limit set $\alpha(x)$ is the set of accumulation points of $\mathcal{O^-}(x)$ in $[0,l]$.
\end{defs}

A point $x \in I^t$ is said to be \textit{$\omega$-recurrent} if $\mathcal{O}^+(x)$ is infinite and $x \in \omega(x)$, it is said to be \textit{$\alpha$-recurrent} if $\mathcal{O}^-(x)$ is infinite and $x \in \alpha(x)$. An $\alpha-$ or $\omega-$ recurrent orbit which is not closed is called \textit{non-trivially recurrent}.
We define a \textit{quasiminimal} of an infinite-type IET then as follows:

\begin{defs}
The topological closure (in $[0,l]$) of a non-trivially $\alpha-$ or $\omega-$ recurrent orbit is called a \textit{quasiminimal} of $T$.
\end{defs}

\subsubsection{Minimality and unique ergodicity}
We now define the two main dynamical behaviours of an IET that we wish to study, \textit{minimality} and the stronger form of \textit{unique ergodicity}:

\begin{defs} An IET $T:[0,l] \to [0,l] $ is called \textit{minimal} if every infinite orbit is dense in $[0,l]$.
\end{defs}

In particular this means that $T$ contains a unique quasiminimal equal to $[0,l]$. Note that here "infinite orbits" also includes non-regular orbits, i.e. infinite orbits which in forward or backward time hit a singularity.

\begin{defs} An IET $T:[0,l] \to [0,l] $ is called \textit{uniquely ergodic} if there exists a unique $T$-invariant Borel probability measure on $[0,l]$.
\end{defs}

This unique measure is then necessarily equal to the Lebesgue-measure on $[0,l]$.

\subsection{Differences in dynamics of finite and infinite-type IETs}

Infinite-type IETs exhibit significantly more complex dynamical behavior than their finite counterparts.
This complexity is illustrated by a construction introduced by Delecroix, Hubert and Valdez in \cite{DeHuVa}.
While the authors primarily use this example to demonstrate the failure of Keane's criterion in the infinite setting, it also serves as a counterexample to the infinite-type extension of Maier’s Theorem.
In this section, we review these two classical results for finite-type IETs and discuss how their infinite-type example fails to satisfy either property.

\subsubsection{Keane's criterion and Maier's Theorem} In the 1930s, during an active period in the study of topological properties of flows on surfaces, A. Maier \cite{Maierfinite} proved the following result for flows on surfaces of finite genus, which in particular holds for finite-type interval exchange transformations (see also \cite{dynamicaldecomposition}, Section 3.5 for more details):

\begin{thm}(Maier's Theorem)
    Any nontrivially $\alpha-$ or $\omega-$recurrent orbit contained in a quasiminimal is dense therein.
\end{thm}

This implies in particular that the intersection of two distinct quasiminimals cannot contain a non-trivially recurrent orbit, in particular, a quasiminimal cannot be strictly contained inside another quasiminimal. Maier's Theorem leads to an upper bound of the number of distinct quasiminimals of a finite-type IET in terms of the number of intervals of the IET (see \cite{Maierfinite}, or \cite{dynamicaldecomposition} for the specific case of IETs).

Another result, which was one of the earliest results in the study of finite-type IETs, was established in 1975 by M. Keane and became known as \textit{Keane's criterion}:

\begin{thm}(Keane's criterion)
    Any finite-type IET without connections is minimal.
\end{thm}

It is a well-known fact that if a finite-type IET has connections, we may restrict it to finitely many subintervals on which the IET is either minimal or all orbits are periodic of the same period.

\subsubsection{A counterexample for the infinite setting}\label{sec:theconstruction} For infinite-type IETs, both Maier's Theorem and Keane's criterion fail to be true. This can be seen using a counterexample constructed by the authors of \cite{DeHuVa} in Section 4.3. For details of the construction we refer the reader to \cite{DeHuVa} and only give an informal summary here: starting from a minimal infinite-type IET with saddle connections, called the \textit{odometer}\footnote{The \textit{odometer} is the infinite-type IET on $[0,1]$ from Figure \ref{fig:intro} where $\lambda_i = \frac{1}{2^i}$.}
, the authors \textit{add} infinitely many periodic intervals inbetween the existing intervals in order to obtain an IET with an invariant Cantor set and with saddle connections. They then slightly perturb the lengths in order to obtain a new IET with an invariant Cantor set and without saddle connections, as given in Figure \ref{fig:3}.

\begin{figure}[H]
    \centering
\includegraphics[width=0.8\linewidth]{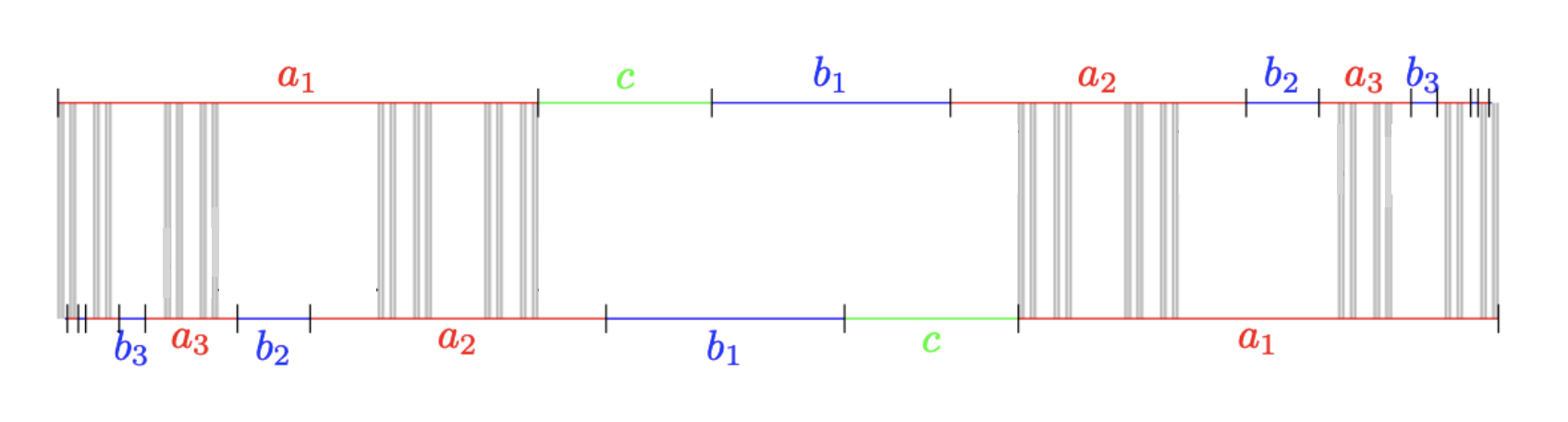}
    \caption{An infinite-type IET without saddle connections and an invariant Cantor set. \textit{From "Infinite translation surfaces in the wild" by Delecroix, Hubert and Valdez \cite{DeHuVa}, p. 215, reprinted with permission.} }
    \label{fig:3}
\end{figure}

\subsubsection{Failure of Keane's criterion and Maier's Theorem in the infinite setting} The previous construction yields an example of an infinite-type IET $T$ which has no saddle connections but which is also not minimal since there exists an infinite orbit which is nowhere dense, hence providing a counterexample for Keane's criterion in the infinite setting. It furthermore contradicts Maier's Theorem in the infinite setting: indeed, using our results from Section 4 (see Remark \ref{rem:orbitstozero}) one can show that all orbits of $T$ must accumulate at $0$. By Poincaré recurrence and the fact that there are no saddle connections there must exist a non-trivially recurrent orbit in the complement of the Cantor set $\Omega_1$, and since it accumulates at zero its closure strictly contains $\Omega_1$, hence $T$ has at least two nested quasiminimals $\Omega_1 \subsetneq \Omega_2$, contradicting Maier's result in the infinite setting. In particular, the proof of an upper bound on the number of quasiminimals for the finite setting fails in the infinite setting, in fact, in \cite{gutierrezinfinite} the authors even construct infinite-type IETs with an infinite \textit{chain} of quasiminimals $(\overline{\gamma_i})_{i \in \N}$ strictly contained inside one another.

These counterexamples already hint that the dynamics of infinite-type IETs is much more complicated and difficult to study than the dynamics of finite-type IETs. Nevertheless, in the remainder of this text, we generalize one of the most important tools in the study of finite-type IETs, the \textit{Rauzy-Veech induction}, to the case of infinite-type IETs and show that it provides a useful point of view also for the study of infinite-type IETs.

\section{Rauzy-Veech induction for infinite-type IETs}

We introduce our main tool for the remainder of this article: Rauzy-Veech induction for infinite-type IETs, a generalization of the renormalization scheme for finite-type IETs introduced by Rauzy \cite{Rauzy} and Veech \cite{Veech}. The key idea behind the induction is to study the long-term dynamics of a given IET by associating to it a sequence of \textit{accelerated} IETs obtained by considering its first return map to smaller and smaller subintervals, where the subintervals are defined according to \textit{winning} and \textit{losing} intervals. However, in the infinite-type case we also have to deal with the case when infinitely many intervals \textit{play} at once.

After having defined the induction, we discuss the notion of \textit{renormalizability} and, for a given sequence of induction steps, define the corresponding Rauzy-Veech matrices which describe the change in the length vector. Using these matrices, we then introduce the notion of \textit{combinatorially minimal} for infinite-type IETs and prove Theorem A, which states that minimality is equivalent to combinatorial minimality.

\subsection{The induction process}\label{sec:standardRVinduction} We define the induction process and discuss the notion of \textit{renormalizability} as well as \textit{towers} for an IET.

\subsubsection{The elementary step} Let $T:[0,l] \to [0,l]$ be an IET. Given $\alpha \in \mathcal{A}$, we denote the left and right endpoints of the top and bottom interval labelled by $\alpha$ by
\begin{align*}
    I_{\alpha}^t = (l_{\alpha}^t, r_{\alpha}^t) \\
    I_{\alpha}^b = (l_{\alpha}^b, r_{\alpha}^b)
\end{align*}
We further denote the rightmost (resp. leftmost) accumulation points of the top and bottom singularities by
\begin{align*}
    s_{\max}^t := \max (Acc(\mathcal{S}^t)), \hspace{2mm}s_{\max}^b := \max (Acc(\mathcal{S}^b))\\
    s_{\min}^t := \min (Acc(\mathcal{S}^t)), \hspace{2mm}s_{\min}^b := \min (Acc(\mathcal{S}^b)).
\end{align*}
For the induction to be well-defined, we further need to assume that the rightmost accumulation points $s_{\max}^t$ and $s_{\max}^b$ are not both equal to $l$.
In the language that we introduce below, this means that we want to exclude the case when infinitely many intervals play against each other both at the top and at the bottom.

We also need the definition of left-sided and right-sided elements of $Acc(\mathcal{S}^t)$ and $Acc(\mathcal{S}^b)$ from Section \ref{sec:singularities}, as well as Definition \ref{def:connection} of a connection.

\begin{defs}\label{def:RVinduction}
	Let $T$ be an IET without $s_{\max}^t = s_{\max}^b = l$. The IET $T^{(1)}$ obtained by one step of Rauzy-Veech induction from $T$ is the IET obtained as the first return map of $T$ to the interval $I^{(1)}:=[0,l^{(1)}]$ where $l^{(1)}$ is defined below.
	Let $\alpha_t, \alpha_b \in \mathcal{A}$ be, when they exist, the label of the rightmost top and bottom intervals, i.e. $r_{\alpha_t}^t = r_{\alpha_b}^b = l$.
	We distinguish several cases.
\begin{itemize}
\item \textit{Two intervals play}:
	Assume $s_{\max}^t \neq l$, $s_{\max}^b \neq l$, i.e. there are two rightmost top and bottom intervals, which in addition do not contain a right-sided accumulation point, i.e. if $s_{\max}^t \in I_{\alpha_b}^b$ or $s_{\max}^b \in I_{\alpha_t}^t$, then $s_{\max}^t$, resp. $s_{\max}^b$, are left-sided. Then $l^{(1)} := \max(l_{\alpha_t}^t, l_{\alpha_b}^b).$

\begin{figure}[H]
\centering
\includegraphics[width=0.45\textwidth]{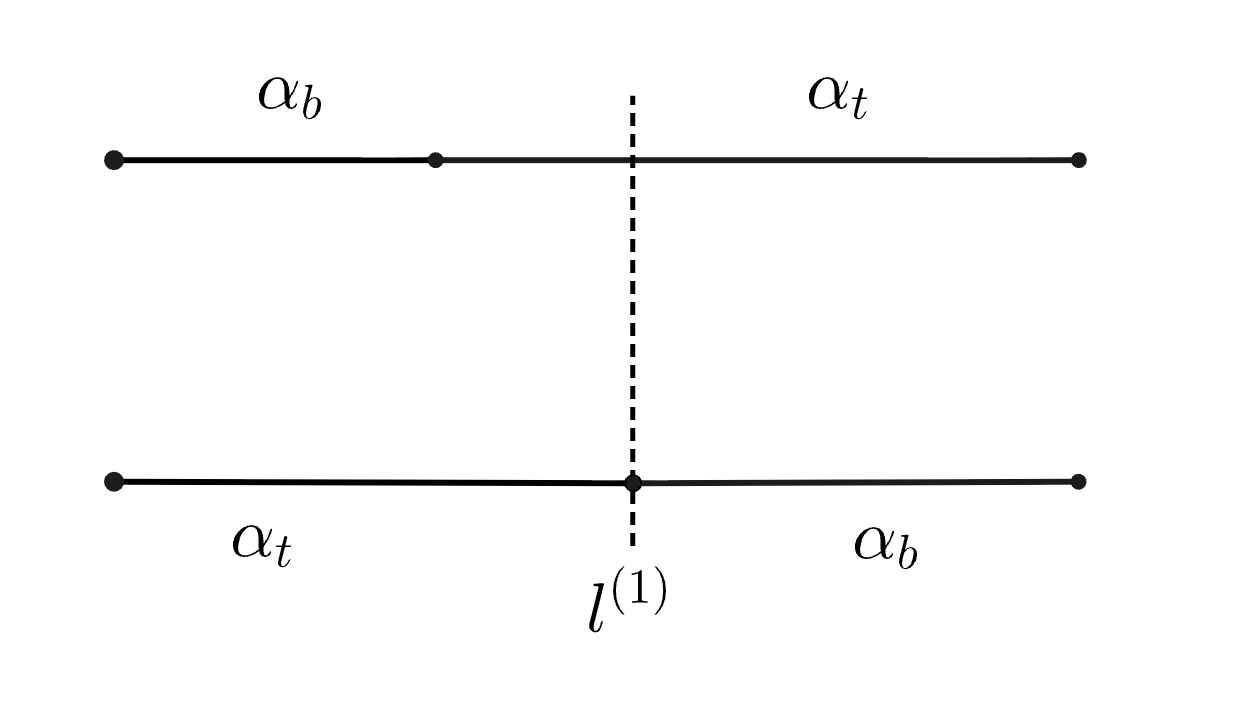}
\end{figure}

\item \textit{A tail plays (right-sided accumulation point)}:
	Assume $s_{\max}^t \neq l$, $s_{\max}^b \neq l$ and $s_{\max}^b \in I_{\alpha_t}^t$ (resp. $s_{\max}^t \in I_{\alpha_b}^b$) where $s_{\max}^b$ (resp. $s_{\max}^t$) is right-sided. Then $l^{(1)} = s_{\max}^b$ (resp. $s_{\max}^t$).

	\begin{figure}[H]
	\centering
	\includegraphics[width=0.45\textwidth]{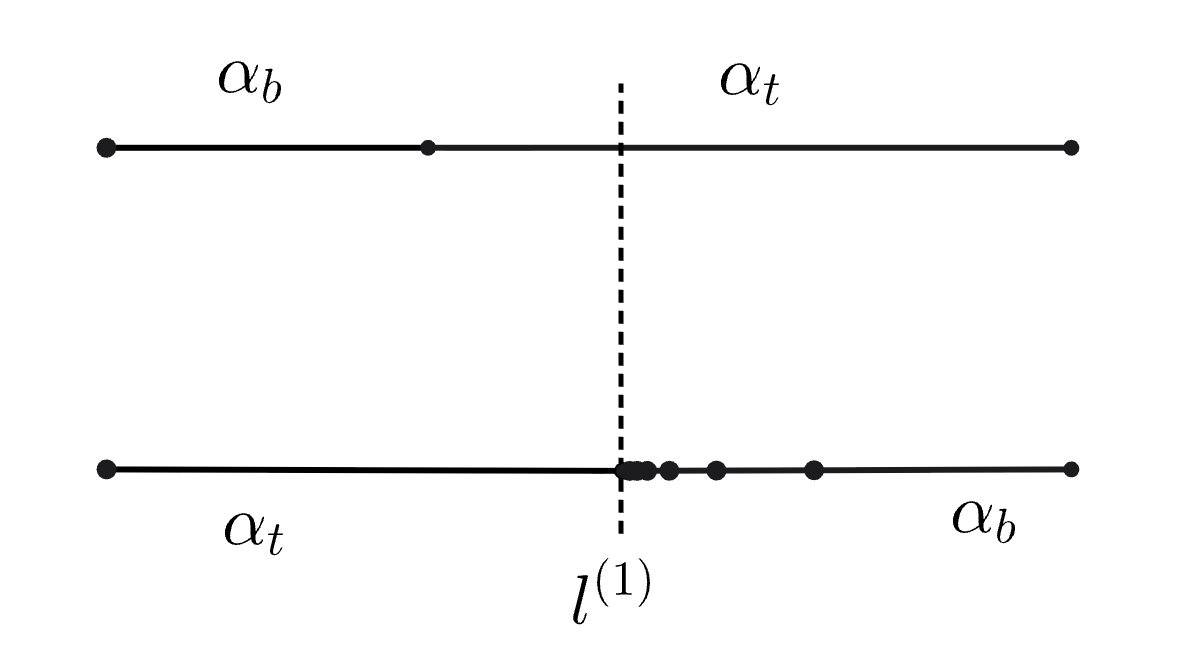}
	\end{figure}

\item \textit{A tail plays (left-sided accumulation point)}:
	assume $s_{\max}^b = l$ (resp. $s_{\max}^t = l$), then it is a left-sided accumulation point and
	$l^{(1)} = \min(\mathcal{S}^b \cap I_{\alpha_t}^t) \hspace{2mm}(\text{resp. min}(\mathcal{S}^t \cap I_{\alpha_b}^b)).$

\begin{figure}[H]
\centering
\includegraphics[width=0.45\textwidth]{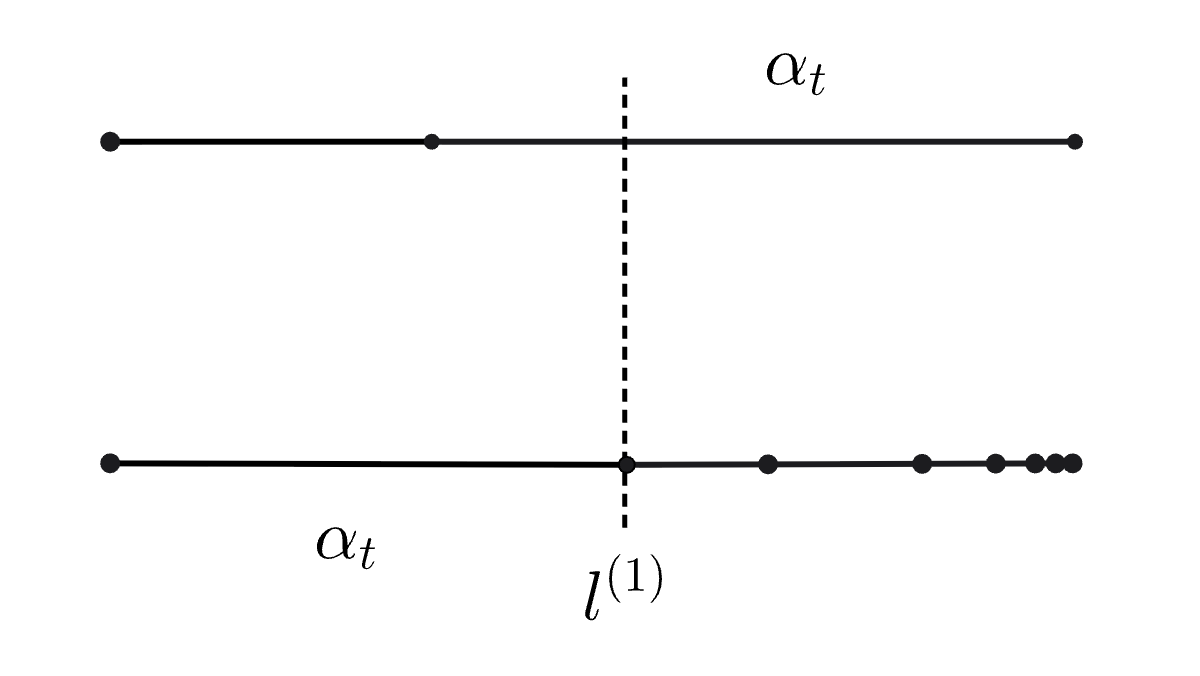}
\end{figure}
\end{itemize}

In all of these cases, if $l^{(1)} \in \mathcal S^t \cup \mathcal S^b$ then $T$ has a connection and the induction is not well defined.
If this is not the case, we say as in the case of finite-type IETs:
\begin{itemize}
	\item labels $\mathcal{L} \subset \mathcal A$ such that $I_\alpha^t \subset [l^{(1)}, l]$ or $I_\alpha^b \subset [l^{(1)}, l]$ for $\alpha \in \mathcal L$ are \textit{losing},
	\item the unique label $w \in \{\alpha_t, \alpha_b\}$ such that $l^{(1)} \in I_w^t$ or $I_w^b$ is \textit{winning}.
\end{itemize}

Note that the cardinality of $\mathcal{L}$ is either one or infinity. Note also that the losing intervals are all simultaneously on the top (resp. bottom). We say in short that top (resp. bottom) loses, and similarly that bottom (resp. top) wins.

\end{defs}

\subsubsection{Explicit description} We can explicitly describe $T^{(1)}:[0,l^{(1)}] \to [0,l^{(1)}]$
as an IET with alphabet $\mathcal{A}$, length vector $\lambda^{(1)}$ and top and bottom intervals $\{I^{t,(1)}_{\alpha_t}\}_{\alpha_t \in \mathcal{A}}$, $\{I^{b,(1)}_{\alpha_t}\}_{\alpha_t \in \mathcal{A}}$ as follows:

\begin{enumerate}[label=(\roman*)]
\item \textbf{(top wins)} if the top interval $w$ wins against a set of bottom intervals $\mathcal{L}$ which are the losers, then

\begin{align*}
    I_{\gamma}^{t,(1)} &= I_{\gamma}^{t} \cap [0,l^{(1)}], \hspace{2mm} \gamma \in \mathcal{A}^{},\\
 T^{(1)}_{\gamma} &=
    \begin{cases}
      T|_{I_{\gamma}^{t,(1)}}, &\hspace{2mm} \gamma \in \mathcal{A}^{} \backslash \mathcal{L},\\
      T^2|_{I_{\gamma}^{t,(1)}}, &\hspace{2mm} \gamma \in \mathcal{L}.
    \end{cases}
\end{align*}

Note that the order $\leq_t$, $\leq_b$ of the intervals in $\mathcal{A}^{}$ (from left to right in $[0,l^{(1)}]$) remains the same for the top partition but changes for the bottom partition: the bottom intervals labeled by $\mathcal{L}$ are now ordered right after the bottom interval labeled by $w$. Explicitly, we obtain two new orders $(\mathcal{A}^{}, \leq_t^{(1)})$, $(\mathcal{A}^{}, \leq_b^{(1)})$, where for $\alpha, \beta \in \mathcal{A}^{}$ the top order remains the same, i.e.
\begin{align*}
\alpha \leq_t^{(1)} \beta &\iff \alpha \leq_t \beta,
\end{align*}
and for the bottom order we have
\begin{align*}
\alpha \leq_b \beta &\iff \alpha \leq_b^{(1)} \beta \text{ if both $\alpha, \beta \in \mathcal{L}$ or both $\alpha, \beta \notin \mathcal{L}$, } \\
\alpha \leq_b w &\iff \alpha <_b^{(1)} \beta \text{ if $\alpha \notin \mathcal{L}$, $\beta \in \mathcal{L}$, } \\
w <_b \alpha &\iff \beta <_b^{(1)} \alpha \text{ if $\alpha \notin \mathcal{L}$, $\beta \in \mathcal{L}$. }
\end{align*}

\item \textbf{(bottom wins)} If now the bottom interval $w$ wins against a set of top intervals $\mathcal{L}$, then

\begin{align*}
I_{\gamma}^{t,(1)} &=
    \begin{cases}
      T^{-1}(I_{\gamma}^{b} \cap [0,l^{(1)}]), &\hspace{2mm} \gamma \in \mathcal{A}^{} \backslash \mathcal{L},\\
      T^{-1}(I_{\gamma}^{t}), &\hspace{2mm} \gamma \in \mathcal{L}.
    \end{cases} \\
 T^{(1)}_{\gamma} &=
    \begin{cases}
      T|_{I_{\gamma}^{t,(1)}}, &\hspace{9mm} \gamma \in \mathcal{A}^{} \backslash \mathcal{L},\\
      T^2|_{I_{\gamma}^{t,(1)}}, &\hspace{9mm} \gamma \in \mathcal{L}.
    \end{cases}
\end{align*}
In the second case, the order of the intervals remains the same for the bottom partition but changes for the top partition: the top intervals labeled by $\mathcal{L}$ keep their order among themselves but are inserted right after the top interval labeled by $w$ with rules given as in (i), where we exchange $\leq_t$ with $\leq_b$.
\end{enumerate}

Regarding the length vector, in both of the two cases explained above, the length vector $\lambda^{(1)}$ only changes at the entry corresponding to the winning letter, i.e.
\begin{equation}\label{eq:lengths}
\lambda_{\gamma}^{(1)}=
    \begin{cases}
   \lambda_{\gamma}^{} \hspace{2mm} &\hspace{2mm} \gamma \in \mathcal{A}^{} \backslash \{w\}, \\
      \lambda_{w} - \sum_{l \in \mathcal{L}} \lambda_l &\hspace{2mm} \gamma = w.
    \end{cases}
\end{equation}

\subsubsection{Renormalizability}\label{sec:iteratingtheinduction} We say that $T$ is renormalizable if we can apply the elementary step of Rauzy-Veech induction and if in addition neither $0$ nor $l$ is an accumulation point both of top and bottom singularities:\footnote{The assumption regarding $0$ is not necessary for the induction to be well-defined, but it will be needed in the proof of the minimality criterion.}

\begin{defs}\label{def:oncerenormalizable} We say that $T$ is \textit{renormalizable} if
\begin{enumerate}[label=(\roman*)]
    \item the rightmost top and bottom  accumulation points $s_{\max}^t$, $s_{\max}^b$ are not both equal to $l$,
    \item the leftmost top and bottom accumulation points $s_{\min}^t$, $s_{\min}^b$ are not both equal to $0$,
    \item the left endpoints of the playing intervals are not equal.
\end{enumerate}
\end{defs}

If $T^{(1)}$ is again renormalizable, we can apply the induction step repeatedly to obtain a sequence of IETs $T^{(n)}: I^{(n)} \to I^{(n)}$ together with
\begin{itemize}
    \item top and bottom intervals $\{I^{t,(n)}_{\alpha}\}_{\alpha \in \mathcal{A}^{}}$, $\{I^{b,(n)}_{\alpha}\}_{\alpha \in \mathcal{A}^{}}$,
    \item length vectors $\lambda^{(n)}$,
    \item orders $\leq_t^{(n)}$, $\leq_b^{(n)}$,
    \item winning and losing letters $w^{(n)}$ and $\mathcal{L}^{(n)}$.
\end{itemize}

\begin{ex}
    For the infinite-type IET from Section \ref{sec:theconstruction} in Figure \ref{fig:3}, the winner and loser of the first two induction steps are
    \begin{align*}
    w^{(1)} = \{a_1\},\; \mathcal{L}^{(1)}= \{b_2, a_3, b_3, \dots\}, \\
    w^{(2)} = \{a_2\},\; \mathcal{L}^{(2)}= \{a_1\}.
    \end{align*}
    At the first induction step, an interval at the bottom wins against a tail at the top, whereas at the second induction step, an interval at the top wins against an interval at the bottom.
\end{ex}

\begin{defs}\label{def:N times renormalizable} For $N \in \N$, we say that an IET $T$ is \textit{renormalizable up to time $N$} if $T^{(n)}: [0, l^{(n)}] \to [0, l^{(n)}]$ is renormalizable for $n=0,1, \dots N-1$. A renormalizable IET $T: [0,l] \to [0,l]$ is called \textit{infinitely renormalizable} if it is renormalizable for all $N \in \N$.
\end{defs}

We now give a sufficient (but not necessary) condition for $T$ to be infinitely renormalizable.

\begin{lem} An IET $T$ is infinitely renormalizable if it is renormalizable and if there are no connections.
\end{lem}

\begin{proof} From Definition \ref{def:RVinduction} of the induction it follows that if there exists $n \in \N$ such that the $n$-th induction step is not defined, then $T^{(n-1)}$ has a connection.
\end{proof}

\subsubsection{Towers over an IET}\label{sec:towers} To conclude this section, we want to define the notion of \textit{towers over an IET}. Let $T: I \to I$ be an IET on top and bottom intervals $\{I^{t}_{\alpha}\}_{\alpha \in \mathcal{A}}$, $\{I^{b}_{\alpha}\}_{\alpha \in \mathcal{A}}$ and let $T^{(n)}: I^{(n)} \to I^{(n)}$ be the IET obtained after $n$ steps of Rauzy-Veech induction with top and bottom intervals $\{I^{t,(n)}_{\alpha}\}_{\alpha \in \mathcal{A}^{(n)}}$, $\{I^{b,(n)}_{\alpha}\}_{\alpha \in \mathcal{A}^{(n)}}$.

\begin{defs} Let $\alpha \in \mathcal{A}$. We define the \textit{tower} over $T$ at time $n$ labelled by $\alpha$ as the finite, disjoint union of intervals
\begin{align*}
    \text{Tow}_T(\alpha,n) := \bigsqcup_{i=0}^{
    r_\alpha^{(n)}-1} T^{i}(I^{t,(n)}_{\alpha}),
\end{align*}
where $r_{\alpha}^{(n)}$ is the first return time\footnote{i.e. the smallest integer $i \geq 1$ such that $T^{i}(I^{t,(n)}_{\alpha}) \subset I^{(n)}$.} of $I^{t,(n)}_{\alpha}$ to $I^{(n)}$. We call the interval $I^{t,(n)}_{\alpha}$ the \textit{base interval} of the tower and $r_\alpha^{(n)}$ the \emph{height} of the tower. Note that up to finitely many points, we have
\begin{align*}
    I = \bigsqcup_{\alpha \in \mathcal{A}} \text{Tow}_T(\alpha,n).
\end{align*}
\end{defs}
Note that furthermore for all $m \in \N$ it holds
\begin{align*}
     \text{Tow}_{T^{(m+1)}}(\alpha,n) \subset \text{Tow}_{T^{(m)}}(\alpha,n).
\end{align*}

\subsection{Rauzy-Veech Matrices}\label{sec:standardRVmatrix} We define the \textit{Rauzy-Veech matrices} associated to a sequence of induction steps which similarly to the finite-type Rauzy-Veech matrices introduced in \cite{Rauzy} describe how the length vector changes throughout the induction. We introduce our version of \textit{Keane's criterion} for infinite-type IETs and prove Theorem A.

\subsubsection{Infinite Matrices and the space of lengths}\label{sec:linearoperator} In this section, we implicitly use the enumeration map from Section \ref{sec:enumerationmap}. Consider the set of possibly infinite matrices of the form $M = (m_{\alpha \beta})_{\alpha,\beta \in \mathcal{A}} \in \mathbb{R}_{>0}^{\mathcal{A} \times \mathcal{A}}$ with positive real entries. We assume that the entries of $M$ are bounded above, i.e. for some $N \in \N$ it holds $m_{\alpha \beta} < N$ for all $\alpha, \beta \in \mathcal{A}$. We define also the \textit{space of lengths} $L$ as
\begin{align*}
    L = \{\lambda \in \R_{>0}^\mathcal{A}\hspace{1mm}|\hspace{1mm}\sum_{\alpha \in \mathcal{A}}\lambda_{\alpha} < \infty \}.
\end{align*}
Note that the matrix $M$ defines a linear operator $M: L \to L$ by setting
\[
\lambda \mapsto M\lambda, \qquad
(M\lambda)_{\alpha} = \sum_{\beta\in \mathcal{A}} m_{\alpha \beta} \lambda_{\beta}, \quad \alpha \in \mathcal{A},
\]
where the entries of $M\lambda$ are again in $L$ since $M$ is bounded.

\subsubsection{Definition of the matrices} We first define the matrices for the $n-$th induction step of an $n$ times, resp. infinitely renormalizable IET.

\begin{defs}\label{def:RVmatrix}
	At the $n$-th induction step, let $w^{(n)}$ and $\mathcal{L}^{(n)}$ be the labels of the winner and the losers as in Definition \ref{def:RVinduction}. Define the \textit{Rauzy-Veech matrix at step $n$} as follows:
\begin{equation}\label{eq:rvmatrix}
    M^{(n)} = Id + E_{w^{(n)}\mathcal{L}^{(n)}}
\end{equation}
where $Id = (\delta_{\alpha \beta})_{\alpha,\beta \in \mathcal{A}}$ denotes the infinite identity matrix and $E_{w^{(n)}\mathcal{L}^{(n)}}$  denotes the elementary matrix with coefficient 1 at row $w^{(n)}$ and columns indexed by $\mathcal{L}^{(n)}$ and zeroes everywhere else.
Note that in each column, $M^{(n)}$ has finitely many entries which are nonzero, but may have rows with infinitely many nonzero entries.
\end{defs}

Let now $m < n$ and define the matrix product\footnote{An efficient way to calculate $M^{(m,n)}$ by hand is as follows: we begin with the matrix $M^{(n)}$ at time $n$ and add at each step the winning column to the losing column(s) until we reach time $m$.}
\begin{align*}
    M^{(m,n)} = M^{(m)} \cdot M^{(m+1)} \dots M^{(n)}.
\end{align*}
Note that the matrix product is well-defined since in each column of the matrices $M^{(n)}$ there are only finitely many positive entries.
Note further that the coefficient $M^{(m,n)}_{\alpha, \beta}$ is the number of times Tow$_{T^{(m)}}(\beta,n$) intersects $I_{\alpha}^{t,(m)}$.

\begin{ex}
    Consider the infinite-type IET from Section \ref{sec:theconstruction} in Figure \ref{fig:3},     recall that $w^{(1)} = \{a_1\},\; \mathcal{L}^{(1)}= \{b_2, a_3, b_3, \dots\}$ and $w^{(2)} = \{a_2\},\; \mathcal{L}^{(2)}= \{a_1\}$. The corresponding matrices are
\begin{align*}
M^{(1)} = \begin{pmatrix}
0& a_1 & c & b_1 & a_2 & b_2 & a_3 & b_3 ...\\
a_1 & 1 & 0 & 0 & 0 & 1 & 1 & 1 ...\\
c & 0 & 1 & 0 & 0 & 0 & 0 & 0 ...\\
b_1 & 0 & 0 & 1 & 0 & 0 & 0 & 0 ...\\
a_2 & 0 & 0 & 0 & 1 & 0 & 0 & 0 ...\\
b_2 & 0 & 0 & 0 & 0 & 1 & 0 & 0 ..\\
a_3 & 0 & 0 & 0 & 0 & 0 & 1 & 0 ..\\
b_3 & 0 & 0 & 0 & 0 & 0 & 0 & 1 ...
\end{pmatrix},
M^{(1,2)} = \begin{pmatrix}
0& a_1 & c & b_1 & a_2 & b_2 & a_3 & b_3 ...\\
a_1 & 1 & 0 & 0 & 0 & 1 & 1 & 1 ...\\
c & 0 & 1 & 0 & 0 & 0 & 0 & 0 ...\\
b_1 & 0 & 0 & 1 & 0 & 0 & 0 & 0 ...\\
a_2 & 1 & 0 & 0 & 1 & 0 & 0 & 0 ...\\
b_2 & 0 & 0 & 0 & 0 & 1 & 0 & 0 ..\\
a_3 & 0 & 0 & 0 & 0 & 0 & 1 & 0 ..\\
b_3 & 0 & 0 & 0 & 0 & 0 & 0 & 1 ..\\
\end{pmatrix}
\end{align*}
\end{ex}

The structure of the matrices $M^{(m,n)}$ can be described as follows:
\begin{lem}\label{lem:formrvmatrix}
	\label{prop:rauzy_mat}
	For every matrix $M^{(m,n)}$ there exists a finite subset $I \subset \mathcal A$, and a partition of the complementary set $\mathcal T_1 \sqcup \dots \sqcup \mathcal T_k = \mathcal A \setminus I$, where $k$ is the number of tails, such that the following properties are satisfied:
	\begin{enumerate}
		\item for all $(\alpha,\beta) \in I \times \mathcal A$, $\rauzymat_{\alpha, \beta}^{(m,n)} \in \{0,1, \dots, 2^{n-m} \}$,
		\item for all $1 \le i \le k$, $\alpha \in I$ and $\beta_1, \beta_2 \in \mathcal T_i$, $\rauzymat_{\alpha, \beta_1}^{(m,n)} = \rauzymat_{\alpha, \beta_2}^{(m,n)}$,
		\item for all $(\alpha,\beta) \in (\mathcal A \setminus I) \times \mathcal A$, $\rauzymat_{\alpha, \beta}^{(m,n)} = \delta_{\alpha,\beta}$.
    \end{enumerate}
For such subsets $I$ we say that the matrix is $I$-supported.
We say it is strictly $I$-supported if it is moreover strictly positive on rows in $I$.

\end{lem}
\begin{proof}
    Follows from multiplying the elementary matrices as defined in (\ref{eq:rvmatrix}).
    The $2^{n-m}$ upper bound comes from the fact that multiplying by an elementary matrix can at most double the maximum of coefficients.
\end{proof}
Hence the entries of $M^{(m,n)}$ are bounded above and according to the previous Section \ref{sec:linearoperator} they each define a linear operator from $L \to L$.
An important fact is that as in the finite-type case, this linear operator relates the lengths $\lambda^{(n)} \in L$ of the iterated maps $T^{(n)}$ as follows:

\begin{lem}\label{prop:lengthrelation} Let $T$ be an $n$ times renormalizable, infinite-type IET, let $m < n$. Then the following holds:
    \begin{align*}
   \lambda^{(m)} = M^{(m,n)} \cdot \lambda^{(n)}.
    \end{align*}
\end{lem}
\begin{proof} Follows from Equation \ref{eq:lengths} of Definition \ref{def:RVinduction}.
\end{proof}

\subsection{The minimality criterion}\label{sec:combinatorialminimality} We now introduce and prove our generalization of Keane's criterion.

\subsubsection{Minimality and combinatorial minimality} We define the notion of \textit{combinatorial minimality} of an IET $T$, which holds if infinitely often any entry of the sequence of Rauzy-Veech matrices corresponding to $T$ is eventually positive.
\begin{defs}
\label{def:irred}
Let $T$ be infinitely renormalizable. We say that an IET $T$ is \textit{combinatorially minimal} if for all $\alpha, \beta \in \mathcal{A}$, for all $k \in \mathbb{N}$ there exists $n \in \mathbb{N}$ such that $M^{(k,k+n)}_{\alpha,\beta} > 0$.
\end{defs}

\begin{rem}\label{rem:monotoneincreasing} By definition of the matrix product the entries $M^{(k,k+n)}_{\alpha,\beta}$ are monotone increasing in $n$, in particular if $M^{(k,k+n)}_{\alpha,\beta} > 0$, then $M^{(k,k+m)}_{\alpha, \beta} > 0$ for all $m \geq n$.
\end{rem}

Our aim now is to prove Theorem A, i.e. we want to show that for infinitely renormalizable IETs, combinatorial minimality is equivalent to minimality and absence of saddle connections. We first prove the following auxiliary lemmas:

\begin{lem}\label{lem:heights&lengths} Let $T: [0,l] \to [0,l]$ be an infinitely renormalizable IET and let $\alpha \in \mathcal{A}$. If $\alpha$ loses infinitely often, it holds:
\begin{enumerate}[label = (\roman*)]
    \item the length of the interval labelled by $\alpha$ goes to zero, i.e. $\lim_{n \to \infty} \lambda_{\alpha}^{(n)} \to 0$,
    \item the height of the tower over $T$ labelled by $\alpha$ goes to infinity, i.e. $\lim_{n \to \infty} r_\alpha^{(n)} \to \infty$.
\end{enumerate}
\end{lem}

\begin{proof} Since $\alpha$ loses infinitely often, there exists an infinite, increasing sequence $(n_k)_{k \in \N} \subset \N$ such that the label $\alpha$ loses at the $n_k$-th induction step for all $k \in \N$. But then
\begin{align*}
    l - \sum_{k \in \N} \lambda_{\alpha}^{(n_k)} \ge 0 \implies \lambda_{\alpha}^{(n)} \to 0.
\end{align*}
Furthermore, whenever a letter loses, the height of the corresponding tower grows, hence the height $r_{\alpha}^{(n)}$ of the tower Tow($\alpha,n$) goes to infinity.
\end{proof}

\begin{lem}\label{lem:winlose} Let $T: [0,l] \to [0,l]$ be an infinitely renormalizable IET and let $\alpha \in \mathcal{A}$. If $\alpha$ wins infinitely often, it must also lose infinitely often.
\end{lem}

\begin{proof} As in Definition \ref{def:RVinduction}, we denote by $w$ the label of the single interval which wins at the first induction step, w.l.o.g. assume it is winning on the top. Define $s_0 :=$ min$(\mathcal{S}^b \, \cap \, I_{w}^t)$. Note further that by definition of the induction, if $I_w^t$ wins, then it continues playing on the top at the next induction step.

If $I_{w}^t \cap I_{w}^b = \emptyset$, then since by definition $T$ has only finitely many accumulation points of intervals, there exists $n \in \N$ such that $l^{(n)} = s_0$ and then at the next step $I_{w}^t$ loses.

If $I_{w}^t \cap I_{w}^b \neq \emptyset$, then there also exists $n \in \N$ such that $l^{(n)} = s_0$, where now $s_0$ is the right endpoint of $I_{w}^b$. Then at the next induction step, $|I_{w}^{t,(n)}| = |I_{w}^{t}| - (l - s_0)$. Similarly, if up to induction time $k \cdot n$ the label $w$ always wins, we obtain $|I_{w}^{t,(2k)}| = |I_{w}^{t}| - k(l - s_0)$. But since the length $|I_{w}^{t}|$ is finite, we can only repeat this process finitely many times, hence eventually $I_{w}^{t}$ must lose.
\end{proof}

Let $T$ be infinitely renormalizable and consider now the sequence of IETs $T^{(n)}: [0, l^{(n)}] \to [0, l^{(n)}]$ obtained by Rauzy-Veech induction.

\begin{lem}\label{lem:leftendpointremains} Define $x_0 = \lim_{n \to \infty} l^{(n)}$. Assume there exists $\alpha \in \mathcal{A}$ such that
$$\inf  \, I_{\alpha}^t = x_0.$$ Then $\inf  \,I_{\alpha}^{t,n} = x_0$ for all $n \in \N$ and $\alpha$ plays infinitely often, where whenever it plays, it plays at the bottom.
\end{lem}

\begin{proof} We first claim that $I_{\alpha}^{b,(n)} \subset [x_0,l]$ for all $n \in \N$: indeed, otherwise, $I_{\alpha}^{t,(k)} = I_{\alpha}^{t}$ for all times $k$ until at the $(k+1)$-th induction step $\alpha$ plays at the top for the first time, but then $\alpha$ has to lose eventually by Lemma \ref{lem:winlose}, meaning that eventually $l^{(n)} < x_0$ which is a contradiction.

Note also that $\inf  \,I_{\alpha}^{b} \neq x_0$ since otherwise $T$ would not be infinitely renormalizable. Hence we have $\sup \,I_{\alpha}^{t} < \sup \,I_{\alpha}^{b}$ and by induction
\begin{equation*}\label{eq:supineq}
    \sup \,I_{\alpha}^{t,(n)} < \sup \,I_{\alpha}^{b,(n)}
\end{equation*}
for all $n \in \N$. Hence whenever $\alpha$ plays, it must play at the bottom, in particular, the left endpoint of $I_{\alpha}^{t,(n)}$ never changes. Note further that since the length $\lambda_{\alpha}$ goes to zero, $\alpha$ must win and in particular play infinitely often.
\end{proof}

\begin{prop}\label{prop:renormalizationlimitandorbitclosures}
     Let $T$ be an infinitely renormalizable IET without connections between accumulation points and define $x_0 = \lim_{n \to \infty} l^{(n)}$. Then $x_0$ is the leftmost endpoint of a quasiminimal.
     Furthermore, the leftmost endpoint of any orbit closure of $T$ is bounded from above by $x_0$.
\end{prop}
\begin{proof}
\textbf{The forward orbit of $x_0$ is infinite}. Since $T$ has no connections between accumulation points, $x_0$ is not a right-sided accumulation point of both top and bottom singularities, hence (up to considering $T^{-1}$) we can assume that there exists $\alpha \in \mathcal{A}$ such that $x_0 \in I_{\alpha}^t$ or $x_0$ is the left endpoint of $I_{\alpha}^t$. Let $I_{\alpha'}^t$ be the top interval of $T$ whose leftmost endpoint is equal to $x_0$ (where one may have to introduce a fake singularity at $x_0$). By Lemma \ref{lem:leftendpointremains}, the intervals $I_{\alpha'}^{t,n}$ will have $x_0$ as a leftmost endpoint for all $n \in \N$. Since the induction accumulates at $x_0$, the length of the interval $I_{\alpha'}^{t,n}$ must go to zero.
Thus it must win infinitely many times and by Lemma \ref{lem:winlose} also lose infinitely many times, in particular the height of the tower labelled by $\alpha'$ over $T$ goes to infinity.
The left endpoints of the intervals in this tower are contained in the forward orbit of $x_0$ (where if $x_0$ is a singularity of $T$ we define here $T(x_0)$ using extension by left-continuity of $T$). Since $T$ is infinitely renormalizable, we can assume that this orbit is infinite: indeed, this is clear if the left endpoint $x_0$ of $I_{\alpha'}^t$ was a genuine and not a fake singularity, because otherwise $T$ would not be infinitely renormalizable. If we had to introduce a fake singularity and the orbit through $x_0$ was not infinite, then there exists a time $k$ where the forward orbit of $x_0$ meets the left endpoint of some $I_{\beta}^{t,k}$ for some $\beta \in \mathcal A$. We can then repeat the argument starting at time $k$ with $I_{\beta}^{t,k}$ as the interval whose leftmost endpoint is a genuine singularity equal to $x_0$.

\textbf{The forward orbit of $x_0$ is non-trivially recurrent.} Since the induction and hence also the intervals of the tower accumulate at $x_0$, for all $\epsilon$ there exists $k \in \mathbb{N}$ such that $|x_0 - T^{(k)}(x_0)| < \epsilon$ and hence the infinite forward orbit through $x_0$ is $\omega-$recurrent. It is non-trivially $\omega-$recurrent, since if it was trivially recurrent, there would exist $n \in \N$ such that $T^{(n)}$ has a connection between the right endpoints of the two top and bottom intervals which contain $x_0$ and hence $T$ would not be infinitely renormalizable. Hence the infinite forward orbit through $x_0$ is non-trivially $\omega-$recurrent and hence its closure is equal to a quasiminimal. The fact that $x_0$ is the leftmost endpoint of this quasiminimal follows from the fact that by Lemma \ref{lem:leftendpointremains} all intervals in the tower of $\alpha'$ are contained in $[x_0,l]$ for all $n \in \N$.

\textbf{The leftmost endpoint of other orbit closures are bounded by $x_0$.}
On the other hand, assume there exists an orbit $w$ whose closure $\mathcal W$ has leftmost endpoint equal to $x$ such that $x > x_0$. If $w$ was finite, then $T$ would not be infinitely renormalizable. Hence we can assume that $w$ is an infinite orbit. Let $n \in \N$ be such that $l^{(n)} < x$ and $\alpha$ be the label of an interval $I_\alpha^{t,(n)}$ in $[0,l^{(n)}]$ such that its tower over $T$ meets the orbit $w$: note that such an interval exists since the union of all towers at time $n$ over $T$ form a partition of $[0,l)$ up to a countable number of points, namely the left endpoints of the intervals in each tower. But since $w$ is infinite, it cannot entirely consist of such left endpoints of towers at time $n$.
Hence, $w$ meets the tower over $I_\alpha^{t,(n)}$ and since $I_\alpha^{t,(n)}$ returns to $[0,l^{(n)}]$ this means that also $w$ returns which is a contradiction.
\end{proof}

\begin{theoremA}\label{thm:minimalcriterion} An infinitely renormalizable IET $T$ is combinatorially minimal if and only if it has no connections and is minimal.
\end{theoremA}

\begin{proof} Let us first show that if $\lim_{n \to \infty} l^{(n)} = x_0 > 0$ then $T$ is neither combinatorially minimal nor without connections and minimal. By Proposition \ref{prop:renormalizationlimitandorbitclosures} if $T$ has no connections then the point $x_0 > 0$ is the leftmost endpoint of a quasiminimal, thus if $\lim_{n \to \infty} l^{(n)} = x_0 > 0$ then $T$ cannot be minimal. On the other hand, if the induction does not accumulate at zero there exists $\alpha \in \mathcal{A}$ such that the length $\lambda_{\alpha}^{(n)}$ does not go to zero as $n$ goes to infinity. But then by Lemma \ref{lem:heights&lengths} it follows that $\alpha$ only loses finitely many times, in particular there exists $k \in \N$ such that $\alpha$ does not lose anymore after time $k$ and hence $M_{\beta,\alpha}^{(k,k+n)} = 0$ for all $n \in \N$ and all $\beta \neq \alpha \in \mathcal{A}$, hence $T$ is not combinatorially minimal.

Hence for both directions of the statement we may assume that $\lim_{n \to \infty} l^{(n)} = 0$. Without loss of generality, assume that $0 \notin Acc(\mathcal{S}^t)$ (otherwise we can consider $T^{-1}$ since $T$ is renormalizable and hence $0$ is not an accumulation point of both top and bottom singularities). We prove both directions separately.\\

$\boxed{\Rightarrow}$ Assume $T$ is combinatorially minimal. It is clear that if $\lim_{n \to \infty} l^{(n)} = 0$, then $T$ has no connections, otherwise $T$ would not be infinitely renormalizable. Hence we are left to show that $T$ is minimal. Consider an infinite orbit $w$. According to Proposition \ref{prop:renormalizationlimitandorbitclosures}, since $x_0 = 0$, the orbit $w$ must accumulate either in forward or backward time at $0$.
 Since $0 \notin Acc(\mathcal{S}^t)$ there exists a leftmost top interval, let $\alpha$ be its label. Then by Lemma \ref{lem:leftendpointremains} the letter $\alpha$ is the label of the leftmost top interval at any induction step, in particular, since $w$ also accumulates at $0$ when taking its first return map to $[0,l^{(n)}$], we have
 Tow$_{T}(\alpha,n) \cap w \neq \emptyset$ for all $n \in \N$.
Let now $x \in [0, l)$ and $\epsilon > 0$.
 Take $k \in \N$ large enough such that all intervals in $I^{(k)}$ have length less than $\frac{\epsilon}{2}$. Let $\beta$ be a label of the interval in $I^{(k)}$ so that all points in its tower over $T$ come $\epsilon$-close to $x$ (note that $x$ may also be a singularity). By assumption there exists $n \in \N$ such that $M^{(k, k+n)}_{\beta, \alpha} > 0$ thus Tow$_T(\alpha, k+n$)
 meets $I^{(k)}_\beta$ and all the intervals in its tower.
 As the orbit of $w$ intersects Tow$_{T}(\alpha, k+n$), it meets $I^{(k)}_\beta$ and hence also comes $\epsilon-$close to $x$. \\
$\boxed{\Leftarrow}$
Assume $T$ has no connections and is minimal, but by contradiction is not combinatorially minimal, i.e. there exist $\alpha, \beta \in \mathcal{A}$ and $k \in \N$ such that $M_{\beta,\alpha}^{(k,k+n)} = 0$ for all $n \in \mathbb{N}$, i.e. for no $n \in \mathbb{N}$, Tow$_{T^{}}(\alpha, k+n$) meets $I_{\beta}^{(k)}$.
Since the induction accumulates at zero, the length of all intervals goes to zero, in particular all letters win infinitely often, and by Lemma \ref{lem:winlose} also all letters lose infinitely often. By Lemma \ref{lem:heights&lengths} the height of any tower over $T^{(n)}$ goes to infinity, in particular the one corresponding to the letter $\alpha$. This means that the forward or backward infinite orbit $w$ corresponding to the left endpoint of the intervals in Tow$_{T^{}}(\alpha, k+n$) for all $n \in \N$ has the property that $w \cap I_{\beta}^{(k)} = \emptyset$, in particular, $T$ is not minimal.
\end{proof}

\subsubsection{Irreducibility and the invariant Cantor set} We return to the counterexample from Section \ref{sec:theconstruction}: the IET in Figure \ref{fig:3} is infinitely renormalizable since by construction it has no connections and its intervals do not accumulate at $0$ (or $l$) both at the top or at the bottom. It further has a quasiminimal $\Omega$ which is a Cantor set, hence it is not minimal. We claim that it is also not combinatorially minimal.

\begin{ex}
    The IET from Figure \ref{fig:3} is not combinatorially minimal.
\end{ex}

Indeed, by construction of the IET it holds for the blue intervals that $I_{b_i}^t \cap \Omega = \emptyset$ for all $i \in \N$. Note furthermore that the left endpoint of $I_{a_1}^{t,n}$ is equal to zero for all $n \in \N$ and since $\Omega$ accumulates at zero it holds that Tow$_T(a_1,n) \cap \Omega \neq \emptyset$ for all $n \in \N$, in particular Tow$_T(a_1,n) \cap I_{b_i}^t = \emptyset$ for all $n \in \N$. Hence $M^{(0,n)}_{b_i, a_1}=0$ for all $n,i \in \mathbb{N}$ which implies that the IET is not combinatorially minimal.

\section{Tail-reversing IETs}\label{sec:revealing}
Having defined Rauzy-Veech induction for the class of renormalizable IETs, we now restrict ourselves to the case of \textit{tail-reversing} IETs. These are infinite-type IETs whose intervals can be partitioned into sets of single intervals and sets of infinitely many intervals (the \textit{tail intervals}), where the order of intervals on the tails is reversed. The key advantage of an IET being tail-reversing is that we can define a \textit{finite} permutation which describes the order of the grouped intervals, and record how this finite permutation changes when performing Rauzy-Veech induction using an additional technique which we call \textit{revealing}. This then allows us to define infinite-type \textit{Rauzy diagrams} which we use to describe this change in permutations.

\subsection{Tail-reversing IETs and their permutations} We define tail-reversing IETs as well as their permutations and introduce the notion of \textit{properness} and \textit{irreducibility} of their permutations. Before we do so, we quickly recall the definition of a permutation in the finite-type case.

\subsubsection{Permutations in the finite-type case} The classical combinatorial datum used to define Rauzy-Veech induction and Rauzy diagrams for finite-type IETs (see \cite{Yoccoz_Course}) are \textit{finite-type permutations}. Given a finite-type IET $T$ on alphabet $\mathcal A$, where $|\mathcal A|=d$, its finite-type permutation $\pi = (\pi_t, \pi_b)$ or its \textit{finite-type combinatorial datum} consists of two bijective maps
\[
\pi_t,\pi_b:\mathcal A \to \{1,\dots,d\},
\]
called the \textit{top} and \textit{bottom} finite-type permutation, so that the order of the top and bottom intervals of $T$ from left to right is given by
\[
I_{\pi_t^{-1}(1)}, I_{\pi_t^{-1}(2)},\ldots, I_{\pi_t^{-1}(d)}
\qquad\text{and}\qquad
I_{\pi_b^{-1}(1)}, I_{\pi_b^{-1}(2)},\ldots, I_{\pi_b^{-1}(d)}.
\]
The finite-type combinatorial data is also recorded by writing
\vspace{2mm}
\[
	\pi = \begin{pmatrix}
		\pi^{-1}_t(1) & \dots & \pi^{-1}_t(d)\\
		\pi^{-1}_b(1) & \dots & \pi^{-1}_b(d)
	\end{pmatrix}.
\]
\vspace{0.2mm}

\subsubsection{Definition of tail-reversing IETs}
Consider now a finite partition $\mathcal{P}$ of $\mathcal{A}$ with $d \geq 2$ elements into sets
$\mathcal{A} = \mathcal{A}_1 \sqcup \dots \sqcup \mathcal{A}_d$
which we call \textit{groups}.
We define the corresponding \textit{grouped top and bottom intervals}
$$
I_{\mathcal{A}_n}^t := \bigsqcup_{\alpha \in \mathcal{A}_n} I_{\alpha}^t \quad \text{and} \quad
I_{\mathcal{A}_n}^b := \bigsqcup_{\alpha \in \mathcal{A}_n} I_{\alpha}^b.
$$

\begin{defs}
\label{def:tail-reversing}
An IET on alphabet $\mathcal{A}$ is called \textit{tail-reversing} if there is such finite partition $\mathcal P$ satisfying:

\begin{enumerate}
\item for all $1 \leq n \leq d$ the cardinality $|\mathcal{A}_n|$ is either $1$ or $\infty$,
\item the closures $\overline{I_{\mathcal{A}_n}^t}$, $\overline{I_{\mathcal{A}_n}^b}$ are connected and
$T(I_{\mathcal{A}_n}^t) = I_{\mathcal{A}_n}^b $,
\item the endpoints of the infinitely many intervals contained in an infinite grouped top interval $I_{\mathcal{A}_n}^{t}$ have a unique accumulation point which is either $\inf I_{\mathcal{A}_n}^{t}$ or $\sup I_{\mathcal{A}_n}^{t}$,
\item the order of the intervals in $I_{\mathcal{A}_n}^t$ and $I_{\mathcal{A}_n}^b$ is reversed.
\end{enumerate}
\end{defs}

If a grouped interval is composed of infinitely many intervals it is called a \textit{tail interval}, otherwise it is called a \textit{single interval}.
We say it is \textit{right-sided} if its accumulation point is on the left --- the intervals accumulate from the right --- otherwise it is called \textit{left-sided}.

\subsubsection{Permutation of a tail-reversing IET}\label{sec:permutationtail-reversing} Let $T$ be a \textit{tail-reversing} IET, let $\mathcal{P}$ be a partition for $T$ with groups $\mathcal{A}_1, \dots \mathcal{A}_d$ for $d \geq 2$. We can define the permutation of $T$ in a way similar to permutations of finite-type IETs, however, we also want to keep track of the direction of the tail intervals using a direction map $\tailtype: \mathcal{P} \to \{\leftarrow, \bullet, \rightarrow\}$, where for $\mathcal{A}_n \in \mathcal{P}$ we have
\begin{align*}
   \tailtype(\mathcal{A}_n) =
    \begin{cases}
     \leftarrow \hspace{2mm} & \text{ if $I_{\mathcal{A}_n}^t$ is a right-sided tail interval, }\\
      \bullet \hspace{2mm} & \text{ if $I_{\mathcal{A}_n}^t$ is a single interval},\\
      \rightarrow \hspace{2mm} & \text{ if $I_{\mathcal{A}_n}^t$ is a left-sided tail interval. }
   \end{cases}
\end{align*}

In order to describe the order of the grouped intervals in $[0,l]$, we use two maps
\begin{align*}
  \pi_t, \pi_b: \mathcal{P} \to \{1, . . . , d\},
\end{align*}
where $d$ is the number of groups and the maps $\pi_t$ and  $\pi_b$, called the \textit{top and bottom permutations}, are bijective.
They describe the order of the grouped intervals in the top and bottom partition respectively, so that the order of the top and bottom  grouped intervals from left to right is
\begin{align*}
I_{\pi^{-1}_t(1)} , I_{\pi^{-1}_t(2)} , \dots I_{ \pi^{-1}_t(d)} \quad \text{and} \quad I_{\pi^{-1}_b(1)}, I_{\pi^{-1}_b(2)} , \dots , I_ {\pi^{-1}_b(d)}.
\end{align*}
We may also think of $\pi = (\pi_t, \pi_b)$ as a finite-type permutation on $d$ intervals.


\begin{defs} \label{def:permutation} We call the set $\Pi := (\mathcal{P}, \tailtype, \pi)$, where $\pi = (\pi_t, \pi_b)$, the \textit{combinatorial data} or the \textit{(grouped) permutation} of $T$.
	The combinatorial data is also recorded by writing
\[
	\Pi = \begin{pmatrix}
		\pi^{-1}_t(1) & \dots & \pi^{-1}_t(d)\\
		\pi^{-1}_b(1) & \dots & \pi^{-1}_b(d)
	\end{pmatrix},
\]
where we write arrows over the elements of $\mathcal{P}$ in $\Pi$ corresponding to tail intervals which indicate their direction, i.e. we write $\overleftarrow{\boldsymbol{\cdot}}$ if $\tailtype(\boldsymbol{\cdot})$ is equal to $\rightarrow$ and $\overrightarrow{\boldsymbol{\cdot}}$ if $\tailtype(\boldsymbol{\cdot})$ is equal to $\leftarrow$.
If $\Pi$ has at least one tail interval, we say it is of \textit{infinite-type}; otherwise its tail-type map $\tailtype$ is trivial and $\Pi$ amounts to the finite-type permutation $\pi$ on $d$ intervals.

\end{defs}

Note that a tail-reversing IET $T$ is uniquely determined by its permutation $\Pi$ together with its length vector $\lambda$ from Definition \ref{def:lengthvector}.
We write $T_{(\Pi, \lambda)}$ for the corresponding IET and say $(\Pi, \lambda)$ is a \textit{presentation} of $T$ if $T = T_{(\Pi, \lambda)}$.
The opposite however is not true: a tail-reversing IET can be presented by infinitely many permutations, as explained below.
\subsubsection{An example: the inverse permutation}\label{sec: inverse permutation} Consider the interval exchange transformation $T$ on alphabet $\mathcal{A} = \{\alpha_1, \alpha_2, \alpha_3, \dots \}$ which reverses the order of its intervals as depicted in Figure \ref{fig:simpleinverseorder}:
\begin{figure}[H]
    \centering
\includegraphics[width=0.6\linewidth]{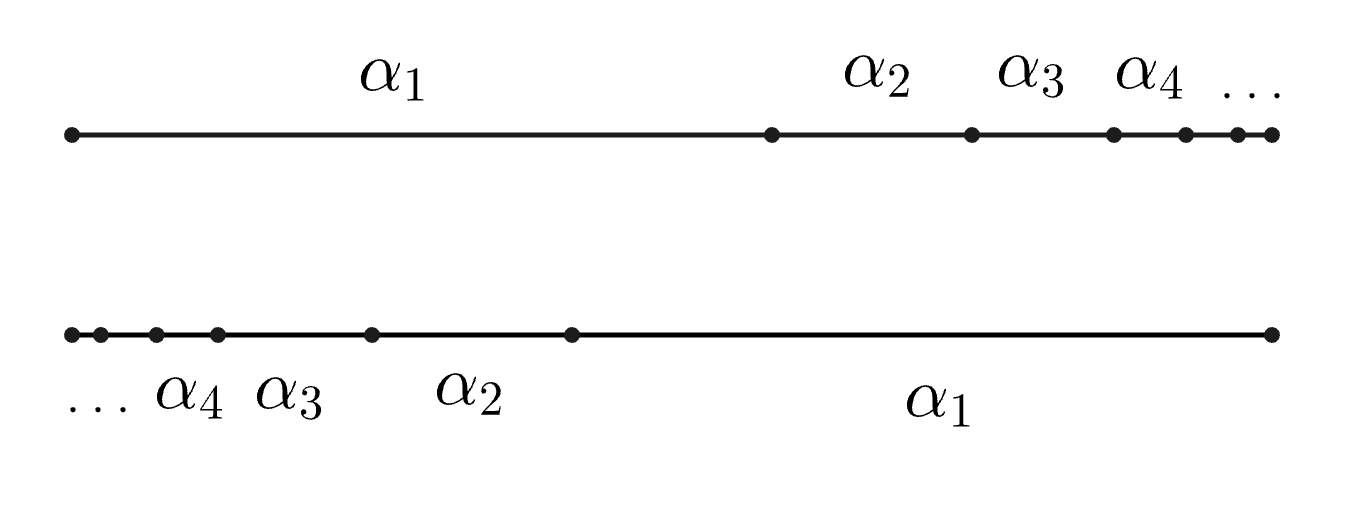}
    \caption{}\label{fig:simpleinverseorder}
\end{figure}
$T$ is a tail-reversing IET with respect to the partitions
$\mathcal{P} := \{ \mathcal{T} \}$, and $\mathcal{P}' := \{ \{\alpha_1 \}, \mathcal{T}' \}$
where $\mathcal{T} := \mathcal{A}$ and $\mathcal{T}' := \mathcal{A}\backslash\{\alpha_1\}$.
The corresponding permutations $\Pi, \Pi'$ are then of the form:
\vspace{2mm}
\begin{align*}
\Pi = \begin{pmatrix}
	\overrightarrow{\mathcal{T}}\\
	\overleftarrow{\mathcal{T}}
	\end{pmatrix}
    \quad \text{and} \quad
    \Pi' = \begin{pmatrix}
	\alpha_1 & \overrightarrow{\mathcal{T}'}\\
	\overleftarrow{\mathcal{T}'} &  \alpha_1
	\end{pmatrix}.
\end{align*}
If $\lambda$ is the length vector associated to $T$, then $(\Pi, \lambda)$ and $(\Pi', \lambda)$ are both presentations of $T$.

\subsubsection{Revealing}\label{sec:revealingprocess} As we saw from the previous example, a tail-reversing IET $T$ may be described by infinitely many permutations $\Pi$. These different permutations are related by a process which we call \textit{revealing}.

\begin{defs} Given a finite partition $\mathcal{P}$ of $\mathcal{A}$ for $T$ and a tail interval labelled by an infinite set $\mathcal{A}_n \in \mathcal{P}$, the \textit{outermost} interval of $\mathcal{A}_n$ is defined to be the rightmost (in $[0,l]$) interval of $I_{\mathcal{A}_n}^t$ if $I_{\mathcal{A}_n}^t$ is right-sided, otherwise it is the leftmost interval.
\end{defs}

\begin{defs} Let $T$ be a tail-reversing IET, let $\mathcal{P}$ be a partition on $d \geq 2$ intervals for $T$ given by
\begin{align*}
   \mathcal{A}_1 \sqcup \dots \sqcup \mathcal{A}_d.
\end{align*}
We say that the partition $\mathcal{P'}$ of $\mathcal{A}$ has been obtained from $\mathcal{P}$ by \textit{revealing a tail interval} if it is given by removing the outermost letter $\alpha$ of a tail interval $\mathcal{A}_n$ and adding it as a single element set of the partition, i.e. $\mathcal{P}'$ is the partition given by
\begin{align*}
    \mathcal{A}_1 \sqcup  \dots \sqcup \mathcal{A}_{n-1} \sqcup \mathcal{A}_n \backslash \{ \alpha \} \sqcup \{ \alpha \} \sqcup \mathcal{A}_{n+1} \sqcup \dots \sqcup \mathcal{A}_d.
\end{align*}
\end{defs}

Similarly we say that the permutation $\Pi'$ has been obtained from $\Pi$ by \textit{revealing a tail interval} if $\mathcal{P'}$ has been obtained from $\mathcal{P}$ by \textit{revealing a tail interval} and $\Pi = (\mathcal{P}, \tailtype, \pi)$ and $\Pi' = (\mathcal{P}', \tailtype', \pi')$ are both permutations for $T$. The following Lemma is immediate:

\begin{lem}
    Assume $(\Pi, \lambda)$ and $(\Pi', \lambda)$ are presentations of the same tail-reversing IET.
    Then $\Pi$ is obtained from $\Pi'$ (or vice versa) by revealing finitely many tail intervals.
\end{lem}
In this case, we say that $\Pi$ and $\Pi'$ are \textit{equivalent}.

\subsubsection{Proper and irreducible permutations} We introduce the notion of \textit{proper} and \textit{irreducible} permutations:

\begin{defs}\label{def:irreducible} We say that combinatorial data $\Pi = (\mathcal{P}, \tailtype, \pi)$ is \textit{irreducible} if for every $1 \leq n \leq d$ we have
\begin{align*}
    \pi_t^{-1}(\{1, \dots, n\}) \neq \pi_b^{-1}(\{1, \dots, n\}).
\end{align*}
\end{defs}

\begin{defs}\label{def:proper} We say that combinatorial data $\Pi = (\mathcal{P}, \tailtype, \pi)$ is \textit{proper} if at least one of the rightmost top and bottom grouped intervals, and one of the leftmost top and bottom grouped intervals, is a single interval.
\end{defs}

Note that if the rightmost (or leftmost) top and bottom grouped intervals are tail intervals with opposite directions, then, up to revealing a tail interval, $\Pi$ is equivalent to a proper permutation. The only real obstruction to a permutation being proper is if $l$ (or $0$) is an accumulation point of both top and bottom singularities, i.e. if the underlying IET is not renormalizable:

\begin{lem} Any renormalizable, tail-reversing IET $T$ without connections can be represented by a proper and irreducible permutation $\Pi$.
\end{lem}

\begin{proof} Any permutation $\Pi$ representing $T$ must be irreducible, since otherwise $T$ has a connection. Also, by the definition of renormalizability, neither $0$ nor $l$ are an accumulation point both of top and bottom singularities. Hence, up to revealing a finite number of intervals, $T$ also admits a permutation $\Pi$ which is proper.
\end{proof}

We will need the following Lemma later on in order to prove that irreducibility is preserved under the \textit{grouped Rauzy-Veech induction}, which we introduce in the coming section.

\begin{lem}
	\label{lem:remove_irred}
	Let $\pi = (\pi_t,\pi_b)$ be a finite-type permutation on a set $\mathcal P$ of $d$ intervals and let $\alpha \in \mathcal P$ with $(\pi_t(\alpha), \pi_b(\alpha)) \notin \{(1, 1), (d, d)\}$. Let $\pi'$ be obtained by restricting $\pi$ to $\mathcal P \setminus \{\alpha\}$. If $\pi'$ is irreducible, then $\pi$ is irreducible.
\end{lem}
\begin{proof}
	Assume by contraposition that $\pi$ is reducible, i.e. there exists $1 \le n < d$ such that $\pi_t^{-1}(\{1, \dots, n\}) = \pi_b^{-1}(\{1, \dots, n\}) =: \mathcal{I}$.
	We show that $\pi'$ is reducible.

	If $\alpha \notin \mathcal I$, then $\pi_t(\alpha), \pi_b(\alpha) > n$, so removing $\alpha$ leaves the first $n$ top and bottom positions unchanged and $\mathcal I$ still witnesses reducibility of $\pi'$, provided $n \le d-2$.
	If $n = d-1$, then $\mathcal I = \mathcal P \setminus \{\alpha\}$, forcing $(\pi_t(\alpha),\pi_b(\alpha)) = (d,d)$, which is excluded; hence $n \le d-2$ and $\pi'$ is reducible.

	If $\alpha \in \mathcal I$, then $\mathcal I \setminus \{\alpha\}$ consists of the first $n-1$ top and bottom positions of $\pi'$ and witnesses its reducibility, provided $n - 1 \ge 1$.
	If $n = 1$, then $\mathcal I = \{\alpha\}$, forcing $(\pi_t(\alpha),\pi_b(\alpha)) = (1,1)$, which is excluded; hence $n \ge 2$ and $\pi'$ is reducible.
\end{proof}

\subsection{Grouped Rauzy-Veech induction} The \textit{grouped Rauzy-Veech induction} for tail-reversing IET $T$ is defined in a similar way as the standard induction from Section \ref{sec:standardRVinduction}. However, we now associate to a sequence of induction steps also a sequence of permutations $\Pi_n$ on a growing number of partition elements. This number is growing, since whenever a \textit{tail interval wins}, we reveal an interval and thus increase the number of partition elements.

\subsubsection{The elementary step of the grouped Rauzy-Veech induction} We define the elementary step as follows:

\begin{defs} \label{def:inductiontail-reversing} Let $T_{(\Pi, \lambda)}$ be a renormalizable, tail-reversing IET with proper and irreducible permutation given by $\Pi = (\mathcal{P}, \tailtype, \pi)$. Let $\mathcal{A}_t, \mathcal{A}_b$ be the rightmost top and bottom grouped intervals of $\Pi$, also called the \textit{playing intervals}.
The tail-reversing IET $T_1 = (\Pi_1, \lambda_1)$ obtained after one induction step, where $\Pi_1 = (\mathcal{P}_1, \tau_1, \pi_1)$, and $\pi_1 = (\pi_{t,1}, \pi_{b,1})$, is defined as follows:

\begin{enumerate}[label=\textbf{\arabic*}.]
    \item \textbf{Top wins}. If $\inf(I_{\mathcal{A}_t}^t) < $ $\inf(I_{\mathcal{A}_b}^b)$, we say that \textit{top wins}. Then
    \begin{enumerate}[label=\alph*)]
	\item \textbf{(tail interval wins)} if $I_{\mathcal{A}_t}^t$ is a tail interval, we set $T_1 = T$ and set $\Pi_1$ to be the permutation obtained from $\Pi$ by revealing the outermost letter of $\mathcal{A}_t$ as well as setting $\lambda_1 = \lambda$.
	\begin{figure}[H]
    \centering
\includegraphics[width=0.75\linewidth]{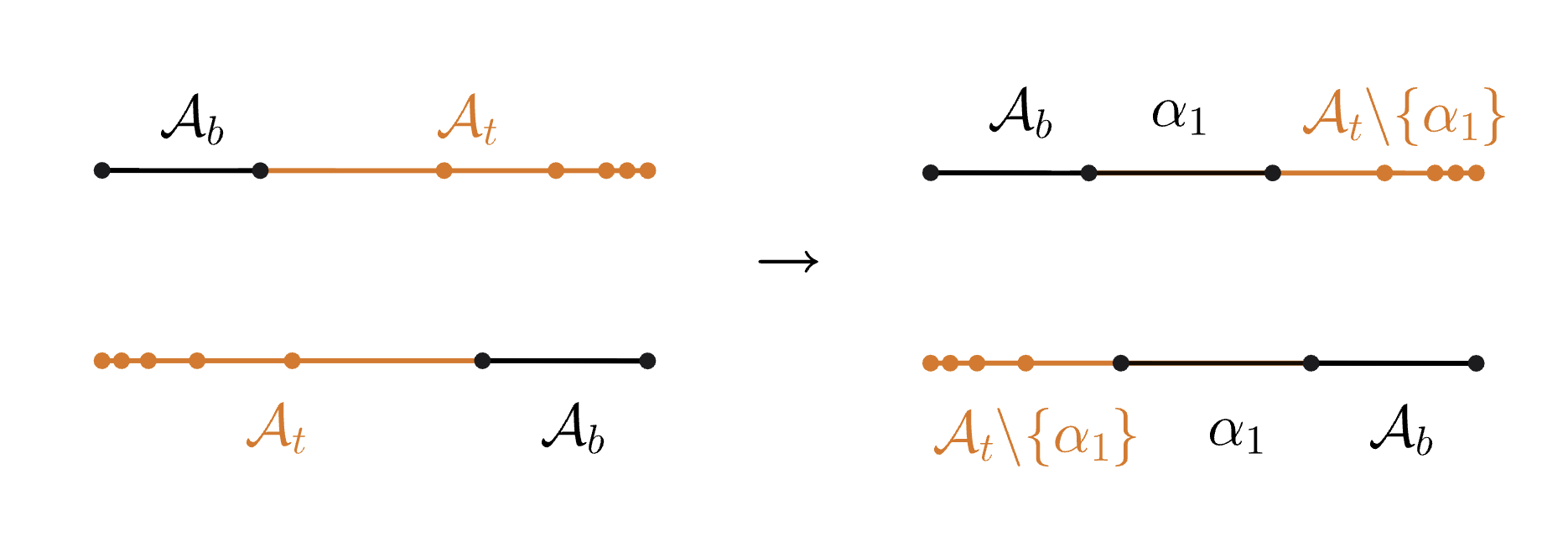}
    \caption{Case 1a) where $I_{\mathcal{A}_t}^t$ is a tail interval and wins on top. We reveal its outermost letter $\alpha_1$.}\label{fig:groupedinduction-tailwins}
\end{figure}
	\item \textbf{(single interval wins)} if $I_{\mathcal{A}_t}^t$ is not a tail interval, we define $T_1$ to be the first return map of $T$ to $[0, $ $\inf(I_{\mathcal{A}_b}^b)]$ (which is explicitly given in Definition \ref{def:RVinduction}(i)  where $w = \mathcal{A}_t$ and $\mathcal{L} = \mathcal{A}_b$).
		We change the bottom permutation $\pi_{b,1}$ as described in Definition \ref{def:RVinduction}(i) and keep all other combinatorial data the same, i.e. $\pi_{t,1} = \pi_t$, $\mathcal{P}_1 = \mathcal{P}$ and $\tau_1 = \tau$.
		The new length vector $\lambda_1$ is described in Definition \ref{def:RVinduction}(iii).

	 \begin{figure}[H]
    \centering
\includegraphics[width=0.75\linewidth]{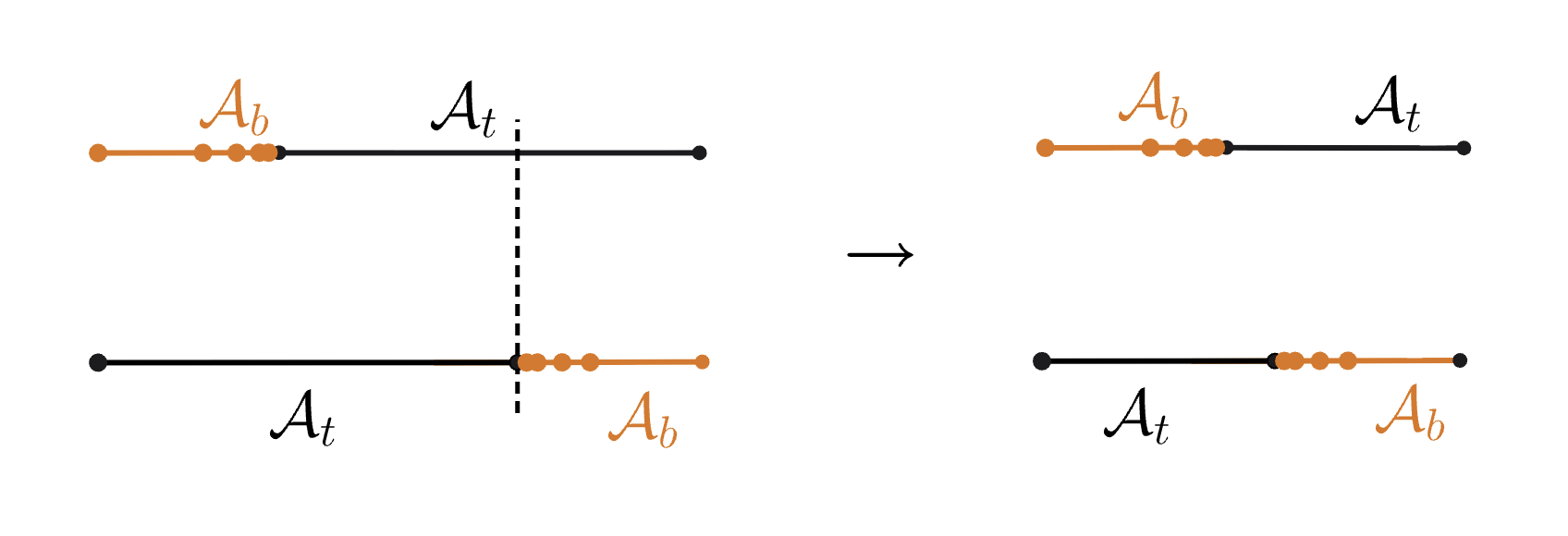}
    \caption{Case 1b) where $I_{\mathcal{A}_t}^t$ is not a tail interval and wins on top. We restrict $T$ to $[0, $ $\inf(I_{\mathcal{A}_b}^b)]$.}\label{fig:groupedinduction-singlewins}
\end{figure}
    \end{enumerate}
    \item \textbf{Bottom wins}. If $\inf(I_{\mathcal{A}_t}^t) > $ $\inf(I_{\mathcal{A}_b}^b)$, we say that \textit{bottom wins}. Then
    \begin{itemize}
    \item \textbf{(tail interval wins)} if $I_{\mathcal{A}_b}^b$ is a tail interval, we set $T_1 = T$ and set $\Pi_1$ to be the permutation obtained from $\Pi$ by revealing the outermost letter of $\mathcal{A}_b$, as well as setting $\lambda_1 = \lambda$.
    \item \textbf{(single interval wins)} if $I_{\mathcal{A}_b}^b$ is not a tail interval, we denote by $T_1$ the first return map of $T$ to $[0, $ $\inf(I_{\mathcal{A}_t}^t$)] (which is explicitly given in Definition \ref{def:RVinduction}(ii)  where $w = \mathcal{A}_b$ and $\mathcal{L} = \mathcal{A}_t$). We change the top permutation $\pi_{t,1}$ (as described in Definition \ref{def:RVinduction}(ii)) and keep all other combinatorial data the same, i.e. $\pi_{b,1} = \pi_b$, $\mathcal{P}_1 = \mathcal{P}$ and $\tau_1 = \tau$. The new length vector $\lambda_1$ is described in Definition \ref{def:RVinduction}(iii).
    \end{itemize}
\end{enumerate}

If $\inf (I_{\mathcal{A}_t}^t) = \inf (I_{\mathcal{A}_b}^b)$ then $T$ has a connection and the induction is not well defined. We also denote
\begin{itemize}
	\item the winning labels by $\mathcal{W} = \mathcal{A}_t$ (resp. $\mathcal{A}_b)$ if top (resp. bottom) wins,
	\item the losing labels by $\mathcal{L} = \mathcal{A}_t$ (resp. $\mathcal{A}_b)$ if bottom (resp. top) wins.
\end{itemize}
Contrary to the induction as defined in Section \ref{sec:standardRVinduction}, the set of winners $\mathcal{W}$ may now consist of infinitely many intervals.
\end{defs}

As in Definition \ref{def:N times renormalizable}, we call $T$ \textit{renormalizable up to time $N$ for the grouped Rauzy-Veech induction} if we can repeat this process $N$ times, and \textit{infinitely renormalizable for the grouped Rauzy-Veech induction} if we can repeat this process infinitely often. We denote by $(T_n)_{n \in \N}$ the sequence of IETs obtained in this way by the induction, where $T_n = T_{(\Pi_n, \lambda_n)}$.

\begin{rem}
	\label{rem:finiteRV} Consider combinatorial data $\Pi_1 = (\mathcal P_1, \tau_1, \pi_1)$ obtained from $\Pi = (\mathcal P, \tau, \pi)$ after one step of the grouped Rauzy-Veech induction. If the winning letter is a single interval, then $\mathcal P_1 = \mathcal P$, $\tau_1 = \tau$ and $\pi_1$ is the permutation induced from $\pi$ by the classical top (or bottom) Rauzy-Veech induction.
\end{rem}

\subsubsection{Induction and irreducibility} Note that the new permutation $\Pi_1$ is uniquely determined by $\Pi$ and whether the top or the bottom interval wins. We also claim that since $\Pi$ is proper and irreducible, also $\Pi_1$ is proper and irreducible. We prove this later in Lemma \cref{lem:invarianceofirreducibility} (together with the analogous invariance under revealing and under removing tail intervals); in particular the two playing intervals are never both tail intervals.

\subsubsection{Standard and grouped Rauzy-Veech induction} The standard Rauzy-Veech induction from Definition \ref{def:RVinduction} is related to the grouped Rauzy-Veech induction from Definition \ref{def:inductiontail-reversing} as follows:
\begin{lem}\label{lem:relationstandardgrouped} Let $T$ be an infinitely renormalizable, tail-reversing IET. Let $(T^{(k)})_{k \in \N}$  be the sequence obtained from the standard induction and let $(T_n)_{n \in \N}$ be the sequence of IETs obtained from the grouped Rauzy-Veech induction. Then there exists an infinite subsequence $(n_k)_{k \in \N}$ such that $T_{n_k} = T^{(k)}$.
\end{lem}
At all times in $\N \backslash (n_k)_{k \in \N}$ either a tail interval wins and we do not change the map $T$, or we do a sequence of steps of the grouped induction which correspond to one step of the standard induction.

\subsection{Rauzy diagrams} The combinatorial changes induced by the grouped Rauzy-Veech induction for a tail-reversing IET can be described in a \textit{Rauzy diagram}, which is a directed graph whose vertices are irreducible permutations and whose edges correspond to one induction step. Our definition of a Rauzy diagram extends the classical definition given for finite-type IETs in  \cite{Rauzy}, where now the number of intervals at each vertex may increase. Similarly to the finite-type case, we can define the \textit{combinatorial rotation number} of an IET which is a path in its Rauzy diagram and introduce the notion of an \textit{infinite-complete} path.

\subsubsection{Definition of the diagram}  Let $\mathcal{A}$ be an alphabet. For a proper and irreducible permutation $\Pi$ we have defined in the last section a new proper and irreducible permutation $\Pi'$ depending only on $\Pi$ and whether top or bottom wins, we write $\Pi' = R_t(\Pi)$ or $\Pi' =  R_b(\Pi)$ accordingly.

\begin{defs} For a given alphabet $\mathcal{A}$, we define a \textit{Rauzy diagram} to be a connected component of the following graph:
\begin{itemize}
    \item the vertices of the graph consist of proper and irreducible permutations on $\mathcal{A}$,
    \item from each vertex $\Pi$ there are exactly two outgoing edges to $R_t(\Pi)$ and $R_b(\Pi)$, which we call top or bottom edges respectively and label by $t$ or $b$.
    \item to each edge $e$ we further associate the winning and losing labels $\mathcal{W}_e, \mathcal{L}_e$ as in the Definition \ref{def:inductiontail-reversing} of the grouped Rauzy-Veech induction.
\end{itemize}
\end{defs}
For $\Pi$ proper and irreducible, we denote by $\mathcal{R}_{\Pi}$ the Rauzy diagram which contains it.
Inside a Rauzy diagram $\mathcal{R}_{\Pi}$, we can further consider infinite paths $\boldsymbol{\gamma} =  e_1e_2e_3\dots$.

\begin{defs}
	\label{def:reduced} An infinite path $\boldsymbol{\gamma}$ is called \textit{reduced} if whenever $m<n$ are such that along the edges $$e_m \dots e_{n-1}$$ the same tail interval wins and along edge $e_{n}$ this tail interval loses, then along edge $e_{n+1}$ the interval which was last revealed must win.
\end{defs}

Note that any path on $\mathcal{R}_{\Pi}$ has a corresponding reduced path, whose permutations at the vertices are equivalent up to revealing.

\subsubsection{The Rauzy diagram of the inverse permutation} A particularly nice example of a Rauzy diagram is the inverse permutation from Section \ref{sec: inverse permutation}:
\begin{figure}[H]
\includegraphics[width=1.04\linewidth]{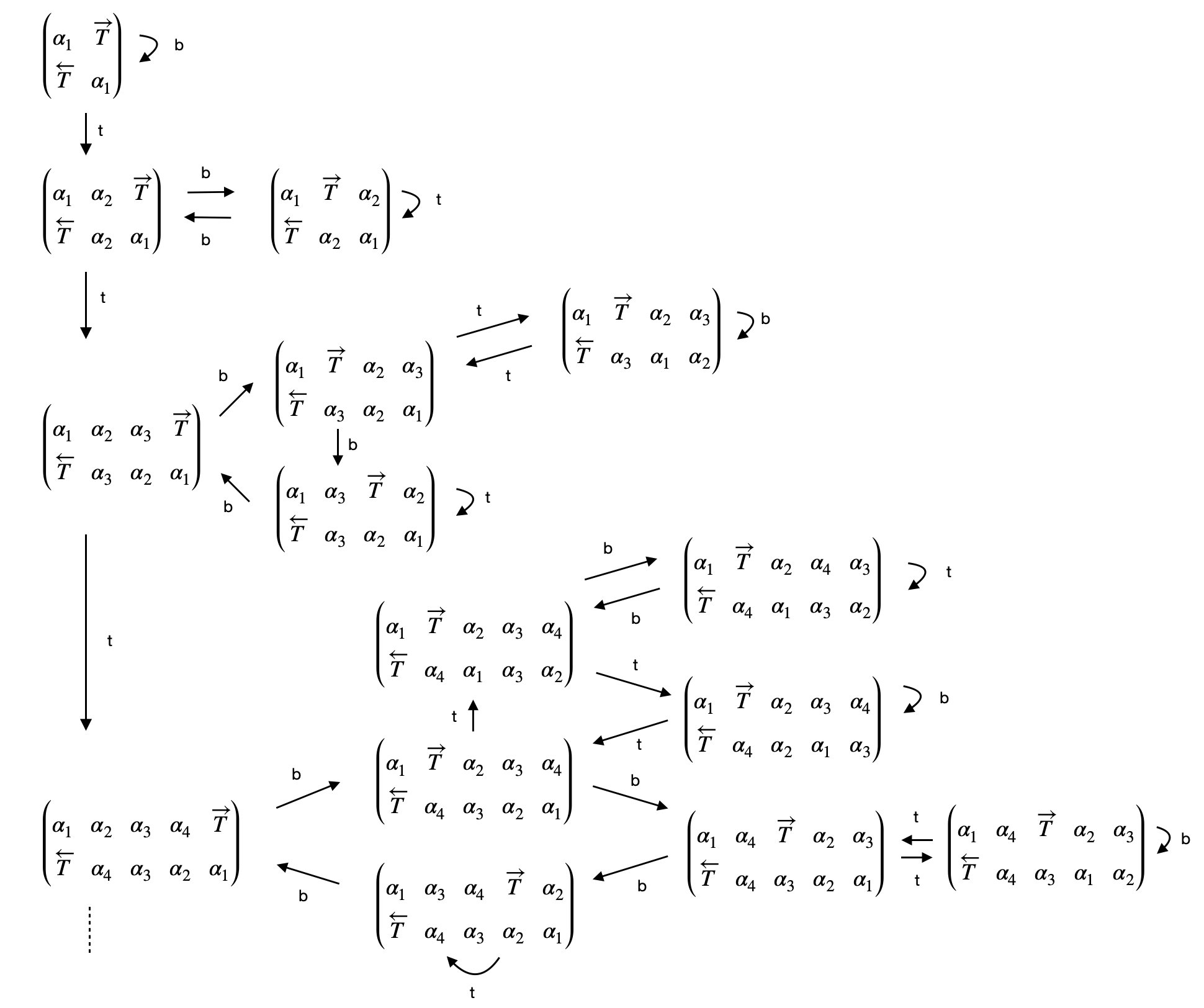}
    \caption{}\label{fig:rauzygraph}
\end{figure}
Its building blocks are finite \textit{hyperelliptic Rauzy diagrams}, meaning Rauzy diagrams with finitely many vertices which contain the finite-type permutation

\begin{align*}
    \pi^d = \begin{pmatrix}
	\alpha_1 & \dots & \alpha_d \\
	\alpha_d & \dots & \alpha_1	\end{pmatrix}
\end{align*}
\vspace{0.5mm}

for some $d \in \N$ and which have been introduced in \cite{Rauzy}. For a given $\pi^d$ we consider only \textit{half} of the hyperelliptic Rauzy diagram on $d$ intervals, i.e. we remove the ingoing and outgoing top edge of $\Pi^d$ and only consider the remaining connected component which contains $\Pi^d$. The Rauzy diagram of the inverse permutation then is obtained as an infinite collection of halves of hyperelliptic Rauzy diagram $\Pi^d$ for all $d \geq 2$, where we replace $\alpha_d$ with a tail interval in all permutations and connect $\Pi^d$ to $\Pi^{d+1}$ with a top edge.

\subsubsection{Rotation numbers}\label{sec:rotationnumber} We give the definition of the \textit{combinatorial rotation number}\footnote{The terminology "rotation number" (see \cite{Yoccoz_Course}) comes from the fact that for finite-type IETs on 2 intervals without connections, this path uniquely determines the classical rotation number.}
of an IET, which is an infinite path in its Rauzy diagram.
For $T_{(\Pi, \lambda)}$ infinitely renormalizable and tail-reversing with proper and irreducible permutation $\Pi$, let $(\Pi_n)_{n \in \N}$ be the sequence of permutations obtained from grouped Rauzy-Veech induction as in Definition \ref{def:inductiontail-reversing}.

\begin{defs}\label{def:rotationnumber} The \textit{combinatorial rotation number} of $T$ is the infinite path
	$$\boldsymbol{\gamma}(T) = e_1e_2e_3\dots$$
in $\mathcal{R}_{\Pi}$ starting at $\Pi$, where $e_1$ connects $\Pi$ to $\Pi_1$ and $e_i$ connects $\Pi_{i-1}$ to $\Pi_{i}$.  For $n, k \in \N$, we further define the finite paths
\begin{align*}
\gamma_n(T) &= e_1 \dots e_n, \\
\gamma_{{k, k+n}}(T) &= e_{k+1} \dots e_{k+n}.
\end{align*}
\end{defs}

\begin{rem}\label{rem:orbitstozero} A first application using the Rauzy graph of the inverse order permutation is to deduce that for any IET $T_{(\Pi, \lambda)}$ which is infinitely renormalizable and where $\Pi$ is the inverse order permutation, all orbits must accumulate at zero. We first claim that along $\boldsymbol{\gamma}(T)$ all letters must lose infinitely often due to the structure of the graph in Figure \ref{fig:rauzygraph}. Indeed, along $\boldsymbol{\gamma}(T)$ infinitely often the tail must win (otherwise the length of the tail intervals of $T$ must be equal to zero) and lose (otherwise the length of the $\alpha_1$ interval must be equal to zero). But by considering the central loop at each level of the Rauzy graph it then follows that all letters must lose infinitely often. In particular, by \cref{lem:heights&lengths}, the induction accumulates at zero and therefore by \cref{prop:renormalizationlimitandorbitclosures} also all orbits must accumulate at zero.
\end{rem}

\subsection{Rauzy-Veech matrices for the grouped induction} Recall the definition of the Rauzy-Veech matrices for the standard induction as stated in Section \ref{sec:standardRVmatrix}. We extend this definition to the case of the grouped Rauzy-Veech induction.

\subsubsection{Extending the definition} We define a sequence of Rauzy-Veech matrices associated to a reduced path in the Rauzy diagram as follows:

\begin{defs}\label{def:groupedmatrix} If $e$ is an edge in the Rauzy diagram with winning and losing labels $w_e$ and $ \mathcal{L}_e$, we define
\begin{equation*}
    M^e = \begin{cases}
    &Id \hspace{2mm}\text{ if $w_e$ is a tail interval,} \\
    &Id + E_{w_e, \, \mathcal{L}_e} \hspace{2mm} \text{otherwise.}
    \end{cases}
\end{equation*}
\end{defs}
We then associate to a finite reduced path  $\gamma_{n} = e_1 e_2 \dots e_n$ a product of matrices
\begin{align*}
    M^{\gamma_{n}} = M^{e_1} M^{e_2} \dots M^{e_n}.
\end{align*}
These matrices relate to the matrices from Section \ref{sec:standardRVmatrix} as follows:
\begin{lem}\label{lem:relationshipRVmatrices}For $T$ a tail-reversing, infinitely renormalizable IET with path $\gamma(T)$, consider the subsequence $(n_k)_{k \in \N}$ as in Lemma \ref{lem:relationstandardgrouped}. Then
\begin{align*}
    M^{\gamma_{n_k}} = M^{(0,k)}.
\end{align*}
\end{lem}

\subsubsection{Infinite-completeness} For an infinite path $\boldsymbol{\gamma}$ in a Rauzy diagram, we now define the notion of \textit{infinite-completeness}.
For finite-type Rauzy diagrams, this notion coincides with the classical definition of infinite-complete (see \cite{Yoccoz}, Section 7.7), which requires all entries of the finite matrices associated to $\boldsymbol{\gamma}$ to be infinitely often eventually positive.\footnote{For finite-type IETs, infinite-completeness ensures that there are no invariant subsets of intervals. It can be seen as a generalization of having an irrational rotation number in the case of the corresponding circle rotation on $d=2$ intervals.}
However, in the infinite setting, we require a slightly stronger assumption: we want all \textit{rows} of matrices associated to $\gamma$ to be infinitely often eventually positive. This is stronger than our definition of combinatorial minimality from Section \ref{sec:combinatorialminimality}, and in particular, by Theorem A, an IET with infinite-complete combinatorial rotation $\boldsymbol{\gamma}$ is minimal.

However, as we show later on in Proposition \ref{prop:nonemptyintersection}, infinite-completeness of a path $\boldsymbol{\gamma}$ also ensures the \textit{existence} of an IET with combinatorial rotation number equal to $\gamma$. We first state the definition of a \textit{complete} path for finite-type IETs.

\begin{defs}\label{def:complete} We say that a path $\gamma$ in a finite-type Rauzy diagram $\mathcal{R}_{\pi}$ is
\textit{complete} if
for all $\alpha \in \mathcal A$ there exists $n \in \N$ such that for all $\beta \in \mathcal A$ it holds $M^{\gamma_{n}}_{\alpha, \beta} > 0$.
\end{defs}

We next state the definition of an \textit{infinite-complete} path both for finite and infinite-type IETs.

\begin{defs}\label{def:infinitecomplete}
	We say that an infinite reduced path $\boldsymbol{\gamma}$ in a Rauzy diagram $\mathcal{R}_{\Pi}$ is \textit{infinite-complete} if for all $\alpha \in \mathcal A$ and all $k \in \N$ there exists $n \geq k$ such that for all $\beta \in \mathcal A$ it holds $M^{\gamma_{k,k+n}}_{\alpha, \beta} > 0$.
\end{defs}

\begin{ex} Consider the Rauzy diagram of the inverse permutation $\Pi$ in Figure \ref{fig:rauzygraph}. Then the path $\boldsymbol{\gamma}_{\infty}$ which passes through all edges at each level of the graph is infinite-complete.
\end{ex}

\begin{lem}\label{lem:infcompisminimal} Let $T$ be an IET with infinite-complete path $\boldsymbol{\gamma}$. Then $T$ is minimal.
\end{lem}

\begin{proof} Since $\boldsymbol{\gamma}$ is infinite-complete, it follows in particular that for all $\alpha, \beta \in \mathcal{A}$, for all $k \in \N$, there exists $n \geq k$ such that $M^{\gamma_{k,k+n}}_{\alpha, \beta} > 0$ and by Lemma \ref{lem:relationshipRVmatrices} the same holds for the matrices from the standard induction $M^{(k, k+n)}$, hence $T$ is combinatorially minimal and by Theorem A it follows that $T$ is also minimal.
\end{proof}

\subsubsection{Poincaré-Yoccoz Theorem for infinite-type IETs} We now show that any two IETs which share the same infinite-complete path $\boldsymbol{\gamma}$ are conjugated. In finite dimension, this statement is known as the Poincaré-Yoccoz Theorem (see Proposition 7 in \cite{Yoccoz_Course}). Our proof closely follows its finite-type analogue, in particular, we show that having the same infinite-complete combinatorial rotation number implies that orbits arrange themselves in the same order. We thus obtain a conjugacy relation on a dense subset of $[0,l]$ which by continuity we can then extend to the closure.

\begin{prop}\label{prop:conjugation} Let $T$, $\tilde{T}$ be two IETs from the interval $I$ to itself which share the same infinite-complete path $\boldsymbol{\gamma}$. Then $T$ and $\tilde{T}$ are topologically conjugate, i.e. there exists a homeomorphism $h: I \to I$ so that
$$
h \circ T = \tilde{T} \circ h \text{ on } I \backslash \mathcal{N},
$$
where $\mathcal{N}$ is the set of discontinuities of \(T\) and the points \(x\) such that \(h(x)\) is a discontinuity of \(\widetilde T\).

\end{prop}

\begin{proof}
\textbf{There is an order-preserving bijection between the floors of towers:} we recall from Section \ref{sec:towers} the definition of the tower, now with
respect to the grouped Rauzy--Veech induction, corresponding to \(\alpha\)
at time \(n\):
\[
\operatorname{Tow}_T(\alpha,n)
:=
\bigsqcup_{i=0}^{r_\alpha^n-1} T^i(I_\alpha^{t,n}),
\]
where \(r_\alpha^n\) is the first return time of \(I_\alpha^{t,n}\) to
\(I^n\). We consider the partition \(\mathcal Q_n\), up to countably many
endpoints, of \(I\) into the disjoint collection of open intervals
\[
\bigsqcup_{\alpha\in\mathcal A} \operatorname{Tow}_T(\alpha,n)
=
\bigsqcup_{\alpha\in\mathcal A}
\bigsqcup_{i=0}^{r_\alpha^n-1}
T^i(I_\alpha^{t,n}),
\]
where an element \(P\in\mathcal Q_n\) is of the form
\(T^i(I_\alpha^{t,n})\). Similarly to Section~\ref{sec:length_data}, we also define a
total order \(\leq_{T,n}\) on \(\mathcal Q_n\), describing the order of
all elements of \(\mathcal Q_n\) inside the interval \(I\), seen from left
to right.

For \(\widetilde T\), we analogously define the first return times
\(\widetilde r_\alpha^n\), the partition \(\widetilde{\mathcal Q}_n\) of
\(I\), and the total order \(\leq_{\widetilde T,n}\). Since \(T\) and
\(\widetilde T\) share the same Rauzy path, by induction\footnote{see also \cite{Yoccoz_Course}, Proposition 7 or \cite{dynamicaldecomposition}, Theorem 4.9.} it follows that,
for all \(n\in\mathbb N\),
\[
r_\alpha^n=\widetilde r_\alpha^n,
\qquad
\leq_{T,n}=\leq_{\widetilde T,n},
\]
where the two orders are identified through the symbolic labels
\((\alpha,i)\). For each \(n\in\mathbb N\), define a bijection between the elements of the
two partitions by
\[
\Phi_n\bigl(T^i(I_\alpha^{t,n})\bigr)
=
\widetilde T^i(\widetilde I_\alpha^{t,n}),
\qquad
0\leq i<r_\alpha^n.
\]
Since \(r_\alpha^n=\widetilde r_\alpha^n\) and
\(\leq_{T,n}=\leq_{\widetilde T,n}\), the map \(\Phi_n\) is order
preserving. Moreover, since \(T\) and \(\widetilde T\) have the same
Rauzy path, these bijections are compatible when refining the partitions:
if \(P'\in\mathcal Q_{n+1}\), \(P\in\mathcal Q_n\), and \(P'\subset P\),
then
\[
\Phi_{n+1}(P')\subset \Phi_n(P).
\]

\textbf{The bijection extends to an increasing homeomorphism on $I$:} let \(E_n\) be the set of left endpoints of intervals of
\(\mathcal Q_n\), together with the endpoints of \(I\), and let
\[
E=\bigcup_{n\geq 0}E_n.
\]
Similarly define \(\widetilde E_n\) and \(\widetilde E\). The compatibility
under refinement gives a well-defined order-preserving bijection
\[
h_0:E\longrightarrow \widetilde E
\]
by sending the left endpoint of an interval \(P\in\mathcal Q_n\) to the
left endpoint of the corresponding interval \(\Phi_n(P)\), and by sending
the endpoints of \(I\) to the endpoints of \(I\).

Since \(\gamma\) is infinite-complete, all letters must lose infinitely often, hence by Lemma \ref{lem:heights&lengths} the lengths
of the intervals in \(\mathcal Q_n\) and \(\widetilde{\mathcal Q}_n\) go
to zero, in particular, \(E\) and \(\widetilde E\) are dense in \(I\). Therefore
\(h_0\) extends uniquely, by order completion, to an increasing
homeomorphism
\[
h:I\longrightarrow I.
\]

\textbf{The map $h$ satisfies the conjugacy relation:} first note that, by construction of
the tower partitions, if \(e\in E\) is not a singularity of \(T\) and
\(h_0(e)\) is not a singularity of \(\widetilde T\), then \(T(e)\in E\)
and
\begin{equation}\label{eq:conj}
h_0(T(e))=\widetilde T(h_0(e)).
\end{equation}

Now define
\[
\mathcal N
=
\operatorname{Sing}(T)
\cup
h^{-1}(\operatorname{Sing}(\widetilde T)).
\]
Let \(x\in I\setminus\mathcal N\). Then \(T\) is continuous at \(x\) and
\(\widetilde T\) is continuous at \(h(x)\). Hence we may choose an open
interval \(U\ni x\) such that \(T\) is continuous on \(U\) and
\(\widetilde T\) is continuous on \(h(U)\). Since \(E\) is dense in \(I\),
choose a sequence \(e_k\in E\cap U\) such that \(e_k\to x\). Then, by
(\ref{eq:conj}),
\[
h(T(e_k))
=
h_0(T(e_k))
=
\widetilde T(h_0(e_k))
=
\widetilde T(h(e_k)).
\]
Passing to the limit, using the continuity of \(h\), \(T\) on \(U\), and
\(\widetilde T\) on \(h(U)\), we obtain
\[
h(T(x))=\widetilde T(h(x)).
\]
Therefore
\[
h\circ T=\widetilde T\circ h
\qquad
\text{on } I\setminus\mathcal N,
\]
as claimed.
\end{proof}

\section{The Simplex of Lengths} We study the \textit{simplex of lengths} $\Delta$ consisting of all strictly positive real sequences which sum to one. We introduce the \textit{nested simplices} $\hat{M}^{\gamma_n}\Delta$ obtained by applying the Rauzy-Veech matrices corresponding to $\gamma$ to the simplex $\Delta$ and explain the relationship between their infinite intersection and invariant measures. Contrary to the finite case the infinite intersection of these simplices is not guaranteed to be non-empty: in Section \ref{sec:non-emptiness} we then give a sufficient condition for it to be.

\subsection{The nested simplices} We define the \textit{simplex of lengths} $\Delta$ and equip it with the $\ell^1$-topology. We also define the nested simplices obtained by Rauzy-Veech induction and relate their infinite intersection to invariant measures of the corresponding IET.

\subsubsection{The $\ell^1$-topology}\label{sec:convexcones}  We recall the definition from Section \ref{sec:linearoperator} of the space of lengths $L$ as the convex cone
\begin{align*}
   L = \{\lambda \in \R_{>0}^\N\hspace{1mm}|\hspace{1mm}\sum_{i \in \N}\lambda_i < \infty \},
\end{align*}
and equip $L$ with the $\ell^1$-topology on $L$ generated by the $\ell^1$-norm
\begin{align*}
    \| \lambda \|_1 = \sum_{i \in \N} \lambda_i.
\end{align*}
We also define the norm restricted to the tails
\begin{align*}
    \| \lambda \|_{k,1} = \sum_{i > k} \lambda_i.
\end{align*}
We now define the \textit{simplex of lengths}
$$\Delta := \{\lambda \in \R_{>0}^\N\hspace{1mm}|\hspace{1mm}\sum_{i \in \N}\lambda_i = 1 \} \subset L,$$
and the normalization of a vector $\lambda \in L$ as $\hat{\lambda} :=  \frac{\lambda}{\| \lambda \|_1}.$
We further equip $\Delta$ with the subspace topology induced by the $\ell^1$-topology on $L$. Note that the closure of $\Delta$ in $L$ is
\begin{align*}
  \overline{\Delta} = \{\lambda \in \R_{\geq 0}^\N\hspace{1mm}|\hspace{1mm}\sum_{i \in \N}\lambda_i = 1 \}.
\end{align*}

\subsubsection{The matrix action on the simplex} We also recall from Section \ref{sec:linearoperator} that the matrices $M^{\gamma_n}$ act (continuously) on $L$ by matrix multiplication and extend this action to a continuous action
$\hat{M}^{\gamma_n}$ on $\Delta$ given by
\begin{align*}
    \hat{M}^{\gamma_n}: \Delta &\to \Delta, \\
    \lambda &\to \frac{M^{\gamma_n}\lambda}{\| M^{\gamma_n}\lambda \|_1}.
\end{align*}

We further define the simplices
\begin{align*}
\Delta^{\gamma_n} &:= \hat{M}^{\gamma_{n}} \Delta.
\end{align*}
Note that by the Definition \ref{def:groupedmatrix} of $M^{\gamma_n}$ as a growing matrix product these sets are convex up to normalizing and nested, i.e. $\Delta^{\gamma_{n+1}} \subset \Delta^{\gamma_{n}}$ for all $n \in \N$.

\subsubsection{The infinite intersection}
For $\boldsymbol{\gamma}$ an infinite path on a Rauzy diagram, we denote by $\gamma_n$ its finite prefix of length $n$ and consider the nested cones
\begin{align*}
   \Delta^{\gamma_n} \subset \dots \subset \Delta^{\gamma_1} \subset \Delta.
\end{align*}

\begin{lem}\label{lem:pathequivalence} Let $\boldsymbol{\gamma}$ be an infinite and reduced path in a Rauzy diagram starting at $\Pi$. Then the infinite intersection of nested sets
	\begin{align*} \Delta^{\boldsymbol{\gamma}, \infty} := \bigcap_{n \geq 0} \Delta^{\gamma_n}
\end{align*}
consists of all length vectors $\lambda \in \Delta$ such that $T_{(\Pi, \lambda)}$ defines an infinitely renormalizable IET whose combinatorial rotation number is equal to $\boldsymbol{\gamma}$.
\end{lem}

\begin{proof} We consider the path $\boldsymbol{\gamma} = e_1e_2e_3 \dots$ and its first edge $e_1$. We note that from Lemma \ref{prop:lengthrelation} together with Lemma \ref{lem:relationshipRVmatrices}, for $\lambda \in \Delta$ we have the relation
\begin{equation*}\label{eq:lengthrelation2}
    \lambda = M^{e_1} \cdot \lambda_{1},
\end{equation*}
where $\lambda_{1}$ is the length vector obtained from $\lambda$ after one step of the grouped induction according to $e_1$.
Consider now the set $$ \hat{M}^{e_1} \Delta.$$
If no tail interval wins along $e_1$, it is similar to the case of finite-type IETs: let $\mathcal L$ be the set of losing labels and $w$ the winning label, then any vector $\lambda \in \hat{M}^{e_1} \Delta$ satisfies $\sum_{\alpha \in \mathcal L} \lambda_{\alpha} < \lambda_w$, thus it defines a once renormalizable IET for the grouped induction whose combinatorial rotation number begins with $e_1$.

On the other hand if a tail interval wins along $e_1$, then $M^{e_1}$ is equal to identity and not all length vectors in $\hat{M}^{e_1} \Delta$ need to be renormalizable.
The ambiguity comes from the fact that revealing could be done without assuming any relations on lengths.
To solve this we can accelerate the induction until the tail loses (plus an extra step).
Let $\gamma_k = e_1 \dots e_k$ be a prefix of $\gamma$ and $k$ be the smallest integer such that the tail loses along $e_{k}$.
Let $\mathcal T_k$ be the set of labels in the tail after $k-1$ revealing steps, $\rho$ the label of the last revealed interval in the tail and $w$ the label of the winning label along $e_k$.
By \cref{def:reduced}, $w$ must lose against $\rho$ along $e_{k+1}$.
Any vector $\lambda \in \hat{M}^{\gamma_{k+1}} \Delta$ then satisfies
$\sum_{\alpha \in \mathcal T_k} \lambda_{\alpha} < \lambda_w < \sum_{\alpha \in \mathcal T_k} \lambda_{\alpha} + \lambda_\rho$.
Then any vector in $ \hat{M}^{\gamma_{k+1}} \Delta$ defines a $(k+1)$-times renormalizable IET whose path begins with $e_1 \dots e_{k+1}$.
\end{proof}

\subsubsection{Invariant measures} Given an infinite-type IET $T$, let $\mathcal{M}_{T}$ denote the set of $T$-invariant Borel probability measures. We now show in a proof analogous to the finite-type case (see Section 8.1 in \cite{Yoccoz}) that the set $\mathcal{M}_{T}$ can be identified with the infinite intersection of simplices $\Delta^{\boldsymbol{\gamma}, \infty}$.

\begin{prop}\label{prop:invmeasures} Let $T = T_{(\Pi,\lambda)}$ be an IET with proper and irreducible permutation $\Pi$ and infinite-complete rotation number $\boldsymbol{\gamma}$. Then $\Delta^{\boldsymbol{\gamma}, \infty}$ and the set of $T$-invariant Borel probability measures $\mathcal{M}_{T}$ are in one-to-one correspondence.
\end{prop}

\begin{proof} \textbf{From $\Delta^{\gamma,\infty}$ to $\mathcal{M}_{T}$}: let \(\lambda_0\in\Delta^{\gamma,\infty}\), and let
\(T_0=T_{(\Pi,\lambda_0)}\) be the corresponding IET. Since \(T\)
and \(T_0\) have the same Rauzy path, Proposition~\ref{prop:conjugation} gives a
homeomorphism \(h:I\to I\) such that
\[
h\circ T =T_0\circ h
\]
outside a countable set. Since \(T_0\) preserves
Lebesgue measure, the pullback measure
\[
F(\lambda_0)(A):=\operatorname{Leb}(h(A)),
\qquad A\in\mathcal B(I),
\]
is a \(T\)-invariant Borel probability measure. Thus $F(\lambda_0)\in\mathcal M_{T}.$

\textbf{From $\mathcal{M}_{T}$ to $\Delta^{\gamma,\infty}$:} Conversely, let \(\mu\in\mathcal M_{T}\). Since \(\gamma\) is
infinite-complete, \(T\) has no finite orbit. Hence \(\mu\) is non-atomic;
otherwise an atom would produce infinitely many atoms of the same positive
mass along its \(T\)-orbit. In particular, the endpoints of all tower
partitions have \(\mu\)-measure zero. Furthermore, since $\boldsymbol{\gamma}$ is infinite-complete, by Theorem A the map $T$ is minimal and hence $\mu$ gives strictly positive mass to nonempty open intervals. Denote by $I_{\alpha}^{t}$ the top intervals of $T$. For $\alpha\in\mathcal A, n\geq 1$ define
\begin{align*}
(\lambda_0)_{\alpha} &:= \mu(I_{\alpha}^{t}) > 0, \\
(\lambda_{0}^n)_{\alpha} &:=\mu(I_{\alpha}^{t,n}) > 0.
\end{align*}
By construction of Rauzy--Veech induction, each interval \(I_{\beta}^t\) is
partitioned, up to endpoints, into levels of towers over the intervals
\(I_{\alpha}^{t,n}\), with multiplicities \(M_{\beta\alpha}^{\gamma_n}\).
Since \(\mu\) is \(T\)-invariant, all levels in the tower over
\(I_{\alpha}^{t,n}\) have the same \(\mu\)-mass \((\lambda_0^n)_{\alpha}\).
Therefore
\[
\mu(I_{\beta}^t)
=
\sum_{\alpha\in\mathcal A}
M_{\beta\alpha}^{\gamma_n}(\lambda_{0}^n)_{\alpha},
\]
or equivalently
\[
\lambda_0=M^{\gamma_n}\lambda_0^n .
\]
Thus \(\lambda_0\in\hat M^{\gamma_n}\Delta\) for every \(n\), and so
$
\lambda_0\in\Delta^{\gamma,\infty}.$
This defines
\[
G:\mathcal M_{T}\longrightarrow\Delta^{\gamma,\infty},
\qquad
G(\mu)=\lambda_0.
\]

\textbf{\(F\) and \(G\) are inverse maps}: let
\(\lambda_0\in\Delta^{\gamma,\infty}\), and set \(\mu=F(\lambda_0)\). By the
construction of the conjugacy in Proposition~\ref{prop:conjugation}, the map \(h\) sends each initial
continuity interval \(I_{\alpha}^t\) for \(T\) to the corresponding interval
\(I_{0,\alpha}^t\) for \(T_0 = T_{(\Pi,\lambda_0)}\), up to endpoints.
Hence
\[
G(F(\lambda_0))_\alpha
=
\mu(I_{\alpha}^t)
=
\operatorname{Leb}(h(I_{\alpha}^t))
=
\operatorname{Leb}(I_{0,\alpha}^t)
=
(\lambda_{0})_{\alpha}.
\]
Therefore $G(F(\lambda_0))=\lambda_0$.

Conversely, let \(\mu\in\mathcal M_{T}\), set as before $\lambda_0=G(\mu)$ and $(\lambda_0^n)_\alpha=\mu(I_{\alpha}^{t,n})$ and let \(T_0=T_{(\Pi,\lambda_0)}\). Let \(\nu=F(\lambda_0)\), and let \(h\) be
the conjugacy used to define \(\nu\). We show that \(\nu=\mu\).

As above, for all $n \in \N$,
\[
\lambda_0=M^{\gamma_n}\lambda_0^n.
\]
Let \(\ell^n\) be the length vector of the map \((T_0)_n\) obtained from $T_0$ by $n$ steps of the grouped induction. Since
\(\lambda_0\in\Delta^{\gamma,\infty}\), the IET \(T_0\) has Rauzy path \(\gamma\).
Hence
\[
\lambda_0=M^{\gamma_n}\ell^n.
\]
Since \(M^{\gamma_n}\) is invertible, we get $\ell^n=\lambda_0^n$ for all $n \in \N$. Now let $P=T^i(I_{\alpha}^{t,n})$
for some \(\alpha\in\mathcal A\) and \(0\le i<r_\alpha^n\). By
\(T\)-invariance of \(\mu\), and since endpoints have \(\mu\)-measure zero,
\[
\mu(P)=\mu(I_{\alpha}^{t,n})=(\lambda_0^n)_\alpha.
\]
On the other hand, by Proposition \ref{prop:conjugation} and the construction of \(h\),
\(h(P)\) is the tower level corresponding to the same pair \((\alpha,i)\) for
\(T_0\), up to endpoints. Therefore
\[
\nu(P)=\operatorname{Leb}(h(P))=\ell^n_\alpha=(\lambda_0^n)_\alpha.
\]
Thus \(\mu\) and \(\nu\) agree on every atom of every tower partition of $T$. Since the tower partitions refine each other and the endpoints of these partitions are dense in
\(I\), the atoms generate \(\mathcal B(I)\) modulo endpoints. As endpoints
have both \(\mu\)- and \(\nu\)-measure zero, we obtain $\mu=\nu$ and therefore $F(G(\mu))=\mu.$
\end{proof}

\subsection{Non-emptiness of the infinite intersection}\label{sec:non-emptiness} We now prove that if $\boldsymbol{\gamma}$ is infinite-complete, then the set $\Delta^{\boldsymbol{\gamma}, \infty}$ is not empty.

\subsubsection{Mass escaping to infinity} A standard approach to do so is to construct a sequence $(\lambda^n)_{n \in \N}$ where $\lambda^n \in \Delta^{\gamma_n}$. Then if $\lambda^n$ has a convergent subsequence whose limit lies in $\Delta$, it follows that the infinite intersection of the sets $\Delta^{\gamma_n}$ is non-empty.

In order to find a convergent subsequence, we want the closed simplex $\overline{\Delta}$ to be compact. However, a key difficulty that arises when working with infinite dimensions is that this is not the case: the closed set
\begin{align*}
   \overline{\Delta} := \{\lambda \in \R_{\geq 0}^\N\hspace{1mm}|\hspace{1mm}\sum_{i \in \N}\lambda_i = 1 \}
\end{align*}
is not compact with respect to the $\ell^1$-topology. For example, the sequence
\begin{align*}
    b_1 = (1,0,0, \dots), \qquad b_2 = (0,1,0, \dots), \qquad b_3 = (0,0,1, \dots),
\end{align*}
does not have a convergent subsequence, since for all $i,j$ we have $\|b_i - b_j \|_1 =2$. This phenomenon is due to the fact that in infinite dimension, we can "move" mass infinitely far into the tail. However, the following Proposition tells us that if we exclude the scenario in which the sequence $(\lambda^n)_{n \in \N}$ "moves" mass to infinity, we can extract an $\ell^1$-convergent subsequence.

\begin{prop}\label{prop:l1convergence} Let $(\lambda^n)_{n\in \N}\subset \overline{\Delta}$. Assume that for every $\epsilon>0$ there exist integers $k,n_0 \in \N$ such that
\[
\| \lambda^n \|_{k,1} = \sum_{i>k} \lambda_i^n<\epsilon \qquad\text{for all } n\ge n_0.
\]
Then there exists a subsequence $(\lambda^{n_j})_{j\in \N}$ and $\lambda \in \overline{\Delta}$
such that
\begin{align*}
    \lambda^{n_j} \to \lambda \in \overline{\Delta} \text{ in } \ell^1.
\end{align*}
\end{prop}

\begin{proof} For each \(i\), the sequence \((\lambda_i^n)_{n\ge1}\) is contained in the compact interval \([0,1]\), and therefore admits a convergent subsequence. By successive extraction\footnote{meaning that one first extracts a subsequence along which the first coordinate converges, then from that subsequence a further one along which the second coordinate converges, and so on. The final extraction is then obtained through a Cantor diagonal argument.}, one
constructs a subsequence \((\lambda^{n_j})_{j\ge1}\) such that for every \(i\ge1\),
\[
\lambda_i^{\,n_j}\longrightarrow \lambda_i
\]
for some $\lambda_i \in [0,1]$. Since all terms are nonnegative, Fatou's lemma gives
\[
\sum_{i=1}^\infty \lambda_i
\le
\liminf_{j\to\infty}\sum_{i=1}^\infty \lambda_i^{\,n_j}
=
1.
\]

We now show that \(\lambda^{n_j}\to \lambda \) in \(\ell^1\). Fix \(\epsilon>0\). By assumption, there exist integers \(k,j_0 \in \N \) such that
\[
\sum_{i>k} \lambda_i^{\,n_j} < \epsilon/3
\qquad\text{for all } j\ge j_0.
\]
Applying Fatou's lemma again to the tails, we obtain
\[
\sum_{i>k} \lambda_i
\le
\liminf_{j\to\infty}\sum_{i>k} \lambda_i^{\,n_j}
\le
\epsilon/3.
\]
On the other hand, since \(\lambda_i^{\,n_j}\to \lambda_i\) for each \(1\le i\le k\), and there are only finitely many such indices, there exists \(j_1\) such that
\[
\sum_{i=1}^k |\lambda_i^{\,n_j}-\lambda_i| < \epsilon/3
\qquad\text{for all } j\ge j_1.
\]
Therefore, for \(j\ge \max\{j_0,j_1\}\),
\[
\|\lambda^{n_j}-\lambda\|_1
\le
\sum_{i=1}^k |\lambda_i^{\,n_j}-\lambda_i|
+
\sum_{i>k} \lambda_i^{\,n_j}
+
\sum_{i>k} \lambda_i
<
\epsilon.
\]
Thus \(\lambda^{n_j}\to \lambda\) in \(\ell^1\). Finally, since the \(\ell^1\)-norm is continuous,
\[
\|\lambda\|_1
=
\lim_{j\to\infty} \|\lambda^{n_j}\|_1
=
1.
\]
\end{proof}

\subsubsection{Non-emptiness} We now prove our main Proposition \ref{prop:nonemptyintersection} of this section, using the previous Proposition \ref{prop:l1convergence} in order to extract $\ell^1$-convergent subsequences.

\begin{prop}\label{prop:nonemptyintersection} Let $\boldsymbol{\gamma}$ be an infinite-complete path in the Rauzy diagram. Then the set $\Delta^{\boldsymbol{\gamma}, \infty} = \bigcap_{n \geq 0} \Delta^{\gamma_n}$ is non-empty.
\end{prop}

\begin{proof} Define
\begin{align*}
    b_{i}^{n} := \hat{M}^{\gamma_n}b_i \in \overline{\Delta},
\end{align*}
where $b_i$ is the $i-$th basis vector whose only nonzero coefficient is the $i$-th entry, on which it is equal to $1$.

\textbf{There exists $b_i^{\infty} \in \overline{\Delta}$ such that $b_{i}^{n} \to b_i^{\infty}$ in $\ell^1$} (up to taking a subsequence):
let $N \in \N$, by infinite-completeness, there exists $n \in \N$ such that $M^{\gamma_{n}}$ is strictly positive on the first $N$ rows.
By Lemma \ref{lem:formrvmatrix} there also exists $k \in \N$ so that $M^{\gamma_{n}}$ acts trivially on rows larger than $k$. Then we have
    $\|M^{\gamma_n}b_{i}^{} \|_1 \ge N \| M^{\gamma_n}b_{i}^{} \|_{k,1}$.
    This inequality is linear and preserved by non-negative linear combination of columns (since $\|b_i+b_j\|_1 = \|b_i\|_1 + \|b_j\|_1$) thus it is preserved by the induction. In particular, for all $N\in \N$ there exists $n_0,k \in \N$ so that for all $n \geq n_0$
    $$
     \frac{\|M^{\gamma_n}b_{i}^{}  \|_{k,1}}{\| M^{\gamma_n}b_{i}^{} \|_1} \le \frac{1}{N}
    $$
    and hence, given $\epsilon > 0$, choosing $N > 1/\epsilon$ yields $\| b_i^{n} \|_{k,1} < \epsilon$ for all $n \geq n_0$.
Then the claim follows from Proposition \ref{prop:l1convergence}. To simplify notation, we do not change the indices when passing to subsequences. All the arguments which follow also analogously work for subsequences.

\textbf{All entries of $b_i^{\infty}$ are strictly positive i.e. $b_i^{\infty} \in \Delta$ :}
note that
\begin{equation}\label{eq:nested}
    \hat{M}^{\gamma_{n+1}}\overline{\Delta} \subseteq \hat{M}^{\gamma_n}\overline{\Delta}
\end{equation}
for all $n \in \N$. Let $\alpha \in \mathcal A$, by infinite-completeness, there exists $n \in \N$ such that all coefficients of the row corresponding to $\alpha$ of $M^{\gamma_{n}}$ are positive.
Since the matrix has integer coefficients, no coefficient in the row is less than 1.
Let $\mu = M^{\gamma_n} \lambda$ where $\lambda \in \overline{\Delta}$, then $1 \le \mu_\alpha$.
Moreover, by Lemma \ref{lem:formrvmatrix} the matrix $M^{\gamma_n}$ acts as identity on all but a finite number of rows thus there is a constant $C$ such that $\| \mu \|_1 \le C$  depending only on the matrix.
In particular, $\frac{\mu_\alpha}{\| \mu \|_1} \ge \frac{1}{C}$ for all $\mu \in M^{\gamma_n} \overline{\Delta}$, meaning that for all $\hat{\mu} \in \hat{M}^{\gamma_n}\overline{\Delta}$ we have $\hat{\mu}_{\alpha} \geq \frac{1}{C}$. By (\ref{eq:nested}) this holds also for all $\hat{\mu} \in \hat{M}^{\gamma_k} \overline{\Delta}$ where $k \geq n$. Since  $b_i^{k} \in \hat{M}^{\gamma_k} \overline{\Delta}$, this bound is also true for all $b_i^{k}$ where $k \geq n$ and hence also for the limit $b_i^{\infty}$, in particular the $\alpha$-coordinate of $b_i^{\infty}$ is strictly positive. By infinite-completeness we can repeat this for all $\alpha \in \mathcal{A}$, and it follows that $b_i^{\infty}$ has strictly positive coordinates and hence  $b_i^{\infty} \in \Delta$.

\textbf{ For all $n \in \N$, it holds $b_i^{\infty} \in \Delta^{\gamma_n}$}: in particular, $b_i^{\infty} \in \Delta^{\boldsymbol{\gamma}, \infty}$ and hence $\Delta^{\boldsymbol{\gamma}, \infty}$ is not empty.
Indeed, consider $d_i^{n} = \hat{M}^{\gamma_{1,n}} b_i$ (where $\gamma_{1,n}$ is the path obtained from $\gamma$ by removing the first edge, as in Definition \ref{def:rotationnumber}).
By the same argument as in the two previous paragraphs, the vectors $d_i^{n}$ converge in $\ell^1$ to some vector $d_i^\infty$ (up to taking a subsequence), and the vector $d_i^\infty$ is again in $\Delta$, where here we are using that in the definition of infinite-completeness, all rows become eventually positive also for $M^{\gamma_{1,n}}$. Since now the action of $\hat{M}^{\gamma_1}$ on $\Delta$ is continuous with respect to $\ell^1$ topology, $b_i^\infty = \hat{M}^{\gamma_1} d_i^\infty$.
Thus  $b_i^\infty \in \Delta^{\gamma_1}$ and by induction $b_i^{\infty} \in \Delta^{\gamma_n}$ for all $n \in \N$, in particular $b_i^\infty \in \Delta^{\boldsymbol{\gamma}, \infty}$ and hence $\Delta^{\boldsymbol{\gamma}, \infty}$ is non-empty.
\end{proof}

\section{Topological genericity of uniquely ergodic IETs} After having established that the infinite intersection $\Delta^{\gamma, \infty}$ is non-empty if $\gamma$ is infinite-complete, we now want to find infinite-complete paths for which this intersection is equal to a point. In particular, in this section we prove Theorem B, which states that the set of $\lambda \in \Delta$ which define such uniquely ergodic paths contains a dense $G_\delta$ set with respect to the $\ell^1$-norm. The main tool to prove density is Proposition \ref{prop:ueexists}, which states that starting from any proper and strongly irreducible permutation $\Pi$ there exists a uniquely ergodic path.  We obtain this Proposition through a combinatorial analysis of infinite-type Rauzy diagrams as well as their relationship with their finite-type counterparts obtained by \textit{removing} tail intervals (see Definition \ref{def:removing}).

\subsection{Combinatorial analysis of infinite-type Rauzy diagrams}
Starting from any permutation $\Pi$, we construct a path along which no tail intervals win and all tail intervals lose at most once and show as an application that in any Rauzy diagram there exists an infinite-complete path. In our construction we use the following classical result for finite-type Rauzy diagrams (see \cite{Rauzy}, resp. \cite{Yoccoz_Course}) where we recall Definition \ref{def:complete} of complete paths in finite-type Rauzy diagrams.

\begin{prop}\label{prop:yoccoz} Let $\pi$ be an irreducible finite-type permutation.
	Its connected component in the finite-type Rauzy diagram $\mathcal{R}_{\pi}$ is strongly connected, in particular there exists a complete finite path in $\mathcal{R}_{\pi}$.
\end{prop}

\subsubsection{Modifying Permutations} In order to use existing results for finite-type diagrams, we define two different ways of modifying a permutation: the first one consists of replacing grouped intervals by single intervals, the second one consists of entirely removing a grouped interval. Let $\Pi$ be a permutation.
\begin{rem} Given a permutation $\Pi = (\mathcal P, \tau, \pi)$ on $d$ groups, the finite-type permutation $\finiteperm$ can be thought of as the permutation for a finite-type IET on $d$ intervals where we replace each tail interval $\mathcal T$ of $\Pi$ by a single interval.
\end{rem}

\begin{ex} Let $\Pi$ be as below where $\mathcal{T} = \{\alpha_3, \alpha_4, \dots \}$, then
\vspace{1mm}
\begin{align*}
\Pi = \begin{pmatrix}
	\alpha_1 & \alpha_2 & \overrightarrow{\mathcal{T}}\\
	\overleftarrow{\mathcal{T}} & \alpha_1 &  \alpha_2
	\end{pmatrix}, \qquad
\finiteperm = \begin{pmatrix}
	\alpha_1 & \alpha_2 & \mathcal T\\
	  \mathcal T & \alpha_1 &  \alpha_2
	\end{pmatrix}.
\end{align*}
\end{ex}
\vspace{1mm}
\begin{defs}\label{def:removing} Given a permutation $\Pi = (\mathcal P, \tau, \pi)$, for a subset $\mathcal{C} \subseteq \mathcal P$ of groups, we denote by $\pi_{{\mathcal{C}}}$ the finite-type permutation obtained from $\pi$ by \textit{removing} the groups in $\mathcal C$, that is, by restricting the (finite) top and bottom permutations $\pi_t, \pi_b$ to $\mathcal P \setminus \mathcal C$.
	If $\mathcal C$ is the set of labels of tail intervals in $\mathcal P$ we denote $\pi_{\mathcal C}$ by $\pi_*$.
\end{defs}
This will only be used to remove labels of tail intervals in what follows.
\begin{ex}\label{ex:notirreducible} Let $\Pi$ be as below, then $\pi_*$ is the following permutation:
\vspace{2mm}
\begin{align*}
\Pi = \begin{pmatrix}
	\alpha_1 & \alpha_2 & \overrightarrow{\mathcal{T}}\\
	\overleftarrow{\mathcal{T}} & \alpha_1 &  \alpha_2
	\end{pmatrix}, \qquad
\pi_* = \begin{pmatrix}
	\alpha_1 & \alpha_2 \\
	  \alpha_1 &  \alpha_2
	\end{pmatrix}.
\end{align*}
\end{ex}
\vspace{2mm}
Note that in this example $\pi_*$ is not irreducible while $\pi$ is. We say that a permutation $\Pi$ is strongly irreducible if after removing all tail intervals, it remains irreducible.

\begin{defs}\label{def:stronglyirreducible}
	We say $\Pi$ is \textit{strongly irreducible} if $\pi_*$ is irreducible.
\end{defs}

\subsubsection{Preserving irreducibility} The following lemma collects, in a single statement, the invariance properties of proper, irreducible and strongly irreducible permutations that we use throughout. They all reduce to \cref{lem:remove_irred}.

\begin{lem}\label{lem:invarianceofirreducibility} Let $\Pi = (\mathcal P, \tau, \pi)$ be a proper permutation.
\begin{enumerate}[label=(\roman*)]
	\item If $\Pi$ is irreducible, any permutation obtained from $\Pi$ by revealing a tail interval is again proper and irreducible, and so are the images $R_t(\Pi)$ and $R_b(\Pi)$ under one step of the grouped Rauzy--Veech induction.
		\label{lem:irred}
	\item If $\Pi$ is strongly irreducible, then revealing a tail interval preserves strong irreducibility. Also applying the induction steps $R_t, R_b$, preserves strong irreducibility, and for every subset $\mathcal C$ of tail intervals of $\Pi$ the permutation $\pi_{\mathcal C}$ is irreducible.\label{lem:stron_irred}
\end{enumerate}
\end{lem}
\begin{proof}
\ref{lem:irred} We first show invariance of properness: since $\Pi$ is proper, no tail interval occupies the position $(1,1)$ or $(d,d)$ in $\pi$.
This remains true after revealing since the single interval guaranteed by properness at each end is never removed. This also remains true when applying a step of Rauzy induction where no tail wins, since the winning single interval is preserved.

We now show invariance of irreducibility: let $\Pi'$ be obtained from $\Pi$ by revealing the outermost interval $\alpha$ of a tail.
Removing $\alpha$ from $\pi'$ gives back $\pi$, and $\alpha$ is not at position $(1,1)$ or $(d,d)$ of $\pi'$; hence $\pi'$ is irreducible by \cref{lem:remove_irred}.
Consider now $R_t(\Pi)$ (resp.\ $R_b(\Pi)$). If a tail interval wins, $R_t(\Pi)$ (resp.\ $R_b(\Pi)$) is a revealing of $\Pi$ and hence irreducible by above argument.
When a single interval wins, $R_t(\Pi) = R_t(\pi)$ (resp.\ $R_b$), where $R_t(\pi)$ denotes the classical Rauzy-Veech induction (see \cref{rem:finiteRV}), which by a classical fact preserves irreducibility.

\ref{lem:stron_irred} Assume in addition that $\pi_*$ is irreducible. If a single interval wins, $(R_t(\Pi))_* = R_t(\pi_*)$ (resp.\ $R_b(\pi_*)$) by \cref{rem:finiteRV}, which as before preserves irreducibility;
and when a tail wins the step is a revealing step, which adds to $\pi_*$ a single interval whose position is neither $(1,1)$ nor $(d,d)$ by properness of $\Pi$, so that $\pi_*$ remains irreducible by \cref{lem:remove_irred}.
Thus strong irreducibility is preserved.

Finally, for a subset $\mathcal C$ of tail intervals, we start from $(\pi_{\mathcal C})_* = \pi_*$ and add back the tail intervals of $\pi_{\mathcal C}$ one at a time.
The added intervals cannot be both at $(1,1)$ or $(d,d)$ since otherwise two tail intervals are at one extremity of the permutation and this is preserved when adding intervals to recover $\pi$, contradicting properness of $\Pi$.
Then \cref{lem:remove_irred} shows that $\pi_{\mathcal C}$ is irreducible.
\end{proof}

\begin{rem}
	This Lemma shows in particular (taking $\mathcal C = \varnothing$ in \ref{lem:stron_irred}) that strong irreducibility implies irreducibility for proper permutations.
\end{rem}

\begin{rem}\label{rem:preservingirreducibility} By \cref{lem:invarianceofirreducibility}, when $\Pi$ is proper and strongly irreducible, every permutation in the Rauzy diagram $\mathcal{R}_{\Pi}$ is proper and strongly irreducible, in particular the two playing (rightmost) intervals are never both tails. Furthermore, removing any subset of tail intervals from $\pi$ yields, as in \ref{lem:stron_irred},  an irreducible permutation
\end{rem}

The following lemma shows that, up to revealing finitely many intervals, one may always assume a proper and irreducible permutation to be strongly irreducible. This will allow us to state the results of this section for strongly irreducible permutations while keeping our main theorems for arbitrary proper and irreducible ones.

\begin{lem}\label{lem:stronglyirreduciblepresentation}
	Every proper and irreducible permutation $\Pi$ becomes proper and strongly irreducible after revealing finitely many tail intervals. Moreover, revealing a tail interval leaves the underlying IET unchanged: if $\Pi'$ is obtained from $\Pi$ in this way, then $T_{(\Pi,\lambda)} = T_{(\Pi',\lambda)}$ for every $\lambda \in \Delta$.
\end{lem}

\begin{proof}
	Revealing a tail interval is the elementary step of the grouped induction in which a tail wins, so it leaves both the IET and the length vector unchanged (\cref{def:inductiontail-reversing}, case 1a) and, by \cref{lem:invarianceofirreducibility}\ref{lem:irred}, it preserves properness and irreducibility. It therefore suffices to show that after finitely many revealings the finite-type permutation obtained by removing all remaining tail intervals is irreducible.

    We denote by $\Pi'$ the permutation obtained by revealing one interval of each tail of $\Pi$. Then after removing all tail intervals, $\pi'_* = \pi$ and hence $\pi'_*$ is irreducible.
\end{proof}

\subsubsection{Lifting paths} Let $\Pi = (\mathcal P, \tau, \pi)$ be proper and irreducible. We now \textit{lift} paths in $\mathcal{R}_{\finiteperm}$ and $\mathcal{R}_{\pi_{{\mathcal{C}}}}$ to $\mathcal{R}_{\Pi}$ as described in the following two Lemmas.

\begin{lem}\label{lem:liftinfinite}
    A path in $\mathcal R_{\finiteperm}$ starting at $\finiteperm$ with winning labels $\alpha_i$ and along which no tail wins defines a path in $\mathcal R_{\Pi}$ starting at $\Pi$ with winning labels $\mathcal{C}_i = \{ \alpha_i \}$.
\end{lem}

\begin{lem}\label{lem:liftfinite} Let $\mathcal C \subset \mathcal P$, a path in $\mathcal R_{\pi_{{\mathcal C}}}$ starting at $\pi_{\mathcal{C}}$  defines a path in $\mathcal R_{\pi}$ starting at $\pi$ with the same winners and losers except for when the removed intervals in $\mathcal{C}$ play, in which case they lose.
\end{lem}

\begin{proof} Consider the permutation obtained from $\pi$ by applying one step of the induction and then removing $\mathcal{C}$.
	If $\mathcal{C}$ does not play, it is the same as the one obtained by applying the same step of the induction to $\pi_{{\mathcal{C}}}$, while if a letter in $\mathcal{C}$ loses, it is the same as $\pi_{\mathcal{C}}$.
	Hence a path below defines a path above with the same winners and losers, except for when $\mathcal{C}$ plays, in which case the corresponding interval loses.
\end{proof}

\subsubsection{Winning and losing paths}\label{sec:winningandlosingpaths} We now construct two paths starting from any proper and irreducible permutation, one along which all tail intervals win and one along which all tail intervals lose. These paths will then be used in this section as well as the next Section \ref{sec:thmBproof} in order to construct infinite-complete and uniquely ergodic paths starting from any given permutation.

Let $\Pi$ be a proper and strongly irreducible permutation in $\mathcal{R}_{\Pi}$.

\begin{prop}\label{prop:alltailswin} There exists a permutation $\Pi'$ in $\mathcal{R}_{\Pi}$ and a finite path $\gamma$ from $\Pi$ to $\Pi'$ along which each tail interval wins exactly once.
\end{prop}

\begin{proof}  Consider the proper and strongly irreducible permutation $\Pi = (\mathcal P, \tau, \pi)$ and let $\finiteperm$ be its finite grouped permutation which is also irreducible.
Using  the connectedness and existence of a complete path in $\mathcal{R}_{\finiteperm}$ by Proposition \ref{prop:yoccoz},  together with Lemma \ref{lem:liftinfinite} one can construct a finite path in $\mathcal{R}_{\Pi}$ from $\Pi$ to a permutation $\Pi^0$ along which no tail intervals win, so that from $\Pi^0$ there is an outgoing edge to a permutation $\tilde{\Pi}^0$ along which a tail, call it $\mathcal{T}_1$, wins.
Note that both $\Pi^0$ and $\tilde{\Pi}^0$ are again proper and strongly irreducible by \cref{lem:invarianceofirreducibility}.
Consider the permutation $\sigma_1 := (\tilde{\pi}^0)_{\mathcal T_1}$ obtained by removing $\mathcal T_1$ from $\tilde{\pi}^0$, the finite permutation of $\tilde \Pi^0$.
By \cref{lem:invarianceofirreducibility}, this permutation is irreducible.
Again using Proposition \ref{prop:yoccoz} for $\mathcal{R}_{\sigma_1}$ there exists a finite path in $\mathcal{R}_{\sigma_1}$ starting at $\sigma_1$ along which no former tail interval wins before the final edge, and a former tail $\mathcal{T}_2$ plays at the final vertex.
By Lemma \ref{lem:liftfinite} together with Lemma \ref{lem:liftinfinite} this path can be lifted to a path in $\mathcal{R}_{\Pi}$ starting at $\tilde{\Pi}^0$ along which $\mathcal{T}_1$ must always lose.
Let $\Pi^1$ be its ending vertex in $\mathcal{R}_{\Pi}$.
By assumption there is an edge along which the tail $\mathcal T_2$ wins, let $\tilde{\Pi}^1$ be the vertex it points to.
Since $\tilde{\Pi}^1$ is again proper and strongly irreducible, we can now consider the finite permutation obtained after removing $\{\mathcal{T}_1, \mathcal{T}_2 \}$ from $\tilde{\pi}^1$ and continue and repeat the same procedure finitely often and concatenate all of the obtained paths to obtain a path along which all tail intervals win exactly once.
\end{proof}

\begin{prop}\label{prop:alltailslose} There exists a permutation $\Pi'$ in $\mathcal{R}_{\Pi}$ and a finite path $\gamma$ from $\Pi$ to $\Pi'$ along which no tail intervals win and each tail interval loses at least once.
\end{prop}

\begin{proof} Choose a tail interval $\mathcal T_1$ and consider the permutation $\pi_{{\{\mathcal T_2, \dots \mathcal T_n \}}}$ obtained by removing all tail intervals from $\Pi$ except for $\mathcal{T}_1$.
	By \cref{lem:invarianceofirreducibility}, this permutation is irreducible.
	As in the previous proof we construct a finite path in $\mathcal R_{\Pi}$ along which no tail interval wins and along whose last vertex $\mathcal T_1$ plays.
	Make the letter $\mathcal T_1$ lose.
	Repeat this procedure with the new permutation by removing all tail intervals except for $\mathcal T_2$, the resulting permutation is proper and strongly irreducible by \cref{lem:invarianceofirreducibility}, and then again until we remove all tail intervals except for $\mathcal{T}_n$.
	Concatenating the obtained paths we obtain a path $\gamma$ from $\Pi$ to a permutation $\Pi'$ with the property that along $\gamma$ no tail intervals win but all tail intervals lose at least once.
\end{proof}

As an application of Proposition \ref{prop:alltailswin} and Proposition \ref{prop:alltailslose} we now prove that starting from any permutation $\Pi$ there exists an infinite-complete path in its Rauzy diagram. Also using Proposition \ref{prop:alltailslose} we will prove a stronger version of this statement in the next Section \ref{sec:thmBproof}, as well as in Section \ref{sec:matrixcriterion}, namely that starting from any permutation there exists a uniquely ergodic path in its Rauzy diagram.

\begin{prop}\label{prop:pathinfcomplete} In any Rauzy diagram $\mathcal{R}$, starting from any proper and strongly irreducible permutation $\Pi \in \mathcal{R}$, there exists an infinite-complete path.
\end{prop}

\begin{proof} Consider the irreducible finite-type permutation $\pi_*$ obtained by removing \textit{all} tail intervals.
By Proposition \ref{prop:yoccoz} it follows that there exists a finite complete path in its Rauzy diagram.
By Lemma \ref{lem:liftfinite} together with Lemma \ref{lem:liftfinite} this path can be lifted to a path $\gamma_1$ from $\Pi$ to a proper and strongly irreducible permutation $\Pi^1$ along which any tail interval which plays loses.
The corresponding Rauzy-Veech matrix $M^{\gamma_1}$ is strictly positive on all coordinates $(\alpha,\beta)$ where $\alpha$ and $\beta$ correspond to the labels of single intervals.

However, we want entire rows to be positive. By Proposition \ref{prop:alltailslose} we can choose a path $\gamma_2$ from $\Pi^1$ to a proper and strongly irreducible permutation $\Pi^2$ with the property that all tail intervals lose at least once and no tail intervals win. Since the tail intervals lose against single intervals, and since all entries of $M^{\gamma_1}$ indexed by single intervals are strictly positive, it follows that the matrix $M^{\gamma_1\gamma_2}$ corresponding to the path $\gamma_1\gamma_2 $ has all rows indexed by single intervals strictly positive. We then make sure to reveal at least one interval of each tail interval: as in Proposition \ref{prop:alltailswin} we find another path $\gamma_3$ from $\Pi^2$ to a proper and strongly irreducible permutation $\Pi^3$ along which all tail intervals win exactly once, and consider $\delta_1 := \gamma_1 \gamma_2 \gamma_3$. Repeating the entire construction, the infinite path
$$\boldsymbol{\delta} = \delta_1 \cdot \delta_2 \cdot \delta_3 \cdot \dots$$
	is infinite-complete, since along this path any $\alpha \in \mathcal{A}$ is eventually revealed and hence infinitely often the entire row labeled by $\alpha$ of the Rauzy-Veech matrices corresponding to $\delta$ is eventually positive.
	\end{proof}

\subsection{Proof of Theorem B}\label{sec:thmBproof} We now turn to the proof of Theorem B. A key tool to prove density of uniquely ergodic IETs is Proposition \ref{prop:approximation}, which states that one can approximate open $\epsilon-$balls in $\Delta$ using finite Rauzy-paths. The proof of this approximation property uses the fact that it is true for finite-type IETs. We then use Proposition \ref{prop:approximation} to show that starting from any permutation there exists a uniquely ergodic path, which implies density of uniquely ergodic IETs, while the fact that uniquely ergodic IETs form a $G_{\delta}$ set will follow from a simple intersection argument.

\subsubsection{Approximation property for finite-type IETs} We obtain the approximation property for finite-type IETs as a corollary of the analogue of our main Theorem B for finite-type IETs, which was proven in 1980 in \cite{KeaneRauzy1980} by Keane and Rauzy and which we state below.
	We denote the finite-dimensional length simplex on $d$ intervals by
	$$\Delta^{d-1} := \{\lambda \in \R_{>0}^{d}\hspace{1mm}|\hspace{1mm}\sum_{i=1}^{d} \lambda_i = 1 \}.$$

	\begin{thm}\label{thm:finiteanalogue} (Theorem 7 in \cite{KeaneRauzy1980}) For any irreducible finite-type permutation $\pi$ on $d$ intervals, the set
	\begin{align*}
	    \{ \lambda \in \Delta^{d-1} \, | \, T_{(\pi, \lambda)} \text{ is uniquely ergodic} \}
	\end{align*}
	contains a $G_{\delta}$-dense subset of $\Delta^{d-1}$ with respect to any euclidean norm, in particular the $\ell^1$-norm, on $\Delta^{d-1}$.
	\end{thm}

    We then obtain the following corollary:

	\begin{cor}\label{cor:cylinderapproxfin} Let $\pi$ be an irreducible finite-type permutation on $d$ intervals, let $\lambda_0 \in \Delta^{d-1}$ and let $\epsilon > 0$. Then there exists a finite path $\gamma$ in $\mathcal{R}_{\pi}$ starting at $\pi$ such that
	\begin{align*}
	    \hat{M}^{\gamma}\Delta^{d-1} \subset B_{\ell^1}(\lambda_0, \epsilon)
	\end{align*}
	\end{cor}
	\begin{proof} Follows from the finite-type theory: by Theorem \ref{thm:finiteanalogue} we can choose a uniquely ergodic $\lambda \in B_{\ell^1}(\lambda_0, \epsilon)$ with path $\gamma$ starting at $\pi$, and since the action of the corresponding Rauzy-Veech-matrices $M^{\gamma}$ is contracting $\Delta^{d-1}$ towards the point $\lambda$, for $n$ large enough the claim follows.
	\end{proof}

    \subsubsection{Approximation property for infinite-type IETs}\label{sec:approximation}  We extend Corollary \ref{cor:cylinderapproxfin} to the case of infinite-type IETs, using the fact that the boundary of $\Delta$ is dense in $\Delta$ together with the combinatorial properties of infinite-type Rauzy graphs from Section \ref{sec:winningandlosingpaths}.

	\begin{prop}\label{prop:approximation} Let $\Pi$ be a proper and strongly irreducible permutation, $\lambda \in \Delta$ and $\epsilon > 0$.
		There exists a finite reduced path $\gamma$ in $\mathcal{R}_{\Pi}$ starting at $\Pi$ such that
	\begin{align*}
	    \Delta^{\gamma} \subset B_{\ell^1}(\lambda, \epsilon).
	\end{align*}
	\end{prop}

	\begin{proof}
		Let us denote for $d \in \N$, $\lambda_{(d)} := (\lambda_1,\ldots,\lambda_d,0,0,\ldots)$ and choose $d$ large enough such that $\lambda_{(d)}$ is $\epsilon$-close to \(\lambda\).
		More precisely, let \(d\) and \(\delta>0\) such that
		\[
			B_{\ell^1}(\hat{\lambda}_{(d)},2\delta)
		    \subset B_{\ell^1}(\lambda,\epsilon).
		\]

		Consider also the permutation $\Pi_d$ obtained from $\Pi$ by revealing the first $d$ letters.
		Note that $\Pi_d$ is proper and irreducible by \cref{lem:invarianceofirreducibility}-\ref{lem:irred}.

	Consider
	\[
		\lambda_{(d)}^* := (\lambda_1,\ldots,\lambda_d)
	\]
	and $\hat{\lambda}_{(d)}^*$ its normalization in the simplex $\Delta^{d-1}$.
	Let $(\pi_d)_*$ be the finite-type permutation on $d$ intervals obtained by removing all tail intervals from $\Pi_d$.
	By \cref{lem:invarianceofirreducibility} this permutation is irreducible.
	By Corollary \ref{cor:cylinderapproxfin}, there exists a finite path $\gamma^*$ in \(\mathcal R_{(\pi_d)_*}\), which we can extend to a finite complete path by Proposition \ref{prop:yoccoz}, such that $M^{\gamma^*}$ is a positive matrix and
	\[
	    \hat{M}^{\gamma^*} \Delta^{d-1}
	    \subset B_{\ell^1}(\hat{\lambda}_{(d)}^* ,\delta).
	\]
	Equivalently, if we denote the columns of the matrix $M^{\gamma^*} = \begin{pmatrix} m_1^*  & \cdots & m_d^* \end{pmatrix},$ then
	\[
		\left\| \hat{m}_i^*- \hat{\lambda}_{(d)}^* \right\|_1<\delta
	    \qquad\text{for every } i=1,\ldots,d.
	\]

	\medskip

	Now lift \(\gamma^*\) to a path \(\gamma\) in the infinite-type Rauzy diagram $\mathcal{R}_{\Pi}$ as in Lemma \ref{lem:liftfinite}, where along $\gamma$ all tail intervals lose when they play.
	Then the matrix of the lifted path has the following block form
	\[
	M^{\gamma} = \left( \begin{array}{c@{\qquad}c} \begin{array}{ccc} m_1^{\!*} & \cdots & m_d^{\!*} \end{array} & \begin{array}{cccc} c_1^{\!*} & c_2^{\!*} & c_3^{\!*} & \cdots \end{array} \\[0.6em] 0 & I \end{array} \right).
	\]
	Denote by $m_{i} = (m_i^*, 0, 0, \dots)$ and $c_{i} = (c_i^*, 0, \dots, 0, 1, 0, \dots)$ the corresponding columns of $M^\gamma$.
	Since each column $\hat{m}_i^*$ is $\delta$-close to $\hat{\lambda}_{(d)}^*$, $\hat{m}_i$ is \(\delta\)-close to  $\hat{\lambda}_{(d)}$, i.e.
	\[
		\left\|\hat{m}_{i} - \hat{\lambda}_{(d)} \right\|_1<\delta
	    \qquad\text{for } i=1,\ldots,d.
	\]
	Choose a loop \(\eta\) based at the ending vertex of $\gamma$ which makes all tail coordinates lose (which exists by Proposition \ref{prop:alltailslose}) and consider the concatenation of $\gamma$ with this loop repeated $N$ times $\gamma \cdot \eta^N.$ For a tail column $c_i^{N}$ of \(M^{\gamma \cdot \eta^N}\), we have an expression of the form
	\[
	    c_i^{N}
	    =
	    e_i+ M_{N}^{\eta},
	\]
	where $e_i$ is a vector of the canonical basis and $M_{N}^{\eta}$ is  a linear combination of columns $m_i$, i.e
	 \[
	    M_{N}^{\eta} = \sum_{j=1}^{d} x_j^N m_j,
	\]
	for some $x_j^N \in \N$ so that $\sum_{j=1}^d x_j^N \ge N$ since every tail is assumed to lose at least once along the loop.
	Since the vectors $m_j^*$ are positive and have integer entries,
	\[
	    \left\|M^\eta_N\right\|_1 \ge \sum_{j=1}^d x_j^N \ge N ,
	\]
	thus $\left\|c_i^N\right\|_1 \ge  N = N \cdot \|c_i^N - M^\eta_N\|_1$ and
	\begin{align*}
		\|\hat c_i^N - \hat M^\eta_N\|_1 &\le \frac{\|c_i^N - M^\eta_N\|_1}{\|c_i^N\|_1} + \|M^\eta_N\|_1 \cdot \left|\frac{1}{\|c_i^N\|_1} - \frac{1}{\|M^\eta_N\|_1}\right|\\
						 &\le \frac{1}{N} + \frac{\left|\|M^\eta_N\|_1 - \|c_i^N\|_1\right|}{\|c_i^N\|_1} \le \frac{2}{N}.
	\end{align*}
	Recall that the vectors \(m_j\) are columns whose normalizations are \(\delta\)-close to \(\hat{\lambda}_{(d)}\) and hence the same is true for the sum $M_{N}^{\eta}$.
	Therefore, for \(N\) sufficiently large, every normalized column of
	\(M^{\gamma \cdot \eta^N}\) is contained in
	\(B_{\ell^1}(\hat{\lambda}_{(d)},2 \delta)\).
	Hence, up to taking $\delta$ small enough,
	\[
	    \hat{M}^{\gamma \cdot \eta^N} \Delta
	    \subset B_{\ell^1}(\hat{\lambda}_{(d)},2 \delta)
	    \subset B_{\ell^1}(\lambda,\epsilon).
	\]
    Up to replacing $\gamma \cdot \eta^N$ with the corresponding reduced path, the proposition follows.
	\end{proof}

We now prove that starting from any permutation $\Pi$, there exists a uniquely ergodic path. For \(E\subset\Delta\), write
\[
\operatorname{diam}_{\ell^1}(E)
:=
\sup\{\|x-y\|_1:x,y\in E\}.
\]

\begin{prop}\label{prop:ueexists}
Let \(\Pi\) be a proper and strongly irreducible permutation, let $\tilde{\gamma}$ be a finite reduced path in $\mathcal{R}_{\Pi}$ starting at $\Pi$. Then there exists an infinite-complete path $\boldsymbol{\gamma}$ starting at \(\Pi\) with prefix $\tilde{\gamma}$ such that  $\lim_{n\to\infty}\operatorname{diam}_{\ell^1}(\Delta^{\gamma_n}) = 0$ and \(\Delta^{\boldsymbol{\gamma},\infty}\) is equal to a single point.
\end{prop}

\begin{proof}
We want to construct inductively finite paths $\eta_{0},\eta_1, \eta_2, \dots$
so that $\gamma^r:=\eta_{0}\cdots\eta_r$ has prefix $\tilde{\gamma}$ and
\[
    \operatorname{diam}_{\ell^1}(\Delta^{\gamma^r})\leq2^{-r+1}
\]
Define $\eta_{0} := \tilde{\gamma}$. Then $\operatorname{diam}_{\ell^1}(\Delta^{\gamma^{0}})\leq2$ is trivially true. Assume now that \(\gamma^{r-1}\) satisfying these properties has already been constructed. Since \(M^{\gamma^{r-1}}\) is a bounded operator and \(\|M^{\gamma^{r-1}}\lambda\|_1 \geq 1\) for every \(\lambda \in \Delta\), the projectivized map
\[
\hat M^{\gamma^{r-1}} \colon \Delta \to \Delta
\]
is Lipschitz for the \(\ell^1\)-metric. Let \(L_r\) be a Lipschitz constant
for this map. Fix $\lambda \in \Delta$. Choose \(\epsilon_r>0\) such that for all $1 \leq i \leq r$
\[
   \lambda_i > \epsilon_r, \hspace{2mm} 2L_r\epsilon_r < 2^{-r}.
\]

By Proposition \ref{prop:approximation} applied at the endpoint of the path
\(\gamma^{r-1}\), there exists a finite path \(\eta_r\) such that
\[
    \Delta^{\eta_r}
    \subset
    B_{\ell^1}(\lambda,\epsilon_r).
\]

Now define
\[
    \gamma^r:=\gamma^{r-1}\eta_r.
\]
Then
\[
\begin{aligned}
    \operatorname{diam}_{\ell^1}(\Delta^{\gamma^r}) = \operatorname{diam}_{\ell^1}(\hat{M}^{\gamma^{r-1}}\hat{M}^{\eta_r}\Delta)
    &\le
    L_r\operatorname{diam}_{\ell^1}(\hat M^{\eta_r}\Delta) \\
    &\le
    L_r\operatorname{diam}_{\ell^1}(B_{\ell^1}(\lambda,\epsilon_r)) \\
    &\le
    2L_r\epsilon_r
    <
    2^{-r}.
\end{aligned}
\]
Consider now the path $\boldsymbol{\gamma} = \eta_0 \eta_1\eta_2\dots$ and as usual let $\gamma_n$ be the prefix of $\gamma$ up to time $n$. Let $r \geq 1$, then for $n$ large enough we obtain
\[
    \operatorname{diam}_{\ell^1}(\Delta^{\gamma_n}) \leq \operatorname{diam}_{\ell^1}(\Delta^{\gamma^r}) \leq 2^{-r}
\]
and hence  $\lim_{n \to \infty} \operatorname{diam}_{\ell^1}(\Delta^{\gamma_n}) = 0$. Since $\Delta^{\boldsymbol{\gamma},\infty} \subset \Delta^{\gamma_n}$ for all $n \in \N$ we have $\operatorname{diam}_{\ell^1}(\Delta^{\boldsymbol{\gamma}, \infty}) = 0$.

We now show that $\Delta^{\boldsymbol{\gamma}, \infty} \neq \emptyset$ and hence $\Delta^{\boldsymbol{\gamma}, \infty}$ is equal to a point. By Proposition \ref{prop:nonemptyintersection} it suffices to check infinite-completeness of $\boldsymbol{\gamma}$. Indeed, every normalized column of
\(M^{\eta_r}\) lies in the closure of \(\hat M^{\eta_r}\Delta\). Since
\[
    \Delta^{\eta_r}
    \subset
    B_{\ell^1}(\lambda,\epsilon_r),
\]
for $1 \leq i \leq r$ every normalized column of \(M^{\eta_r}\) has $i$-th coordinate at least
\[
    \lambda_{i}-\epsilon_r>0,
\]
and in particular the entire $i$-th row of \(M^{\eta_r}\) is positive. Up to replacing $\boldsymbol{\gamma}$ by its corresponding reduced path, $\boldsymbol{\gamma}$ is infinite-complete.
\end{proof}

\begin{theoremB} For any proper and irreducible permutation $\Pi$, the set of lengths $\lambda \in \Delta$ for which the IET $T_{(\Pi, \lambda)}$ is uniquely ergodic contains a dense $G_\delta$ set with respect to the $\ell^1-$topology.
\end{theoremB}

\begin{proof}
By \cref{lem:stronglyirreduciblepresentation}, revealing finitely many tail intervals turns $\Pi$ into a proper and strongly irreducible permutation $\Pi'$ with $T_{(\Pi,\lambda)} = T_{(\Pi',\lambda)}$ for every $\lambda \in \Delta$. The set of lengths for which the IET is uniquely ergodic is therefore the same for $\Pi$ and $\Pi'$, so we may assume that $\Pi$ is strongly irreducible.

For $m \in \N$, define
\[
\mathcal O_m
:=
\bigcup_{\substack{\gamma\text{ finite reduced path from }\Pi,\\ \operatorname{diam}_{\ell^1}(\Delta^\gamma)<1/m}}
\Delta^\gamma .
\]
Each \(\mathcal O_m\) is open. Indeed, for a finite reduced path $\gamma$ the matrix $M^{\gamma}$ differs from the identity in only finitely many rows: by \cref{def:groupedmatrix} each edge contributes either $\mathrm{Id}$ or $\mathrm{Id} + E_{w,\mathcal L}$, so $M^{\gamma} = \mathrm{Id} + N$ with $N$ supported on finitely many rows, and its (integer) inverse is likewise $(M^{\gamma})^{-1} = \mathrm{Id} + N'$ with $N'$ finitely supported. Hence
\[
\Delta^{\gamma} = \hat{M}^{\gamma}\Delta = \bigl\{\mu \in \Delta : (M^{\gamma})^{-1}\mu \in \R_{>0}^{\N}\bigr\}
\]
is cut out of $\Delta$ by the finitely many strict inequalities $\bigl((M^{\gamma})^{-1}\mu\bigr)_i > 0$ for $i$ in the support of $N'$ --- the remaining coordinates only give $\mu_i > 0$, which holds throughout $\Delta$.
Each of these is a strict linear condition with countably many uniformly bounded coefficients, hence $\ell^1$-continuous, so $\Delta^{\gamma}$ is open in $\Delta$ and therefore so is the union \(\mathcal O_m\). The intersection
\[
\mathcal G:=\bigcap_{m\ge1}\mathcal O_m
\]
is then a countable intersection of open sets and hence a \(G_\delta\) subset of \(\Delta\). We further claim that $\mathcal{G}$ is dense. Consider any open set $U \subset \Delta$. By Proposition \ref{prop:approximation}, there exists a finite reduced path $\tilde{\gamma}$ starting at $\Pi$ such that
\[
\Delta^{\tilde{\gamma}} \subset U.
\]
By Proposition \ref{prop:ueexists} we can extend $\tilde{\gamma}$ to an infinite-complete path $\boldsymbol{\gamma}$ with prefix $\tilde{\gamma}$ where the diameters of $\Delta^{\gamma_n} \subset U$ contract towards a single point $\lambda \in U$. Then for all $m \in \N$ there exists $n \in \N$ such that
$\operatorname{diam}_{\ell^1}(\Delta^{\gamma_n})<1/m$ and $\lambda \in \Delta^{\gamma_n}$, in particular $\lambda \in \mathcal{G}$ and hence $\mathcal{G} \cap U$ is nonempty.

It remains to show that
\[
\mathcal G\subset \{\lambda\in\Delta : T_{\Pi,\lambda}\text{ is uniquely ergodic}\}.
\]
Let $\lambda \in \mathcal G$, then for all $m \in \N$ there exists $\gamma_m$ such that
\[
\lambda \in \Delta^{\gamma_m}, \; \operatorname{diam}_{\ell^1}
\bigl(\Delta^{\gamma_m}\bigr) < \frac{1}{m}.
\]
By Lemma \ref{lem:pathequivalence}, $\gamma_m$ is then a prefix of the path $\gamma_{\lambda}$ defined by $\lambda$ for all $m \in \N$. Since the diameters of $\Delta^{\gamma_m}$ tend to \(0\), the lengths of the paths \(\gamma_m\) must tend to infinity. Hence the Rauzy path \(\boldsymbol{\gamma_\lambda}\) is infinite. Furthermore, for all $m \in \N$,
\[
\operatorname{diam}_{\ell^1}
\bigl(\Delta^{\boldsymbol{\gamma_{\lambda}}, \infty}\bigr)
\le
\operatorname{diam}_{\ell^1}(\Delta^{\gamma_m})
<
\frac1m,
\]
hence $\{\lambda \} = \Delta^{\boldsymbol{\gamma_{\lambda}}, \infty}$ is equal to a point and by Proposition \ref{prop:invmeasures} the IET $T_{(\Pi, \lambda)}$ is uniquely ergodic.

\end{proof}

\section{A matrix condition for unique ergodicity} \label{sec:matrixcriterion}

This last section is dedicated to the proof of Theorem C, which yields a sufficient condition for a given tail-reversing infinite-type IET to be uniquely ergodic using \textit{contraction} of the Rauzy-Veech matrices along its path. The proof of Theorem B did not rely on such a result, since we could use the fact that the boundary of $\Delta$ is dense in $\Delta$ together with finite-type theory. However, Theorem C explicitely describes a large set of uniquely ergodic paths and therefore might enable us to prove measure-theoretic genericity results for tail-reversing IETs in the future.

\subsection{Unique ergodicity} We introduce a sufficient condition for unique ergodicity using the Hilbert projective metric as well as explicit bounds on the diameters of the Rauzy-Veech matrices.
\newcommand{\diam}{D}
\subsubsection{The Hilbert projective metric} We equip the simplex $\Delta$ with a canonical extended projective metric (also called the \textit{Hilbert projective metric}) defined for all $x,y \in \Delta$ by \[d(x, y) = \log \sup_{i,j \in \N} \frac{x_i y_j}{y_i x_j}.\]
For any finite subset $A \subset \N$ one can define a restricted simplex
$$\Delta_A := \{\lambda \in \R_{>0}^A\hspace{1mm}|\hspace{1mm}\sum_{i \in A}\lambda_i = 1 \}.$$
We similarly define a projective metric on the simplex $\Delta_A$ for all $x, y \in \Delta_A$ by
$$d_A(x, y) = \log \sup_{i,j \in A} \frac{x_i y_j}{y_i x_j}.$$
Notice that for all $A \subset B \subset \N$,
\[d_A(x, y) \le d_B (x, y) \le d(x, y).\]

For a matrix $M$, denoting by $M^k$ its $k$-th column, let us define the $A$-diameter of the matrix
\[\diam_A(M) = \sup_{k,l \in \N} \, d_A (M^k, M^l) = \log \sup_{\substack{i, j \in A\\ k, l \in \N}} \frac{M_{ik} M_{jl}}{M_{il} M_{jk}}.\]
Notice that this constant takes into account the coefficients of the matrix on the tail intervals when they are losing.

	\subsubsection{The matrix condition}\label{sec:uniqueergodicitycriterion}

	In the estimates which follow, since the projective metric is invariant under rescaling, we write $M$ rather than $\hat{M}$.

	\begin{lem}
		\label{lem:contraction}
		Let $A \subset B \subset \N$ then for all $x, y \in \Delta$,
		\[ d_A( M x_B, M y_B) \le \left(1- e^{-\diam_A(M)}\right) d_B (x,y) \]
		where $v_B$ denotes the projection to the vector in $\R_{\geq 0}^\N$ which coincide with $v$ on $B$ and is $0$ for coordinates in $\overline B$.
		When $\diam_A(M) = \infty$ one simply has the contraction $d_A( M x_B, M y_B) \le d_B (x,y).$
	\end{lem}
	\begin{proof}
		This is a classical property of the Hilbert Metric, see e.g.\ Section 1.2 in \cite{Viana97} for a proof.
	\end{proof}

	Since we are interested in bounding the distance $d_A(Mx,My)$, we need to estimate the error created by tail coordinates.
	In the following, for $A \subset \N$ and $x \in \Delta$, we denote by $\| x \|_A$ the $\ell^1$ norm of $x_A$ and by $\bar{A}$ the complementary set of $A$ in $\N$.
	Recall from Lemma \ref{lem:formrvmatrix} that we say a matrix is $A$-supported if on rows in $\bar{A}$ it is equal to the identity matrix;
	strictly $A$-supported if it is moreover strictly positive on rows in $A$.

    Note that due to Lemma \ref{lem:formrvmatrix} the matrix $M$ only has finitely many different columns when restricted to $A$, hence the constant $N_A(M)$ in the Proposition below is well-defined.

	\begin{lem}
		\label{lem:perturbation}
		Let $A \subset B \subset \N$ be finite sets and let $M$ be a Rauzy--Veech matrix which is strictly $A$-supported. For $x \in \Delta$, let $x_B$ and $x_{\overline{B}}$ be its projections onto $B$ and $\N \setminus B$, and set
		\[ N_A(M) := \max_{p,q \in \N} \frac{\|M^p\|_A}{\|M^q\|_A}, \qquad C_A(M) := \bigl(e^{\diam_A(M)} - 1\bigr)\, N_A(M). \]
		Then for all $x \in \Delta$,
		\[
			d_A(M x_B, M x) \le C_A(M) \, \frac {\|x\|_{\overline{B}}}{\| x \|_1}.
		\]
	\end{lem}

	\begin{proof}
		Let $u = Mx_B = \sum_{k \in B} x_k M^k$ and $v = Mx_{\overline{B}} = \sum_{k \in \overline{B}} x_k M^k$ (both are strictly positive on the rows indexed by $A$, since $M$ is strictly positive there).
		We want to estimate $d_A(u, u+v)$. For every $i \in A$ define $\beta_i = \frac{v_i}{u_i} > 0$ and let $\beta_{\max}$ ($\beta_{\min}$) be the largest (smallest) entry of $\beta$.
	Then $u_i +v_i = u_i(1+\beta_i)$ and
	\begin{equation}\label{eq:max}
	\exp d_A(u, u+v)
	= \sup_{i,j \in A} \frac{u_i(u_j+v_j)}{u_j(u_i+v_i)}
	= \sup_{i,j \in A} \frac{u_iu_j(1+\beta_j)}{u_ju_i(1+\beta_i)}
	= \frac{1+\beta_{\max}}{1+\beta_{\min}}.
	\end{equation}

	Let $K = e^{\diam_A(M)}$, then by definition
	$M_{ik} M_{jl}\le K M_{il} M_{jk}$ for all $i,j \in A$, $k,l \in \N$. Thus
	\[
	\sum_{k\in B,\ell\in\bar B}x_kx_\ell M_{ik}M_{j\ell}
	\le
	K\sum_{k\in B,\ell\in\bar B}x_kx_\ell M_{i\ell}M_{jk}.
	\]
	hence $u_i v_j\le K v_i u_j$ and dividing  by $u_i u_j>0$ yields $\beta_j \le K\beta_i$ for all $i,j \in A$ and in particular $\beta_{\max}\le K\beta_{\min}$.

	Now notice that
	\[
	  \frac{\beta_i}{1+\beta_i}
	  = \frac{(Mx_{\bar{B}})_i}{(Mx_B)_i + (Mx_{\bar{B}})_i}
	  = \frac{(Mx_{\bar{B}})_i}{(Mx)_i},
	\]
	so $\beta_{\min}/(1+\beta_{\min}) = \min_{i \in A}\, \beta_{i}/(1+\beta_{i}) = \min_{i \in A}\, (Mx_{\bar{B}})_i/(Mx)_i$.
	Since a minimum over a set is bounded above by any weighted average of elements of this set, we may choose the weights $\frac{(Mx)_i}{\|Mx\|_A}$ and obtain
	\[
	  \frac{\beta_{\min}}{1+\beta_{\min}}
	  \;\leq\; \sum_{i \in A} \frac{(Mx)_i}{\|Mx\|_A} \cdot  \frac{(Mx_{\bar{B}})_i}{(Mx)_i}
	  = \frac{\|Mx_{\bar{B}}\|_A}{\|Mx\|_A}.
	\]
	Let \(a_k:=\|M^k\|_A\); recall $N_A(M) = \max_{p,q \in \N} a_p/a_q$.
	By linearity and positivity,
	\[
	\|Mx_{\bar B}\|_A=\sum_{k\in\bar B}x_k a_k,
	\qquad
	\|Mx\|_A=\sum_{k\in\mathbb N}x_k a_k.
	\]
	Moreover, by definition of \(N_A(M)\), \(a_k\le N_A(M)\,a_l\) for all \(k,l \in \mathbb N\). Summing \(x_kx_l a_k\le N_A(M)\,x_kx_l a_l\) over
	\(k\in\bar B\), \(l\in\mathbb N\), gives
	\[
	\Big(\sum_{k\in\bar B}x_k a_k\Big)\Big(\sum_{l\in\mathbb N}x_l\Big)
	\le
	N_A(M)
	\Big(\sum_{k\in\bar B}x_k\Big)
	\Big(\sum_{l\in\mathbb N}x_l a_l\Big).
	\]
	Dividing by $\sum_{l\in\mathbb N} x_l \sum_{l\in\mathbb N}x_l a_l$ yields
	\[
	\frac{\|Mx_{\bar B}\|_A}{\|Mx\|_A}
	\le
	N_A(M)\frac{\|x_{\bar B}\|_1}{\|x\|_1}.
	\]
	Overall, it follows that
	\[
	  d_A(u,\, u+v)
	  = \log\frac{1+\beta_{\max}}{1+\beta_{\min}}
	  \leq \log\frac{1+K\beta_{\min}}{1+\beta_{\min}}
	  = \log\!\left(1 + \frac{(K-1)\,\beta_{\min}}{1+\beta_{\min}}\right)
	  \leq \frac{(K-1)\,\beta_{\min}}{1+\beta_{\min}}.
	\]
	where in the last step we used that $K \geq 1$. Using the previous inequality
	\[
	  d_A\!\left(Mx_B,\, Mx \right)
	  \;\leq\; (K-1) \, N_A(M) \,\frac{\|x_{\bar{B}}\|_1}{\|x\|_1}
	  \;=\; (e^{\diam_A(M)}-1)\, N_A(M) \, \frac{\|x_{\bar{B}}\|_1}{\|x\|_1}
	  \;=\; C_A(M)\,\frac{\|x_{\bar{B}}\|_1}{\|x\|_1}.
	\]
	This proves the first inequality. The second then follows from it together with Lemma \ref{lem:contraction} and the fact that
	\[ d_A( M x, M y) \leq d_A(M x, M x_B) + d_A(M x_B, M y_B) + d_A(M y_B, M y).
	\]

	\end{proof}

	We can now give our sufficient condition for a path $\boldsymbol{\gamma}$ to describe a uniquely ergodic IET. We first establish a \emph{localization} lemma, which bounds the $A$-diameter of a product $MN$ by the (larger) $B$-diameter of its second factor, provided that $N$ reveals enough intervals in each tail.

	\begin{lem}
		\label{lem:localization}
		Let $A \subseteq B \subset \N$ be finite sets, and let $M$ and $N$ be Rauzy--Veech matrices such that $M$ is $A$-supported and $N$ is strictly $B$-supported. Assume that $N = M^{\gamma}$ for a path $\gamma$ which reveals at least $n$ intervals from each tail of its starting grouped permutation; that is, along $\gamma$ at least $n$ intervals of each tail have been turned into single intervals, and their labels belong to $B$. Then
		\[ \diam_A( M N ) \le \diam_B(N) + \frac{2}{n}.\]
	\end{lem}

	\begin{proof}
		Recall that $\diam_A(MN) = \sup_{i,j} d_A\bigl((MN)^i, (MN)^j\bigr)$, where $(MN)^j = M N^j$ is the $j$-th column of $MN$.
		Fix a column index $j$. Since $N$ is $B$-supported, its rows outside $B$ coincide with those of the identity, so the column $N^j$ splits as
		\[ N^j = v_B + w, \qquad v_B := (N^j)_B, \quad w := (N^j)_{\overline B}, \]
		where $w = e_j$ if $j \in \overline B$ and $w = 0$ if $j \in B$.

		We show that $d_A(M N^j, M v_B) \le \tfrac1n$.
		If $j \in B$ then $M N^j = M v_B$ and there is nothing to prove.
		Assume $j \in \overline B$, in which case $M N^j = M v_B + M^j$.
		The index $j$ labels an interval of one of the tails of $M$, and by assumption $\gamma$ has revealed at least $n$ intervals of that same tail; their labels $k$ lie in $B$, satisfy $(v_B)_k = N^j_k \ge 1$, and the columns $M^k$ coincide with $M^j$ on coordinates in $A$ (see  \cref{prop:rauzy_mat}).
		Hence, on coordinates $i \in A$,
		\[ (M v_B)_i = \sum_{k \in B} (v_B)_k\, M^k_i \;\ge\; \sum_{k\ \mathrm{revealed}} M^k_i \;\ge\; n\, M^j_i, \]
		Notice that $(Mv_B)_i > 0$ since $M_i^i > 0$ and $(v_B)_i > 0$.
		Thus $0 \le M^j_i/(Mv_B)_i \le 1/n$ for every $i \in A$.
		Therefore, as in Equation (\ref{eq:max}) of the previous Lemma,
		\[ d_A(M N^j, M v_B) = \log \frac{\max_{i\in A}\bigl(1 + M^j_i/(Mv_B)_i\bigr)}{\min_{i\in A}\bigl(1 + M^j_i/(Mv_B)_i\bigr)} \le \log\!\Bigl(1+\tfrac1n\Bigr) \le \tfrac1n. \]
		Now let $i, j$ be arbitrary. Using the triangle inequality, the claim above (applied to the columns $i$ and $j$), and \cref{lem:contraction} (whose contraction factor is at most $1$),
		\begin{align*}
			d_A(M N^i, M N^j) &\le  d_A\bigl(M (N^i)_B,\, M (N^j)_B\bigr) + d_A\bigl(M N^i,\, M (N^i)_B \bigr) + d_A\bigl(M N^j,\, M (N^j)_B \bigr)\\ &\le
            d_A\bigl(M (N^i)_B,\, M (N^j)_B\bigr) + \frac{2}{n} \\
            & \le d_B(N^i, N^j) + \frac{2}{n}.
		\end{align*}
		Taking the supremum over $i,j$ gives $\diam_A(MN) \le \diam_B(N) + \tfrac2n$.
	\end{proof}

	\begin{theoremC}
		\label{prop:unique_ergodicity_criterion}
		Let $\boldsymbol{\gamma}$ be an infinite path in a Rauzy diagram $\mathcal{R}$. Assume there exist
		\begin{itemize}
			\item an increasing sequence of finite subsets $A_1 \subset A_2 \subset \dots \subset \N$ with $\bigcup_n A_n = \N$,
			\item a factorisation $\boldsymbol{\gamma} = \zeta_1 \cdot \zeta_2 \cdots$ such that, for every $n$, the matrix $M_n = M^{\zeta_n}$ is strictly $A_n$-supported and $\zeta_n$ reveals at least $n$ intervals in each tail,
			\item a constant $C$ with $\sup_n \diam_{A_n}(M_n) < C$.
		\end{itemize}
		Then there is a unique IET associated with $\boldsymbol{\gamma}$, and it is uniquely ergodic.
	\end{theoremC}

\begin{proof} Since $D_{A_n}(M_n) < \infty$ for all $n \in \N$, the matrix $M_n$ is strictly positive on rows in $A_n$, in particular, $\boldsymbol{\gamma}$ is infinite-complete and the intersection $\Delta^{\boldsymbol{\gamma}, \infty}$ is non-empty.
	It remains to show that it is equal to a point.

	Set $\delta := \tfrac12\bigl(C - \sup_n \diam_{A_n}(M_n)\bigr) > 0$, so that $\diam_{A_n}(M_n) \le C - 2\delta$ for all $n$, and write $d_m := d_{A_m}$ and $\rho := 1 - e^{-C} \in (0,1)$. Note again that each $M_n$ is strictly positive on its rows in $A_n$, as $\diam_{A_n}(M_n) < \infty$. Note also that $|A_n| \ge n$: the $\ge n$ intervals revealed by $\zeta_n$ in any one tail are single intervals whose labels lie in $A_n$. We record two facts.

	\emph{(a) Column leakage.} Let $M$ be strictly $A_m$-supported, equal to $M^{\gamma}$ for a path revealing at least $m$ intervals in each tail. Then
	\begin{equation} \label{eq:colum_bound}
		\frac{\|M^i\|_{\overline{A_m}}}{\|M^i\|_1} \le \frac1m \quad (i \in \N), \qquad \text{hence} \qquad \frac{\|Mx\|_{\overline{A_m}}}{\|Mx\|_1} \le \frac1m \quad (x \in \Delta).
	\end{equation}
	Indeed, for $i \in A_m$ the numerator vanishes (as $M$ is $A_m$-supported), while for $i \in \overline{A_m}$ one has $\|M^i\|_{\overline{A_m}} = 1$ and, since $M$ is positive on its rows in $A_m$ and $|A_m| \ge m$,
	\[ \|M^i\|_{A_m} = \sum_{j \in A_m} M_{ji} \ge |A_m| \ge m. \]
	Thus $\|M^i\|_1 = \|M^i\|_{A_m} + 1 \ge m+1$, giving the first inequality; the second follows by averaging over columns.

	\emph{(b) Block diameters.} For $a < b$, applying \cref{lem:localization} with $A = B = A_b$ --- the prefix $M_a \cdots M_{b-1}$ is $A_{b-1}$-supported, hence $A_b$-supported, and $M_b$ reveals $\ge b$ intervals --- gives
	\begin{equation} \label{eq:block_diam}
		\diam_{A_b}(M_a \cdots M_b) \le \diam_{A_b}(M_b) + \tfrac2b \le C - 2\delta + \tfrac2b,
	\end{equation}
	which is $< C$ as soon as $b > 1/\delta$. Moreover, for any Rauzy--Veech matrix $Q$ the columns of $(M_a \cdots M_b)Q$ are nonnegative combinations of the columns of $M_a \cdots M_b$, and the projective $A_b$-diameter of the image of the positive cone is the largest pairwise distance between those columns; hence
	\begin{equation} \label{eq:diam_mono}
		\diam_{A_b}\bigl((M_a \cdots M_b)\, Q\bigr) \le \diam_{A_b}(M_a \cdots M_b).
	\end{equation}

	\emph{One contraction step.} We claim that, for any sequence satisfying the hypotheses of the proposition and any $\epsilon > 0$, there is $n_\ast$ such that for all $n \ge n_\ast$, $N \ge n$ and $x, y \in \Delta$,
	\begin{equation} \label{eq:onestep}
		d_1(M_1 \cdots M_N x,\, M_1 \cdots M_N y) \le \rho\, d_n(M_2 \cdots M_N x,\, M_2 \cdots M_N y) + \epsilon.
	\end{equation}
	Write $u = M_2 \cdots M_N x$ and $v = M_2 \cdots M_N y$. By the triangle inequality,
	\[ d_1(M_1 u, M_1 v) \le d_1(M_1 u_{A_n}, M_1 v_{A_n}) + d_1(M_1 u_{A_n}, M_1 u) + d_1(M_1 v_{A_n}, M_1 v). \]
	By \cref{lem:contraction}, the first term is at most $\bigl(1 - e^{-\diam_{A_1}(M_1)}\bigr) d_n(u, v) \le \rho\, d_n(u, v)$.
	By \cref{lem:perturbation}, applied to $M_1$ at the levels $A_1 \subset A_n$, together with \eqref{eq:colum_bound} for the strictly $A_n$-supported factor $M_2 \cdots M_n$ of $u = (M_2 \cdots M_n)(M_{n+1} \cdots M_N x)$,
	\[ d_1(M_1 u_{A_n}, M_1 u) \le C_{A_1}(M_1)\, \frac{\|u\|_{\overline{A_n}}}{\|u\|_1} \le \frac{C_{A_1}(M_1)}{n}, \]
	and the same bound holds for $v$. As $C_{A_1}(M_1)$ depends only on the fixed matrix $M_1$, both error terms are $\le \epsilon/2$ once $n \ge n_\ast := \lceil 2\, C_{A_1}(M_1)/\epsilon\rceil$, which proves \eqref{eq:onestep}. The argument used only the hypotheses of Theorem C, so \eqref{eq:onestep} holds for \emph{every} sequence satisfying them.

	\emph{Iteration.} Fix $x, y \in \Delta$, $\epsilon > 0$ and $k \in \N$. We choose $1 = n_0 < n_1 < \dots < n_k$ inductively so that, with $n_{-1} := 0$,
	\begin{equation} \label{eq:induction}
		\begin{split}
		d_1(M_1 \cdots M_N x,\, M_1 \cdots M_N y) \le \ &\rho^l\, d_{n_l}\bigl(M_{n_{l-1}+1} \cdots M_N x,\ M_{n_{l-1}+1} \cdots M_N y\bigr) \\
							      &+ (1 + \rho + \dots + \rho^{l-1})\,\epsilon
		\end{split}
	\end{equation}
	holds for all $N \ge n_l$ and all $1 \le l \le k$.
	For $l = 1$ this is \eqref{eq:onestep} applied to the original sequence, with $n_1 \ge n_\ast(\epsilon)$. Assume \eqref{eq:induction} holds for $l - 1$. The shifted sequence headed by the block $M_{n_{l-2}+1}\cdots M_{n_{l-1}}$, namely
	\[ \bigl(M_{n_{l-2}+1}\cdots M_{n_{l-1}},\ M_{n_{l-1}+1},\ M_{n_{l-1}+2},\ \dots\bigr) \quad \text{at levels} \quad \bigl(A_{n_{l-1}},\ A_{n_{l-1}+1},\ A_{n_{l-1}+2},\ \dots\bigr), \]
	again satisfies the hypotheses of the proposition with the same constant $C$: its first matrix is strictly $A_{n_{l-1}}$-supported with $A_{n_{l-1}}$-diameter $< C$ by \eqref{eq:block_diam} for $n_{l-1}$ large enough, reveals at least one interval per tail, and the remaining matrices $M_{n_{l-1}+j-1}$ reveal $\ge n_{l-1}+j-1 \ge j$ intervals, while all diameters are $< C$. Applying \eqref{eq:onestep} to it, we obtain an index $n_l > n_{l-1}$ such that for all $N \ge n_l$
	\[ d_{n_{l-1}}\bigl(M_{n_{l-2}+1}\cdots M_N x,\ M_{n_{l-2}+1}\cdots M_N y\bigr) \le
	\rho\, d_{n_l}\bigl(M_{n_{l-1}+1}\cdots M_N x,\ M_{n_{l-1}+1}\cdots M_N y\bigr) + \epsilon . \]
	Substituting into \eqref{eq:induction} for $l-1$ yields \eqref{eq:induction} for $l$.

	Taking $l = k$, by the definition of the diameter together with bounds from \eqref{eq:diam_mono}, for $n_k$ large enough, and \eqref{eq:block_diam}, for all $N \ge n_k$ we have
	\begin{align*}
		d_{n_k}\bigl(M_{n_{k-1}+1} \cdots M_N x,\ M_{n_{k-1}+1} \cdots M_N y\bigr) &\le \diam_{A_{n_k}}\bigl(M_{n_{k-1}+1} \cdots M_N\bigr) \\
											   &\le \diam_{A_{n_k}}\bigl(M_{n_{k-1}+1} \cdots M_{n_k}\bigr) < C,
	\end{align*}
	 so that $d_1(M_1 \cdots M_N x,\, M_1 \cdots M_N y) \le \rho^k\, C + \dfrac{\epsilon}{1-\rho}$ for all $N \ge n_k$. This bound is uniform in $x, y$, hence $\diam_{A_1}(M_1 \cdots M_N) \le \rho^k C + \tfrac{\epsilon}{1-\rho}$ for $N \ge n_k$; letting $N \to \infty$, then $k \to \infty$, then $\epsilon \to 0$ gives
	\[ \lim_{N \to \infty} \diam_{A_1}(M_1 \cdots M_N) = 0. \]
	As this derivation used only the hypotheses of the proposition, it applies to any sequence satisfying them: the projection of the corresponding nested intersection onto the \emph{first} level is a single point. In particular, for the sequence at hand, the projection of $\bigcap_N M^{\zeta_1 \cdots \zeta_N}\Delta$ onto $A_1$ is a single point.

	\emph{All levels.} Fix $s \in \N$ such that $s > 1/\delta$ and apply this to the sequence headed by the block $M_1 \cdots M_s$,
	\[ \bigl(M_1 \cdots M_s,\ M_{s+1},\ M_{s+2},\ \dots\bigr) \quad \text{at levels} \quad \bigl(A_s,\ A_{s+1},\ A_{s+2},\ \dots\bigr), \]
	which satisfies the hypotheses with constant $C$ (its first matrix is strictly $A_s$-supported with $\diam_{A_s}(M_1 \cdots M_s) < C$ by \eqref{eq:block_diam}, and reveals at least one interval per tail) and whose associated path is again $\boldsymbol{\gamma} = \zeta_1 \cdots \zeta_s \cdot \zeta_{s+1} \cdots$. Its products are $M_1 \cdots M_N$, so its first level being $A_s$, the conclusion gives that the projection of $\bigcap_N M^{\zeta_1 \cdots \zeta_N}\Delta$ onto $A_s$ is a single point.

	Since $\bigcup_s A_s = \N$, the set $\Delta^{\boldsymbol{\gamma}, \infty}$ is reduced to a single point, and by \cref{prop:invmeasures} the set $\mathcal{M}_T$ of invariant probability measures of the corresponding IET $T$ is a singleton. Thus $\boldsymbol{\gamma}$ determines a unique IET, which is uniquely ergodic.
\end{proof}

\subsubsection{Uniquely ergodic paths}\label{sec:uepath} We show in addition that there exists a uniquely ergodic path in any Rauzy diagram, which together with Proposition \ref{prop:approximation} yields an alternative proof of Theorem B.

\begin{prop}\label{prop:uniquely_ergodic_path}
	In any Rauzy diagram $\mathcal{R}$, starting from any proper strongly irreducible permutation $\Pi \in \mathcal{R}$, there exists a path satisfying the assumptions of Theorem C.
\end{prop}

\begin{proof}
Fix $C > 0$. We build a factorisation $\boldsymbol{\gamma} = \zeta_1 \cdot \zeta_2 \cdots$ satisfying the hypotheses of Theorem C, where $\zeta_n$ denotes the $n$-th block of $\boldsymbol{\gamma}$.

We construct the blocks $\zeta_n$ inductively, where at each stage the ending permutation of $\zeta_n$ is proper and strongly irreducible by \cref{rem:preservingirreducibility}. Define $\Pi_0 := \Pi$ and suppose $\zeta_1, \dots, \zeta_{n-1}$ have been built and end at a permutation $\Pi_{n-1}$ for some $n \geq 1$. At each step, the block $\zeta_n$ is defined as the concatenation of the following two paths:
\begin{itemize}
	\item \emph{Revealing.} Starting from $\Pi_{n-1}$, \cref{prop:alltailswin} applied $n$ times gives a path $\delta$ along which each tail interval wins exactly $n$ times, revealing at least $n$ intervals in each tail. Let $\Pi_{n-1}'$ be its ending infinite-type permutation.
	\item \emph{Finite-type contraction.} The path $\delta$ is the lift of a path $\delta_*$ in $\mathcal{R}_{(\pi_{n-1}')_*}$ since the tail intervals of $\Pi_{n-1}'$ do not win.
	By \cref{prop:yoccoz} the finite-type Rauzy diagram $\mathcal{R}_{(\pi_{n-1}')_*}$ is strongly connected and carries a complete loop $\ell_*$ starting with $\delta_*$. We consider its lift $\ell$ to the infinite-type Rauzy diagram $\mathcal{R}_{\Pi_{n-1}'}$, which is possible since along this loop no tail intervals in $\Pi_{n-1}'$ win.
    Let $A_n$ be the set of labels of the single intervals of $\Pi'_{n-1}$. The matrix $M^{\ell}$ restricted to coordinates in $A_n$ is equivalent to the strictly positive finite-type Rauzy--Veech matrix $M^{\ell_*}$ and the same remains true when repeating the loop $\ell$. Hence there exists $
    k$ large enough so that for the loop $\ell^k$ obtained by repeating $\ell$ a number $k$ of times, the finite-type matrix $M^{\ell_*^k}$ obtained by restricting $M^{\ell^k}$ to $A_n$ has diameter less than $C$.
    \item \emph{Infinite-type contraction.} Note however that we want the diameter restricted to entire \textit{rows} indexed by $A_n$ to be less than $C$. By \cref{prop:alltailslose} there exists a path $\sigma$ starting from the endpoint of $\ell^k$ along which all tails loose and no tail wins.
	    When a tail loses the winning column is added to the losing column, thus on rows indexed by $A_n$ the losing column coincides with one of the columns of the corresponding matrix restricted to $A_n$.
	    In particular for $\zeta_n := \ell^k \cdot \sigma$ it holds that $D_{A_n}(M^{\zeta_n}) \le C$.
\end{itemize}
By construction $M^{\zeta_n}$ acts trivially outside of rows in $A_n$, hence is strictly $A_n$-supported, and $\zeta_n$ reveals at least $n$ intervals in each tail. Furthermore, the sets $A_n$ are increasing and $\bigcup_n A_n = \mathcal{A}$, since every interval is eventually revealed. Hence, writing $M_n := M^{\zeta_n}$, the path $\boldsymbol{\gamma} = \zeta_1 \cdot \zeta_2 \cdots$ satisfies the three hypotheses of Theorem C, where the last hypotheses is satisfied up to taking a larger constant.
\end{proof}

\thispagestyle{empty}
\printbibliography
\end{document}